\documentclass{svjour3}                     
\journalname{}
\usepackage[
    backend=bibtex, 
    style=alphabetic, 
    ]{biblatex}
\usepackage{amsmath,amsxtra,amsfonts,amscd,amssymb,bm,bbm,mathrsfs}
\usepackage{nicefrac}       
\usepackage{extpfeil}
\usepackage{mathtools}
\usepackage{booktabs}
\usepackage{arydshln}
\usepackage[table]{xcolor} 
\usepackage{colortbl}
\usepackage{multirow} 
\usepackage{float}
\usepackage{hyperref}
\hypersetup{colorlinks=true,linkcolor=blue,citecolor=blue}
\usepackage{makecell}
\usepackage{lipsum}
\usepackage{graphicx}
\usepackage{epstopdf}
\usepackage{algorithmic}

\ifpdf
  \DeclareGraphicsExtensions{.eps,.pdf,.png,.jpg}
\else
  \DeclareGraphicsExtensions{.eps}
\fi
\newtheorem{assumption}{Assumption}

\makeatletter
\@addtoreset{equation}{section}
\makeatother

\usepackage{tikz}
\usepackage{pgfplots}
\usetikzlibrary{positioning, arrows.meta,calc,decorations.text,bending}
\usetikzlibrary{arrows.meta, cd}
\usetikzlibrary{matrix,shadows,positioning,decorations.pathreplacing,decorations.markings}
\usetikzlibrary{matrix,shapes.arrows, shapes.geometric}
\usetikzlibrary{patterns}
\usepackage{tikz-3dplot-circleofsphere}
\usetikzlibrary{shapes}
\usepgfplotslibrary{fillbetween}

\usepackage{standalone}
\usepackage{bm}
\usepackage{enumitem}
\usepackage[normalem]{ulem}

\smartqed  

\newcommand{\innerp}[1]{\left\langle {#1} \right\rangle }
\newcommand{\norm}[1]{\left\| {#1} \right\| }
\newcommand{\abs}[1]{\left | {#1} \right |}
\newcommand{\zkh}[1]{\left[ {#1} \right]}
\newcommand{\hkh}[1]{\left\{ {#1} \right\}}
\newcommand{\kh}[1]{\left( {#1} \right)}
\newcommand{\expnumber}[2]{{#1}\mathrm{e}{#2}}

\DeclareMathOperator*{\argmin}{arg\,min}

\DeclareMathOperator{\diag}{diag}

\DeclareMathOperator{\rank}{rank}
\DeclareMathOperator{\conv}{conv}

\DeclareMathOperator{\vecspan}{span}

\DeclareMathOperator{\ima}{im}
\DeclareMathOperator{\dime}{dim}

\DeclareMathOperator{\kernel}{ker}

\newcommand{\boldt}[1]{\textbf{#1}}

\newcommand{\mdps}{\mathcal{S}}
\newcommand{\mdpa}{\mathcal{A}}
\newcommand{\probmatrix}{\mathrm{P}}

\newcommand{\projection}{\mathcal{P}}
\newcommand{\manifold}{\mathcal{M}}
\newcommand{\desing}{\mathcal{M}_{\mathrm{desing}}}

\newcommand{\banifold}{\mathcal{B}}

\newcommand{\aanifold}{\mathcal{A}}
\newcommand{\lanifold}{\mathcal{L}}
\newcommand{\wanifold}{\mathcal{W}}

\newcommand{\canifold}{\mathcal{C}}

\newcommand{\eanifold}{\mathcal{E}}

\newcommand{\transport}{\mathcal{T}}
\newcommand{\kanifold}{\mathcal{K}}
\newcommand{\xanifold}{\mathcal{X}}

\newcommand{\nanifold}{\mathcal{N}}
\newcommand{\sanifold}{\mathcal{S}}

\newcommand{\hanifold}{\mathcal{H}}

\newcommand{\boundedrank}{\mathbb{R}_{\leq r}^{m \times n}}
\newcommand{\lowrank}{\mathbb{R}_{s}^{m \times n}}

\newcommand{\stiefel}{\mathrm{St}}
\newcommand{\so}{\mathrm{SO}}
\newcommand{\grassmann}{\mathrm{Gr}}

\newcommand{\trace}{\mathrm{tr}}
\newcommand{\orth}{\mathcal{O}}

\newcommand{\oblique}{\mathrm{Ob}}
\newcommand{\diff}{\mathrm{D}}
\newcommand{\tangent}{\mathrm{T}}
\newcommand{\normal}{\mathrm{N}}
\newcommand{\retrac}{\mathrm{R}}
\newcommand{\tanL}{\mathrm{\bf L}}
\newcommand{\tanQ}{\mathrm{\bf Q}}

\newcommand{\myfig}{{Fig.\,}}
\newcommand{\complexity}{O}
\newcommand{\frob}{\mathrm{F}}
\newcommand{\frechet}{{}}
\newcommand{\bouli}{{}}

\newcommand{\Fnorm}[1]{\|#1\|_\mathrm{F}}

\newcommand{\sksym}{\mathrm{Skew}}
\newcommand{\symac}{\mathrm{sym}}
\newcommand{\sym}{\mathrm{Sym}}

\newcommand{\mbR}{\mathbb{R}}
\newcommand{\mbRmn}{\mathbb{R}^{m\times n}}

\definecolor{blue}{rgb}{0, 0.451, 1}
\definecolor{myblue}{rgb}{0,0,.5}
\definecolor{bblue}{rgb}{0,0,.85}
\definecolor{mygreen}{rgb}{0, 0.62, 0.38}
\definecolor{red}{rgb}{0.796, 0.059, 0.063}
\definecolor{myred}{rgb}{.5,0,0}
\definecolor{LightGray}{rgb}{0.98, 0.98, 0.98}
\definecolor{Gray}{gray}{0.85}
\colorlet{revisecolor}{red!0!black}

\newcommand{\revise}[1]{\textcolor{revisecolor}{#1}}

\definecolor{comblue}{RGB}{2,0,255}
\definecolor{comgreen}{HTML}{d0e1d1}
\definecolor{borderColor}{HTML}{135412}

\begin{document}

\title{A space-decoupling framework for optimization on bounded-rank matrices with orthogonally invariant constraints\thanks{This work was supported by the National Key R\&D Program of China (grant 2023YFA1009300). BG was supported by the Young Elite Scientist Sponsorship Program by CAST. YY was supported by the National Natural Science Foundation of China (grant No. 12288201).}}

\titlerunning{Low-rank optimization with orthogonally invariant constraints}

\author{Yan Yang  \and Bin Gao \and Ya-xiang Yuan}


\institute{
Yan Yang \at
State Key Laboratory of Mathematical Sciences, Academy of Mathematics and Systems Science, Chinese Academy of Sciences, and the University of Chinese Academy of Sciences, Beijing, China
\\
\email{yangyan@amss.ac.cn} 
\and
Bin Gao \and Ya-xiang Yuan \at
State Key Laboratory of Mathematical Sciences, Academy of Mathematics and Systems Science, Chinese Academy of Sciences, Beijing, China \\
\email{\{gaobin,yyx\}@lsec.cc.ac.cn; }
}

\date{Received: date / Accepted: date}

\maketitle
\begin{abstract}
Imposing additional constraints on low-rank optimization has garnered growing interest. However, the geometry of coupled constraints hampers the well-developed low-rank structure and makes the problem intricate. To this end, we propose a space-decoupling framework for optimization on bounded-rank matrices with orthogonally invariant constraints. The ``space-decoupling" is reflected in several ways. We show that the tangent cone of coupled constraints is the intersection of tangent cones of each constraint. Moreover, we decouple the intertwined bounded-rank and orthogonally invariant constraints into two spaces, leading to optimization on a smooth manifold. Implementing Riemannian algorithms on this manifold is painless as long as the geometry of additional constraints is known. In addition, we unveil the equivalence between the reformulated problem and the original problem. Numerical experiments on fruitful applications---spherical data fitting, graph similarity measuring, low-rank SDP, model reduction of Markov processes, reinforcement learning, and deep learning---validate the superiority of the proposed framework.

\end{abstract}

\section{Introduction}\label{sec:intro}
Low-rank optimization, aiming to exploit the low-dimensional structure in matrix data for memory and computational efficiency, achieves success in a multitude of applications, e.g., factor analysis \cite{chu2005lowrankoblique}, system identification \cite{markovsky2008systemide,zhu2022learningmarkov}, large language models \cite{hu2022lora}, synchronization \cite{boumal2024orthsynch}. With extra equality constraints in addition to the low-rank requirement, this paper is concerned with the following matrix optimization problem,
\begin{equation}\label{eq:lowrank_orthinvar}\tag{P}
    \begin{aligned}
    \min_{X\in\mbR^{m\times n}}\ \ & f(X)\\[-1.5mm]
    \mathrm{s.\,t.}\ \ \ \ \,& \rank(X)\le r,
        \\
        & h(X)=0, 
    \end{aligned}
\end{equation}
where the rank parameter $r\le\min\{m,n\}$, $f:\mbR^{m\times n}\to\mbR$ is twice continuously differentiable, and $h:\,\mathbb{R}^{m\times n}\rightarrow \mathbb{R}^{q}$ is \emph{orthogonally invariant} in the sense that $ h(XQ) = h(X)$ for all $Q$ in the orthogonal group $\orth(n):=\{\tilde{Q}\in\mbR^{n\times n}:\tilde{Q}^\top \tilde{Q}=I_n\}$. We denote the set of bounded-rank matrices by
\[
\boundedrank := \hkh{X \in \mathbb{R}^{m \times n} : \rank(X) \leq r}
\]
and the level set of ~$h$ by
\[
\hanifold := \hkh{X \in \mathbb{R}^{m \times n} : h(X) = 0}.
\]
Consequently, the coupled feasible region can be written as $\boundedrank \cap \hanifold$, which is both nonsmooth and nonconvex. In the vanilla scenario $h(X)\equiv 0$, which implies $\hanifold=\mbR^{m\times n}$, problem \eqref{eq:lowrank_orthinvar} reduces to minimizing a function over bounded-rank matrices~\cite{schneider2015Lojaconvergence,levin2023remedy}. Throughout this paper, the following blanket assumption is imposed on $h$.

\begin{assumption}[Blanket]\label{assu:h} 
    The mapping $h:\,\mathbb{R}^{m\times n}\rightarrow \mathbb{R}^{q}$ is smooth and orthogonally invariant, and has full rank $q$ in the level set $\hanifold$, i.e., $\rank(\diff h_X)=q$ for {all} $X\in\hanifold$, where $\diff h$ denotes the Jacobian matrix.
\end{assumption}

\subsection{Applications}\label{sec:applications}
The general formulation~\eqref{eq:lowrank_orthinvar} satisfying Assumption~\ref{assu:h} encompasses an array of structured optimization problems with low-rank matrix variables. We now present a brief {overview} on representative applications.

\paragraph{\emph{\textbf{Spherical data fitting.}}} Finding a low-rank approximation of normalized data points lays a foundation for various problems, including concept mining, pattern classification, and information retrieval. Specifically, given a matrix $A\in\mbR^{m\times n}$ {of which} rows represent data points and {have} unit length, Chu et al. \cite{chu2005lowrankoblique} {proposed} the approximation task as follows,
\revise{
\begin{equation}\label{eq:sphericalfitting}
    \begin{aligned}
    \min_{X\in\mbR^{m\times n}}\ \ &\frac{1}{2} \norm{X-A}_{{\frob}}^2
    \\
    \mathrm{s.\,t.}\ \ \ \ \,&\rank(X)\le r,
        \\
        &\diag(XX^\top)-{\bf 1}=0.
    \end{aligned}
\end{equation}}
The {operator} $\diag(\cdot)$ extracts the diagonal elements from a square matrix, and ${\bf 1}$ denotes the all-ones vector. Note that problem~\eqref{eq:sphericalfitting} is an instance of~\eqref{eq:lowrank_orthinvar} with $h(X)=\diag(XX^\top)-{\bf 1}$, where $\hanifold$ defines the \emph{oblique manifold}, $\hanifold = \oblique(m,n):=\{X\in\mbR^{m\times n}:\,\diag(XX^\top)={\bf 1}\}$. {Specifically,} the interaction between $\boundedrank$ and $\hanifold$ {can} lead to a complicated geometry; {see an} illustration in \myfig\ref{fig:boundedcapoblique}.

\begin{figure}[htbp]
    \centering
    \includegraphics[scale=.95]{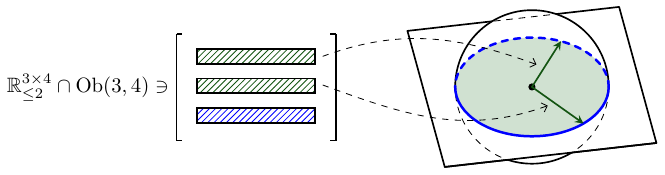}
    \caption{Illustration of $\boundedrank\cap\oblique(m,n)$ with $(m,n,r)=(3,4,2)$. Given a matrix $X\in\oblique(3,4)$, if the first two rows in the unit sphere span a plane, then the rank constraint $\rank(X)\le 2$ confines the third row in the intersection of this plane with the sphere.}
    \label{fig:boundedcapoblique}
\end{figure}

\paragraph{\emph{\textbf{Graph similarity measuring.}}} In chemical structure comparison, biological data analysis, and web searching, measuring the similarity of two graphs {is a fundamental task. Given $X\in\mbR^{m\times n}$ of which an} entry evaluates the connection between {two nodes in different graphs},
Blondel et al. \cite{blondel2004measuresimilarity} proposed an iterative process 
to measure the desired node-to-node similarity. Furthermore,
to explore the low-rank structure of the similarity matrix and to reduce the computational cost, Cason et al. \cite{cason2013iterative} proposed the measuring problem as follows,
\begin{equation}\label{eq:similarmeasure}
    \begin{aligned}
    \min_{X\in\mbR^{m\times n}}\ \ &-\trace(X^\top \lanifold\circ\lanifold(X))
    \\[-1.5mm]
    \mathrm{s.\,t.}\ \ \ \ \,&\rank(X)\le r,
        \\
        &\|X\|^2_{\mathrm{F}}-1=0,
    \end{aligned}
\end{equation}
where $\lanifold:\mbR^{m\times n}\to\mbR^{m\times n}$ is a linear operator and $\|\cdot\|_{\frob}$ denotes the Frobenius norm. Problem \eqref{eq:similarmeasure} falls into the scope of~\eqref{eq:lowrank_orthinvar} with $h(X)=\|X\|^2_{\mathrm{F}}-1$ and the associated \revise{$\hanifold = \mathrm{S}_{\mathrm{F}}(m,n):=\{X\in\mbR^{m\times n}:\,\|X\|^2_{\mathrm{F}}=1\}$.}

\paragraph{\emph{\textbf{Low-rank semidefinite programming (SDP).}}} Applications across various fields can be addressed by the semidefinite program,
\begin{equation}\label{eq:SDP}
    \min\ \innerp{C,Y},\quad\ \mathrm{s.\,t.}\ Y\in\sym(n),\,Y\succeq 0,\,\lanifold(Y)=b,
\end{equation}
where $\sym(n)$ is the set of $n\times n$ symmetric matrices, $\lanifold:\mbR^{n\times n}\to\mbR^q$ is a linear operator, and $b\in\mbR^{q}$. As shown in the paper of Pataki~\cite{pataki1998lowranksolution}, a low-rank solution for~\eqref{eq:SDP} exists when the feasible region is compact. Therefore, we can employ the Burer--Monteiro factorization $Y=XX^\top$ \cite{burer2003BM} and impose a rank constraint on~\eqref{eq:SDP}, resulting in an instance of~\eqref{eq:lowrank_orthinvar},
\begin{equation}\label{eq:lowrankSDP}
    \begin{aligned}
    \min_{X\in\mbR^{n\times n}}\ \ &\langle C,XX^\top \rangle
    \\[-1.5mm]
    \mathrm{s.\,t.}\ \ \ \ \,&\rank(X)\le r,
    \\
    &\lanifold(XX^\top)-b=0.
    \end{aligned}
\end{equation}
Under conditions given in the works~\cite{journee2010conesdp,boumal2016sdpBMmanifold}, the set $\hanifold=\{X\in\mbR^{n\times n}:\lanifold(XX^\top)=b\}$ is a smooth manifold. Several specific examples were considered including the oblique manifold $\oblique(n,n)$~\cite{goemans1995maxcutsdp}, the unit Frobenius sphere $\mathrm{S}_{\mathrm{F}}(n,n)$ \cite{journee2010conesdp}, and the stacked Stiefel manifold $\stiefel(n,p)^k:=\{[X_1~X_2\cdots X_k]^{\top}\in\mbR^{kp\times n}_{}: X_j\in\stiefel(n,p)\text{ for }j=1,\ldots,k\}$~\cite{boumal2015blockdiagonal}, where the \emph{Stiefel manifold} is defined by $\stiefel(n, p):=\{X \in \mbR^{n \times p}: X^{\top} X=I_p\}$, and $I_p$ is the identity matrix of size $p\times p$.

\paragraph{\emph{\textbf{Model reduction of Markov processes.}}} Numerous systems rely on Markov processes as the basic model, where each event's probability depends only on the previous state. However, a large state space $\mdps=\{s_1,s_2,\ldots,s_n\}$, usually appearing in complicated systems, confines the identification of the Markov model $\probmatrix\in\mbR^{\abs{\mdps}\times \abs{\mdps}}$ in which an entry $\probmatrix(s,s^\prime)$ represents the probability of the transition from state $s$ to state $s^\prime$. Therefore, we can resort to the low-rank property of the Markov process to reduce system complexity and enable tractable computations. Additionally, since the \emph{probability matrix} consists of nonnegative real numbers, with each row summing to one, it is reasonable to represent $\probmatrix$ using the \emph{Hadamard parameterization} \cite{li2023simplexHadamard}, $X\odot X,\,X\in\oblique(\abs{\mdps},\abs{\mdps})$. Consequently, imposing a rank constraint on $X$, we propose the following model to find reduced-dimension representations of Markov processes,
\begin{equation}\label{eq:markovaggregation}
    \begin{aligned}
    \min_{X\in\mbR^{\abs{\mdps}\times \abs{\mdps}}}\ \ &\frac{1}{2}\|X\odot X-\hat{\probmatrix}\|^2_\frob
    \\[-1.5mm]
    \mathrm{s.\,t.}\ \ \ \ \,&\rank(X)\le r,
        \\
        &\diag(XX^\top)-{\bf 1}=0,
    \end{aligned}
\end{equation}
where $\hat{\probmatrix}$ is an empirical estimate of the ground-truth $\probmatrix$. Formulation \eqref{eq:markovaggregation} provides fresh insights on identifying low-rankness in Markov processes, which will aid in extracting state features and thus facilitate various downstream tasks including state abstraction and reinforcement learning \cite{abel2018stateabstrac}.

\paragraph{\emph{\textbf{Reinforcement learning (RL).}}} To make sequential decisions in Markov processes, RL is an effective framework, which aims for the optimal policy maximizing the expected cumulative reward. Specifically, with $\mdpa$ representing the action space, a \emph{policy} $\pi\in\mathbb{R}^{\abs{\mdps}\times \abs{\mdpa}}$ serves as an action selection rule, i.e., $\pi(s,a)$ gives the probability of performing action~$a$ at state~$s$. Moreover, in a Markov decision process (MDP), a state-action pair $(s,a)\in\mdps\times\mdpa$ incurs a reward ${R}(s,a)$ and transitions to state $s^\prime\in\mdps$ with probability $\probmatrix_{\mathrm{d}}(s,a,s^\prime)$, where the tensor $\probmatrix_{\mathrm{d}}\in\mbR^{\abs{\mdps}\times\abs{\mdpa}\times\abs{\mdps}}$ denotes the transition dynamics. Consequently, given $\mu$ as the initial state distribution and $\gamma \in (0,1)$ as the discount factor, the central task of RL is to maximize the objective $J(\pi):=\mathbb{E}[\sum_{t=0}^\infty \gamma^t {R}(s_t,a_t)\ |\ s_0\sim\mu,\pi]$. However, the optimization problem suffers from the curse of dimensionality when the state space $\mdps$ and the action space $\mdpa$ are large \cite{sutton2018RL}. To alleviate this, exploiting the low-rank structure in RL is an advisable approach. Specifically, recognizing that a policy is inherently a probability matrix, we adopt the Hadamard parameterization $\pi(X):=X\odot X,\,X\in\oblique(\abs{\mdps},\abs{\mdpa})$, introduce a low-rank requirement on the variable $X$, and thus propose the following low-rank formulation for reinforcement learning,
\begin{equation}\label{eq:lowrankRL}
    \begin{aligned}
    \min_{X\in\mbR^{\abs{\mdps}\times\abs{\mdpa}}}\ \ & -J(\pi(X))
    \\[-1.5mm]
    \mathrm{s.\,t.}\ \ \ \ \ \,&\rank(X)\le r,
        \\
        &\diag(XX^\top)-{\bf 1}=0.
    \end{aligned}
\end{equation}

\paragraph{\emph{\textbf{Neural network training.}}} Recent success in deep learning has witnessed the critical role of neural network architectures in both training efficiency and inference performance. Among various designs, two principles receive increasing attention. Typically, one is \emph{weight normalization}, which expresses the weight matrix $W\in\mbR^{m\times n}$ by $W= (g{\bm 1}^\top)\odot X$ with $g\in\mbR^{m}$ and $X\in\oblique(m,n)$. Therefore, each row of $X$ encodes the direction, while $g$ is responsible for the scaling. It is reported that such a separation enhances training stability \cite{salimans2016WN}. Meanwhile, another principle, namely \emph{low-rank compression}, harnesses the inherent low-rankness of large-scale network weights, effectively circumventing parameter redundancy and preserving decent generalization \cite{idelbayev2020lowrankcompress}. Integrating the two principles---weight normalization and low-rank compression---presents an enlightening avenue for further development. Specifically, for the $i$-th layer in a neural network, we train the normalized weight $X_i$ subject to the low-rank constraint, together with scale parameters $g_i$, leading to the following training model,
\begin{equation}\label{eq:lowrankNN}
    \begin{aligned}
    \min_{\{X_i,g_i\}_{i=1}^l}\ \ &f(\{X_i,g_i\}_{i=1}^l)
    \\[-1.5mm]
    \mathrm{s.\,t.}\ \ \ \ \ \,&\rank(X_i)\le r_i,
    \\
    &\diag(X_iX_i^\top)-{\bf 1}=0,\ \text{for}\ i=1,2,\ldots,l,
    \end{aligned}
\end{equation}
where $f$ is the loss function and $l$ is the number of layers.

\subsection{Motivation and related work}

Apart from the aforementioned real-world applications, many problems exhibit the structure of low-rankness and orthogonal invariance. However, a unified framework to address problem \eqref{eq:lowrank_orthinvar} is lacking, which essentially suffers from three pains.
\begin{enumerate}[left=0pt]
\item The optimality conditions of problem \eqref{eq:lowrank_orthinvar} remain unexplored, largely due to the involved geometry of the coupled constraints. This unawareness poses an impediment to defining stationarity measures, thereby hindering the development of algorithms.
\item The feasible region inherits the nonsmoothness of $\boundedrank$, which gives rise to the undesirable phenomenon: there exist feasible sequences along which the stationarity measures tend to zero, but with limit points non-stationary \cite{levin2023remedy}.
\item Projections onto the coupled feasible set $\boundedrank\cap\hanifold$ are unclear for general $\hanifold$, making it intractable to translate existing techniques in low-rank optimization, such as the projected gradient descent framework adopted by previous works~\cite{schneider2015Lojaconvergence,olikier2022P2GDR,olikier2024PGD}.
\end{enumerate}

\paragraph{{\bf Optimality conditions.}} Research on optimality conditions gains momentum in low-rank optimization, where the variational geometry, including tangent and normal cones to the feasible set, plays a pivotal role. Typically, several kinds of cones---the Mordukhovich normal cone \cite{luke2013Mordukhovich}, the Bouligand tangent cone \cite{cason2013iterative,schneider2015Lojaconvergence}, and the Clarke tangent cone and the corresponding normal cone \cite{hosseini2019MordukhovichClarke,li2019optimalitylowrank}---to the bounded-rank set $\boundedrank$ have been well studied. 

However, an additional constraint set $\canifold$ coupled with $\boundedrank$ renders the geometry of the feasible region more complicated. Specifically, when $m=n$ and $\canifold =\sym(n)$, Tam et al. \cite{tam2017sparsesdp} derived the Mordukhovich normal cone to $\mathbb{R}^{n\times n}_{\le r}\cap\canifold$; and the Fr\'echet normal cones were formulated by Li et al. \cite{li2020jotaspectral} when $\canifold$ is the intersection of $\sym(n)$ with the closed unit Frobenius ball, the symmetric box or the spectrahedron. In essence, the above results stem from the symmetry of $\canifold$, which coincide with the preimages of some permutation-invariant sets $\canifold_{\mathrm{p}}\subseteq\mbR^n$ under the eigenvalue map. This observation reduces the low-rankness of matrices in $\canifold$ to the sparsity of vectors in~$\canifold_{\mathrm{p}}$. Therefore, computing the normal cone to $\{v\in\mbR^n:\|v\|_0\le r\}\cap\canifold_{\mathrm{p}}$ by \cite[Proposition 3.2]{pan2017restrictedLICQ} and applying \cite[Theorem 8.3]{lewis1996groupinvariance} reveal the results in \cite{tam2017sparsesdp,li2020jotaspectral}, where $\|\cdot\|_0$ denotes the cardinality of a vector. It is worth noting that these developments require the matrix variable to be square and symmetric; but when $m\neq n$ breaks the symmetry, existing results on sparsity optimization \cite{pan2017restrictedLICQ} no longer apply, which necessitates analysis tailored for the geometry of the coupled feasible region. 

Recently, Li and Luo \cite{li2023normalboundedaffine} characterized the normal cone to $\boundedrank\cap\canifold$ with $\canifold$ as an affine manifold. More relevantly, a closed-form expression of the tangent cone to $\boundedrank\cap\mathrm{S}_{\mathrm{F}}(m,n)$ was given in~\cite{cason2013iterative}, which turns out to be a special instance of the feasible region in {problem}~\eqref{eq:lowrank_orthinvar}. 

\paragraph{{\bf Optimization on bounded-rank matrices.}} In general, finding a global minimizer of a function solely subject to the bounded-rank constraint is NP-hard \cite{gillis2011NPlowrank}. Existing literature~\cite{schneider2015Lojaconvergence,levin2023remedy,olikier2023RFDR} reveals that it remains possible to compute a (first-order) \emph{B-stationary} point, a zero of the B-stationarity measure---norm of the projection of negative gradient onto the Bouligand tangent cone to~$\boundedrank$\revise{, denoted by $s(X):=\|\projection_{\tangent_X\boundedrank}(-\nabla f(X))\|_\frob$}. However, the nonsmooth and nonconvex nature of the bounded-rank set introduces the main obstacles. \revise{Typically, the phenomenon \emph{apocalypse} may occur: a convergent sequence $\{X_k\}$ satisfies $\lim_{k\to\infty}s(X_k)=0$ while $s(\lim_{k\to\infty}X_k)>0$~\cite{levin2023remedy}.} A line of algorithms is developed for the optimization on bounded-rank matrices, and they are categorized into three groups.

The first group is based on the projected gradient descent (PGD) framework, i.e., updating the iterate along an appropriate direction and then projecting it onto the bounded-rank set, which amounts to computing a truncated singular value decomposition (SVD) of a matrix. The PGD method in \cite{olikier2024PGD} has recently been proven to accumulate at B-stationary points. The well-known~$\mathrm{P^2GD}$~\cite{schneider2015Lojaconvergence}, additionally projecting negative gradient onto the tangent cone as the search direction, can result in the apocalypse; an explicit example was given in \cite{levin2023remedy}. To resolve this, Olikier et al.~\cite{olikier2022P2GDR} introduced a rank-adaptive strategy and obtained a modification $\mathrm{P^2GDR}$.

The second group evolves around the retraction-free principle. In detail, search directions are chosen from the so-called \emph{restricted tangent cone} to~$\boundedrank$ (see \cite{olikier2023RFDR}) such that updates are performed along straight lines, and thus the algorithm gets rid of expensive projections onto the bounded-rank set. The original retraction-free descent (RFD) method~\cite{schneider2015Lojaconvergence} may encounter the apocalypse, and the variant RFDR \cite{olikier2023RFDR} equipped with a rank reduction mechanism is guaranteed with the convergence to B-stationary points. Recently, Olikier and Absil~\cite{olikier2024ERFDR} proposed a more efficient version named ERFDR, which preserves the benign convergence property of RFDR. 

The third group resorts to the parameterization of the bounded-rank set, which constructs a smooth manifold $\manifold$ and an associated mapping $\phi$ such that $\phi(\manifold)=\boundedrank$. Subsequently, the original problem on the bounded-rank set can be reformulated as minimizing $\bar{f}:=f\circ \phi$ over $\manifold$, a Riemannian optimization problem. For instance, the LR lift \cite{park2018findingLR,ha2020equivalenceLR} uses $\phi(L,R)=LR^\top$ to factorize the low-rank matrix. Levin~et~al.~\cite{levin2024effectlift} proposed the parameterization $\phi(U,S,V)=USV^\top$ with symmetric $S$ and orthogonal~$U$ and~$V$. Moreover, viewing the bounded-rank set as the determinantal variety~\cite{harris1992algebraic}, Khrulkov and {Oseledets}~\cite{khrulkov2018desingularization} derived the \emph{desingularization}, which is a modified version of the classical Room--Kempf procedure~\cite{room1938,kempf1973}, and its geometry was further explored in a recent work~\cite{rebjock2024boundedrank}. 
A benefit of these parameterizations is the ability to circumvent the apocalypse; see~\cite{levin2024effectlift}.

\bigskip

Translating the existing methods to address problem~\eqref{eq:lowrank_orthinvar} appears challenging. In general, the PGD-based method stagnates due to the unknown projection onto the coupled feasible set. The retraction-free technique relies on 
the specific structure of $\boundedrank$ while the additional constraint $\hanifold$ distorts the landscape, presenting obstacles to promoting the technique. Moreover, finding a parameterization for $\boundedrank\cap\hanifold$ requires new insights, and in this scenario, the equivalence between the reformulated problem and the original problem is worth further investigation.

\subsection{Contributions} In this paper, we develop a space-decoupling framework by taking advantage of the geometry of the bounded-rank set and the extra orthogonally invariant constraints. The space-decoupling principle features the following aspects.

General properties of the orthogonally invariant mapping~$h$ are studied, which turn out to be closely related to rank information of the matrix variable. Specifically, Assumption~\ref{assu:h} reveals that $\hanifold$---the level set of $h$---is a smooth manifold. \revise{We further show that $h$ induces a sequence of manifolds $\hanifold^s$ embedded in lower-dimensional spaces $\mathbb{R}^{m\times s}$ ($s=1,2,\ldots,n$), which inherit the geometric properties from $\hanifold$. Then, Proposition~\ref{pro:decompose_geo_H} shows that when $\rank(X)=s$, the tangent space to $\hanifold$ at $X$ is decomposed into a space related to $\hanifold^s$ and a space orthogonal to~$X$.} The results build up a connection between $\boundedrank$ and $\hanifold$, paving the way for investigating the geometry of~$\boundedrank\cap\hanifold$.

We give a closed-form expression for the tangent cone to the coupled feasible set (Theorem~\ref{the:tangent_intersection_rule}), which essentially unveils the intersection rule---the tangent (resp. normal) cone to $\boundedrank\cap\hanifold$ can be decomposed as the intersection (resp. direct sum) of tangent (resp. normal) cones to each of them, i.e.,
\begin{align*}
\tangent_X\kh{\boundedrank\cap\hanifold} &= \tangent_X\boundedrank \cap \tangent_X\hanifold,
\\
\normal_X\kh{\boundedrank\cap\hanifold} &= \normal_X^\frechet\boundedrank \oplus \normal_X\hanifold.
\end{align*}
Moreover, the optimality conditions for \eqref{eq:lowrank_orthinvar} are developed.

\revise{Viewing the \emph{Grassmann manifold}~\cite{bendokat2024grassmann} as an embedded submanifold in $\sym(n)$, i.e., $\grassmann(n,s):=\{G\in\sym(n): G^2=G,\,\rank(G)=s\}$, we borrow the idea from the desingularization technique~\cite{khrulkov2018desingularization,rebjock2024boundedrank}, and propose the following \emph{space-decoupling parameterization} for the coupled constraints $\boundedrank\cap\hanifold$,}
\begin{equation*}
    \manifold_{h}:=\left\{(X, G) \in \mathbb{R}^{m \times n} \times \grassmann(n, n-r): X G=0,\ h(X)=0\right\}
\end{equation*}
with the smooth mapping $\phi:\mbR^{m\times n}\times\sym(n)\to\mbR^{m\times n}:(X,G)\mapsto X$ satisfying $\phi(\manifold_h)=\boundedrank\cap\hanifold$. We prove that the set $\manifold_h$ is a smooth manifold (Theorem~\ref{the:manifold_h}), and thus reformulate the original nonsmooth coupled-constrained optimization problem as a smooth Riemannian optimization problem,
\begin{align}
    \min_{(X,G) \in \manifold_h} \bar{f}(X,G):=f\circ \phi(X,G). \label{eq:G}\tag{P-${\manifold_h}$}
\end{align}
Specifically, an element $(X,G)$ in~$\manifold_h$ has two coordinates, with $X$ subject to the orthogonally invariant constraints and $G$ encoding the rank information. In this manner, the parameterization decouples the intertwined constraints into two spaces; see Fig.\,\ref{fig:diagram} for an illustrative diagram.

\begin{figure}[h]
    \newcommand{\ratio}{0.5}
    \newcommand{\customrectangle}[7]{
            \draw[densely dashed, draw=#6, fill=#7] 
                (#1-#3/2, #2-#4/2) rectangle 
                (#1+#3/2, #2+#4/2); 
            \node at (#1, #2) {#5}; 
        }
    \begin{minipage}{1\textwidth}
    \begin{center}
    \begin{tikzpicture}
        \newcommand{\nodegap}{1.3}
        \tikzset{
        node distance=\nodegap cm,
        post/.style={->,shorten >=2pt,shorten <=2pt,>={Stealth[round]},thick},
        space/.style={
          draw=none, 
          fill=none, 
          inner sep=0pt,
          minimum size=6mm
        },
        hookarrow/.style={{Hooks[left]}->}
        }
        
        \node[space] (Mh) {\normalsize $\manifold_h$};
        \node[space] (ori) [below=1.2cm of Mh] {\normalsize $\boundedrank\cap\hanifold$};
        \node[space] (R) [right=2.5cm of ori] {\normalsize $\mbR$};

        \node[space] (embedding) [left=3.5cm of Mh] {\normalsize $\mbR^{m\times n} \times \sym(n)$};

        \draw[post] (Mh) to node[midway, left] {\normalsize $\phi$} (ori);
        \draw[post] (Mh) to node[midway, above, xshift=16pt] {\normalsize $\bar{f}=f\circ \phi$} (R);
        \draw[->,shorten <=4pt,>={Stealth[round]},thick] (ori) to node[midway, below] {\normalsize $f$} (R);
        \draw[->,hookarrow,shorten <=4pt,shorten >=4pt, >={Stealth[round]},thick] (Mh) -- (embedding);

        \node (XG) at ([xshift=1.5cm, yshift=-1.85cm]embedding) {
        \begin{tikzpicture}
            \customrectangle{0}{0}{5.5*\ratio}{1*\ratio}{\normalsize rank information}{black}{gray!0}
        \end{tikzpicture}
        };

        \node (hX) at ([xshift=-1.3cm, yshift=-1.85cm]embedding) {
        \begin{tikzpicture}
            \customrectangle{0}{0}{3.6*\ratio}{1*\ratio}{\normalsize $h(X)=0$}{black}{gray!0}
        \end{tikzpicture}
        };
        
        \draw[-, thick] 
            ([shift={(0.6,0)}] embedding.south) 
            -- ([shift={(0,0.)}] XG.north); 
        \draw[-, thick] 
            ([shift={(-1,0)}] embedding.south) 
            -- ([shift={(0,0.)}] hX.north);

    \end{tikzpicture}
    \end{center}
    \end{minipage}
    \caption{Illustration of optimization via the space-decoupling parameterization.}
    \label{fig:diagram}
\end{figure}
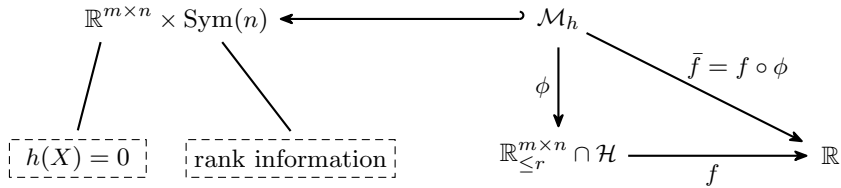

To facilitate Riemannian optimization methods, we provide Riemannian derivatives, retractions, and vector transports on $\manifold_h$. It reveals that implementing Riemannian algorithms on $\manifold_h$ is painless if the geometry of $\hanifold$ is understood \revise{a priori}. Moreover, we identify the equivalence between the reformulated and the original optimization problems (Theorem~\ref{pro:11_21_Mh}): if $(X,G)$ is a second-order stationary point for \eqref{eq:G} then $X=\phi(X,G)$ is a first-order stationary point for \eqref{eq:lowrank_orthinvar}; and when $\rank(X)=r$, the first-order stationarity of $(X,G)$ for~\eqref{eq:G} concludes the first-order stationarity of $X$ for~\eqref{eq:lowrank_orthinvar}. Convergence guarantees for the algorithms on $\manifold_h$ are also provided.

The effectiveness and efficiency of the proposed framework are validated by fruitful numerical experiments across applications discussed in section~\ref{sec:applications}. Note that the proposed models---\eqref{eq:markovaggregation} for Markov process reduction, \eqref{eq:lowrankRL} for low-rank reinforcement learning, and \eqref{eq:lowrankNN} for deep learning---exhibit good performance and thus contribute to their respective fields. In practice, Riemannian algorithms generate a sequence $\{(X_k,G_k)\}$ on $\manifold_h$, ensuring that $\{X_k\}$  satisfies the constraints of~\eqref{eq:lowrank_orthinvar}. This fact is favorable not only to the early stopping of an algorithm when a feasible solution is required but also to some specific problems (e.g.,~\eqref{eq:lowrankRL}) where the objective cannot be defined outside of the feasible region.

\subsection{Notation} 
Let $\mbR^{m\times n}_s=\{X\in \mbR^{m\times n}: \rank(X)=s\}$ be the set of fixed-rank matrices, and $\sksym(n):=\{\Omega\in\mbR^{n\times n}:\,\Omega^\top=-\Omega\}$ be the set of skew-symmetric matrices. The symmetric part of a square matrix is defined by $\symac(M):=\frac{1}{2}(M+M^\top)$. Given a smooth manifold $\xanifold$, $\tangent\xanifold$ denotes the tangent bundle, and $\tangent_X\xanifold$ denotes the tangent space at $X\in\xanifold$. Given a mapping $F:\xanifold_1 \to \xanifold_2$ between two manifolds, $\diff F_{X}: \tangent_X\xanifold_1 \to \tangent_{F(X)}\xanifold_2$ denotes the differential of $F$ at $X\in\xanifold_1$. Given a matrix $V\in\stiefel(n,r)$, $V_{\bot}$ is an orthonormal completion of it in the sense of $[V\ V_\bot]\in\orth(n)$.  The standard inner product in an Euclidean space is given by $\innerp{X_1,X_2}:=\trace(X_1^\top X_2)$. Let $\projection_\xanifold $ denote the projection onto the set $\xanifold$. \revise{Throughout this paper, whenever the rank of $X\in\mbRmn$ is explicitly specified, e.g., $\rank(X)=s$ or $X\in\lowrank$, and the SVD $X=U\varSigma V^\top$ is considered, the factor sizes are by default taken as $U\in\stiefel(m,s)$, $\varSigma\in\mbR^{s\times s}$, and $V\in\stiefel(n,s)$.}

\subsection{Organization}
In section~\ref{sec:H}, we investigate the general properties of the orthogonally invariant mapping $h$ and the associated level set $\hanifold$. The tangent and normal cones to the coupled feasible region are identified, and the optimality conditions are also formulated in section~\ref{sec:OptCondition}. Section~\ref{sec:M_h} proposes the space-decoupling parameterization $\manifold_h$ and its Riemannian geometry. In section~\ref{sec:RiemannianOpt}, we specify the Riemannian derivatives, retractions, and vector transports on $\manifold_h$, analyze the equivalence between \eqref{eq:G} and \eqref{eq:lowrank_orthinvar}, and provide the convergence guarantees for Riemannian algorithms. Section~\ref{sec:experiments} reports the numerical performance of the proposed framework. Finally, we draw the conclusion in section~\ref{sec:conclusion}.

\section{Properties of the orthogonally invariant mapping}\label{sec:H}
This section uncovers the connection between the orthogonal invariance of the mapping $h$ and the rank information of the variable $X$, and thus bridges the geometry of $\hanifold$ and $\boundedrank$.

The \emph{Bouligand tangent cone} \cite{rockafellar2009variationalanalysis} to a set $\mathcal{X}$ at a point $X\in\xanifold$ is 
\begin{equation*}
    \tangent_X\mathcal{X} :=\left\{ \eta =\lim _{k\rightarrow \infty} \frac{\left( X_k-X \right)}{t_k}:\ X_k\in \mathcal{X} , t_k>0\ \text{for all}\,\,k, \lim_{k\rightarrow\infty}t_k = 0 \right\}.
\end{equation*}
A vector $\eta\in\tangent_X\xanifold$ is \emph{derivable} if it admits a curve $\gamma(t)$ on $\xanifold$ with $\gamma(0)=X$ and $\gamma^\prime(0)=\eta$, and the set $\xanifold$ is \emph{geometrically derivable} at $X$ if any vector in $\tangent_X\xanifold$ is~{derivable}. Taking the polar operation on $\tangent_X\mathcal{X}$ introduces the \emph{Fr\'echet normal cone}, $\normal_X\mathcal{X}:=\kh{\tangent_X\mathcal{X}}^\circ$. Specifically, if $\mathcal{X}$ is a smooth manifold, the tangent (resp. normal) cone reduces to a vector space, i.e., the tangent (resp. normal) space. 

\subsection{Geometry of $\hanifold$}
The full-rankness of $h$ in Assumption~\ref{assu:h} implies that the level set $\hanifold$ is a smooth manifold embedded in $\mbR^{m\times n}$; see~\cite[Corollary 5.14]{lee2012manifolds}. Moreover, the orthogonal invariance of $h$ sheds light on the fact that the tangent space $\tangent_X\hanifold$ contains a subspace that is independent of the choice of $h$; see the following proposition. 

\begin{proposition}\label{pro:kerDh}
    Given $X\in\hanifold$ with $\rank{(X)}=s$ and the SVD $X=U\varSigma V^\top$, it holds that
    \begin{equation}\label{eq:mT_X}
    \kanifold_X := \hkh{U\varSigma\Omega V^\top + BV_{\bot}^{\top}:\,\Omega\in\sksym(s),\,B\in\mbR^{m\times (n-s)}} \subseteq \tangent_X\hanifold.
    \end{equation}
\end{proposition}
\begin{proof}
    The fact $\tangent_X\hanifold=\kernel(\diff h_X)$ shifts the focus to prove $\kanifold_X\subseteq \kernel(\diff h_X)$. Notice that $\kanifold_X$ is a linear space and any $\eta\in\kanifold_X$ can be expressed in the form of $\eta = \eta_1 + \eta_2$ where $\eta_1 = U\varSigma \Omega V^{\top}+UB_{1}V_{\bot}^{\top}$ and $\eta_2 = U_{\bot}B_{2}V_{\bot}^{\top}$ with $\Omega\in\sksym\kh{s},\,B_1\in\mbR^{s\times \kh{n-s}}$, and $B_2\in \mbR^{\kh{m-s}\times \kh{n-s}}$.

    We first validate $\eta_1\in\kernel(\diff h_X)$. In view of the orthogonal invariance of $h$, the mapping $F:\stiefel(n,s) \rightarrow \mathbb{R}^{q}:\,\tilde{V}\mapsto h(U\varSigma \tilde{V}^\top)$ is constant since $F(\tilde{V})=h(U\varSigma\tilde{V}^\top)=h(XQ_{V}^{}Q_{\tilde{V}}^{\top})=h(X)$ where $Q_V:=[ {V}\,\,{V}_{\bot} ]\in\orth(n)$ and $Q_{\tilde{V}}:=[ \tilde{V}\,\,\tilde{V}_{\bot} ]\in\orth(n)$. Therefore, the differential $\diff F_{\tilde{V}}\equiv0$ for all $\tilde{V}\in\stiefel(n,s)$ including $V$. Taking into account the tangent space $\tangent_{{V}}\stiefel(n,s)$ (see~\cite[\S3.5]{absil2008optimization}) and given ${\xi}=-V\Omega+V_{\bot}B_1^{\top}\varSigma^{-1}\in\tangent_{{V}}\stiefel(n,s)$, it holds that $\diff {h}_{X}[\eta_1] = \diff {h}_{X}[U\varSigma{\xi}^{\top}] = \frac{\mathrm{d}}{\mathrm{d}t}h( X+tU\varSigma{\xi}^{\top} ) \big |_{t=0} = \frac{\mathrm{d}}{\mathrm{d}t}F(V+t{\xi}) \big |_{t=0} = \diff F_{{V}}[{\xi}] = 0$.

    Due to the orthogonal invariance of $h$, it is direct to obtain $h(X+t\eta_2)=h((X+t\eta_2)Q)=h(X-t\eta_2)$ for $t\in \mbR$ with $Q=[ {V}\,-\!{V}_{\bot} ][V\,V_\bot]^\top\in\orth(n)$. Thus, it implies $\diff h_X[\eta_2] = \frac{\mathrm{d}}{\mathrm{d}t}h\left( X+t\eta_2 \right) \big |_{t=0} = 0$.
\end{proof}

The following corollary reveals that any element orthogonal to $X$ falls into $\tangent_X\hanifold$, which plays an important role in subsequent developments.

\begin{corollary}\label{cor:orth_tangentX}
    Given $X\in\hanifold$, if $Y\in\mbR^{m\times n}$ satisfies $XY^\top=0$, it holds that $Y\in\tangent_X\hanifold=\kernel(\diff h_X)$.
\end{corollary}
\begin{proof}
    The SVD of $X=U\varSigma V$ and $X Y^\top=0$ imply that $Y=BV_\bot^\top$, which is in the subspace $\kanifold_X\subseteq\tangent_X\hanifold$ from Proposition~\ref{pro:kerDh}.
\end{proof}

Proposition~\ref{pro:kerDh} demonstrates that for any orthogonally invariant $h$, there is a subset $\kanifold_X\subseteq\tangent_X\hanifold$. Taking the polar operation on both sides yields $\normal_X\hanifold = \kh{\tangent_X\hanifold}^\circ \subseteq \kanifold_X^\circ$, which leads to the following lemma.

\begin{lemma}\label{lem:normal_H}
    Given $X\in\hanifold$ with $\rank{(X)}=s$ and the SVD $X=U\varSigma V^\top$, it holds that
    \begin{equation}\label{eq:N_XH_estimate}
        \normal_X\hanifold \subseteq \kanifold_X^\circ = \left\{ AX:\ A\in \mathbb{R} ^{m\times m}, \ U^\top AU\in \sym(s)\right\}.
    \end{equation}
\end{lemma}
\begin{proof}
    It suffices to identify $\kanifold_X^\circ$, which is the orthonormal completion of the linear space $\kanifold_X$. Specifically, any $\eta\in\mbR^{m\times n}$ can be expressed in the form of $\eta = AX+\tilde{A}V_{\bot}^{\top}$ with $A\in \mathbb{R} ^{m\times m}$ and $\tilde{A}\in \mathbb{R} ^{m\times \left( n-s \right)}$. It follows from $\eta\in\kanifold_X^\circ$ and the expression of $\kanifold_X$ in~\eqref{eq:mT_X} that $\langle AX+\tilde{A}V_{\bot}^{\top},U\varSigma \Omega V^{\top}+BV_{\bot}^{\top} \rangle =0$ for all $\Omega\in\sksym(s)$ and $B\in\mbR^{m\times (n-s)}$, reducing to ${\langle AU\varSigma ,U\varSigma \Omega \rangle +\langle \tilde{A},B \rangle=0}$.
    Consequently, the arbitrariness of $\Omega$ and $B$ implies $U^\top AU\in \sym(s)$ and $\tilde{A}=0$.
\end{proof}

Note that the developed properties of $\hanifold$ extract the rank of $X$ via the singular value decomposition, and thus link the orthogonal invariance of $h$ to the rank information.

\subsection{Inheritance principle}\label{sec:inheritance}
Based on the above results, we demonstrate that the properties of $h$ can be inherited from $\mbR^{m\times n}$ to lower-dimensional spaces. For $s=1,2,\ldots,n$, we define a sequence of natural embeddings $i^s:\,\mbR^{m\times s} \rightarrow \mbR^{m\times n}:\,H\mapsto [H\ 0^{m\times \kh{n-s}}]$, the associated mappings $h^s:=h\circ i^s$, and the level sets 
\[
\hanifold^s:=\hkh{H\in\mbR^{m\times s}:h^s(H)=0}.
\] 
Given $X\in\mbR^{m\times n}$ \revise{with $\rank(X)\le s$}. The developments of this subsection rely on a decomposition $X=HV^\top$ with $H \in \mbR^{m\times s}$ and $V \in \stiefel(n, s)$, which is always achievable in practice; e.g., singular value decomposition and QR factorization.
Subsequently, $h$ and $h^s$ are connected via the values and differentials.

\begin{lemma}\label{pro:X_H_fullrankdiff}
    Given $X\in\mbR^{m\times n}$ with a decomposition $X=HV^\top$ where $H \in \mbR^{m\times s}$ and $V \in \stiefel(n, s)$, it follows that 
    \[
    h(X) = h^s(H)\quad\text{and}\quad\diff h_X[KV^\top]=\diff h^s_H[K]\quad\text{for all}\ K\in\mbR^{m\times s}.
    \]
    Moreover, $X\in\hanifold$ if and only if $H^{}\in\hanifold^s$.
\end{lemma}
\begin{proof}
    Let $Q=[ V\,\,V_{\bot}]\in\orth(n)$. It follows from the orthogonal invariance of $h$ and $XQ =[ H\,\,0^{m\times (n-s)} ]$ that $h(X) = h(XQ) = h^s(H)$. \revise{Moreover, $h^s=h\circ i^s$ and $\diff i^s = i^s$ imply that $\diff h^s_H[K] = \diff h_{i^s(H)}[\diff i^s_H(K)] = \diff h_{i^s(H)}[i^s(K)]= \diff h_{XQ}[KV^{\top}Q]=\frac{\mathrm{d}}{\mathrm{d}t}h( XQ+tKV^{\top}Q) \big |_{t=0}=\frac{\mathrm{d}}{\mathrm{d}t}h( X+tKV^{\top}) \big |_{t=0}=\diff h_{X}[KV^{\top}]$.}
\end{proof}

Lemma~\ref{pro:X_H_fullrankdiff} indicates that $h^s$ is also orthogonally invariant, and thus the statements regarding $\hanifold$---including Proposition~\ref{pro:kerDh}, Corollary~\ref{cor:orth_tangentX}, and Lemma~\ref{lem:normal_H}---can be generalized to $\hanifold^s=\{H\in\mbR^{m\times s}:\,h^s(H)=0\}$ in parallel. Furthermore, the following proposition shows that $\hanifold^s$ is a smooth manifold embedded in a lower-dimensional space.

\begin{proposition}\label{pro:manifoldHs}
   If  $\hanifold^s$\revise{, defined as the level set of $h^s:\mbR^{m\times s}\to\mbR^q$,} is nonempty, it is an embedded submanifold in $\mbR^{m\times s}$ of dimension $(ms-q)$.
\end{proposition}
\begin{proof}
    We aim to show $\diff h^s$ is full-rank in the level set $\hanifold^s$. Using Corollary~\ref{cor:orth_tangentX}, we find $\{\zkh{0^{m\times {s}}\,\,C}:C\in\mbR^{m\times \kh{n-s}}\} \subseteq\kernel(\diff h_{i^s(H)})$ for $H\in\hanifold^s$. This observation, together with the full-rankness of $\diff h$, reveals that any $b\in\mbR^q$ admits a preimage $\eta$ in the form of $\eta = [\eta^s\,\,0^{m\times (n-s)}] \in\mbR^{m\times n}$, and thus $b=\diff h_{i^s( {H} )}[ \eta ] =\diff h^s_{H}[ \eta^s ]$. The arbitrariness of $H\in\hanifold^s$ and $b\in\mbR^q$ implies $\diff h^s$ has the full rank $q$ in $\hanifold^s$. At last, we apply \cite[Corollary 5.14]{lee2012manifolds} to arrive at the conclusion.
\end{proof}

In succession, we show the relationship between the geometry of $\hanifold$ and $\hanifold^s$.

\begin{proposition}\label{pro:decompose_geo_H}
    Given $X\in\mbR^{m\times n}$ with a decomposition $X=HV^\top$ where $H \in \mbR^{m\times s}$ and $V \in \stiefel(n, s)$, it holds that
    \begin{align}
        \tangent_X\hanifold =& \hkh{KV^\top + BV^\top_\bot:\ K\in\tangent_H\hanifold^s,\,B\in\mbR^{m\times (n-s)}} \label{eq:decompose_TXH}
        \\
        =& \kh{\tangent_H\hanifold^s}V^\top \oplus \kh{\mbR^{m\times(n-s)}}V^\top_\bot, \label{eq:decompose_TXH1}
    \end{align}
    where $\oplus$ denotes the direct sum. Moreover,
    \begin{equation}\label{eq:decompose_NXH}
         \normal_X\hanifold = \hkh{NV^\top:\ N\in\normal_H\hanifold^s} = \kh{\normal_H\hanifold^s}V^\top.
    \end{equation}
\end{proposition}
\begin{proof}
    By Lemma~\ref{pro:X_H_fullrankdiff}, we have $\diff h_X[KV^\top]=\diff h^s_H[K]=0$ for $K\in\tangent_H\hanifold^s$, which implies $KV^\top\in\tangent_X\hanifold$, and thus $\kh{\tangent_H\hanifold^s}V^\top\subseteq\tangent_X\hanifold$. Combining this result with $\{BV^\top_\bot:\ B\in\mbR^{m\times (n-s)}\}\subseteq\tangent_X\hanifold$ revealed by \eqref{eq:mT_X}, we conclude the ``$\supseteq$" part of the equality in~\eqref{eq:decompose_TXH}. Moreover, the dimension of the space on the right side of \eqref{eq:decompose_TXH} is $ \dime(\hanifold^s)+\dime(\mbR^{m\times(n-s)})=(ms-q) + m(n-s) = mn-q$, which coincides with $\dime(\tangent_X\hanifold)$. Therefore, we confirm the equality in \eqref{eq:decompose_TXH}. Taking the orthogonal complement on both sides leads to $\normal_X\hanifold = \kh{\normal_H\hanifold^s}V^\top$.
\end{proof}

Proposition~\ref{pro:decompose_geo_H} decomposes $\tangent_X\hanifold$ into two components: one originating from the tangent space to $\hanifold^s$ and the other orthogonal to $X$. As a byproduct, it is straightforward from~\eqref{eq:decompose_TXH1} that 
\begin{equation*}
    \tangent_H\hanifold^s = \kh{\tangent_X\hanifold}V.
\end{equation*}
In other words, the geometry of $\hanifold$ enlightens that of $\hanifold^s$.

Consequently, computing the projection of $E\in\mbR^{m\times n}$ onto $\tangent_X\hanifold$ reduces to the following optimization problem,
\begin{align*}
    \argmin_{K\in\tangent_H\hanifold^s,\,B\in\mbR^{m\times(n-s)}}\|E - (KV^\top + BV^\top_\bot)\|_\frob^2=\|EV-K\|_\frob^2+\|{EV_\bot-B}\|_\frob^2.
\end{align*}
Collecting the solution $K^*=\projection_{\tangent_H\hanifold^s}(EV),\,B^*=EV_\bot$ yields
\begin{equation}\label{eq:proj_TXH}
    \projection_{\tangent_X\hanifold}(E) = \projection_{\tangent_H\hanifold^s}(EV)V^\top + EV_\bot V_\bot^{\top},
\end{equation}
\revise{which is decomposed into two parts: a projection onto $\tangent_{H}\hanifold^s$ in a lower-dimensional space, and an additional term along the orthogonal direction of $X$. The following example verifies~\eqref{eq:proj_TXH} for $\hanifold=\oblique(m,n)$.}

\medskip

\revise{Let $X\in\hanifold=\oblique(m,n)$ with $\rank(X)=s$ and $X=HV^\top$, where $H\in\hanifold^s=\oblique(m,s)$ and $V\in\stiefel(n,s)$. 
It is known that $\normal_{X}\hanifold=\{DX : D\in\mathbb{R}^{m\times m} \ \text{is diagonal}\}$. Hence, $\projection_{\tangent_X\hanifold}(E)=E-\projection_{\normal_X\hanifold(E)}=E-[I\odot(XE^\top)]X$.
Similarly, we have
\[
\projection_{\tangent_H\hanifold^s}(EV)
=EV-[I\odot(H(EV)^\top)]H
=EV-[I\odot(XE^\top)]H.
\]
Substituting the expression of $\projection_{\tangent_H\hanifold^s}(EV)$ into the right side of~\eqref{eq:proj_TXH} leads to
\[
\begin{aligned}
\projection_{\tangent_H\hanifold^s}(EV)V^\top + EV_\bot V_\bot^\top
&=EVV^\top - [I\odot(XE^\top)]HV^\top + EV_\bot V_\bot^\top \\
&=E-[I\odot(XE^\top)]X\\
&=\projection_{\tangent_X\hanifold}(E).
\end{aligned}
\]
}

\section{Optimality conditions}\label{sec:OptCondition} 
In this section, we explore the geometry of~$\lowrank\cap\hanifold$, and then employ the results to identify the tangent and normal cones to $\boundedrank\cap\hanifold$. Consequently, the optimality conditions for problem \eqref{eq:lowrank_orthinvar} are derived.

The geometry of low-rank sets is well developed; see~\cite{vandereycken2013lowrankcompletion,schneider2015Lojaconvergence}. As a fixed-rank layer of $\boundedrank$, $\lowrank$ is indeed an analytic manifold. Given $X\in\lowrank$ with the singular value decomposition $X=U\varSigma V^\top$, the tangent and normal spaces are outlined below,
\begin{align}
    & \tangent_X\lowrank = \hkh{\left[ U\,\,U_{\bot} \right] \left[ \begin{matrix}
	\mathbb{R} ^{s\times s}&		\mathbb{R} ^{s\times \left( n-s \right)}\\
	\mathbb{R} ^{\left( m-s \right) \times s}&		0^{\left( m-s \right) \times \left( n-s \right)}\\
    \end{matrix} \right] \left[ V\,\,V_{\bot} \right] ^{\top}}, \label{eq:Tcone_lowrank}
    \\
    &\normal_X\lowrank = \hkh{\left[ U\,\,U_{\bot} \right] \left[ \begin{matrix}
	0^{s\times s}&		0^{s\times \left( n-s \right)}\\
	0^{\left( m-s \right) \times s}&		\mathbb{R} ^{\left( m-s \right) \times \left( n-s \right)}\\
    \end{matrix} \right] \left[ V\,\,V_{\bot} \right] ^{\top}}. \label{eq:Ncone_lowrank}
\end{align}
Assembling the layers $\lowrank\ (s=0,1,\ldots,r)$ yields the bounded-rank set $\boundedrank$, with its tangent and normal cones at $X$ formulated as follows,
\begin{align}
    &\tangent_X\boundedrank = \tangent_X\lowrank + \hkh{N\in\normal_X\lowrank:\ \rank\kh{N}\le r-s}, \label{eq:tangent_cone_boundedrank}
    \\
    &\normal_X\boundedrank = \begin{cases}
        \normal_X\lowrank,\,\,\quad if\,\,s=r,\\
    \left\{ 0 \right\}, \,\,\quad\quad\quad \ \,  if\,\,s<r.
    \end{cases} \label{eq:normal_cone_boundedrank}
\end{align}
\revise{Let $P_W := WW^\top$ and $P_{W\bot}:=I-P_W$ for any $W\in\stiefel(n,s)$.} The projection of $E\in\mbR^{m\times n}$ onto $\tangent_X\boundedrank$ is given by
\begin{equation}\label{eq:projection_TR}
    \projection_{\tangent_X\boundedrank}(E) = EP_V + P_UEP_{V_\bot} + \projection_{\mbR^{m\times n}_{r-s}} \kh{P_{U_\bot}EP_{V_\bot}}.
\end{equation}

\subsection{Variational geometry} Armed with the properties of $\boundedrank$ and those of $\hanifold$ obtained in section~\ref{sec:H}, we delve into the geometry of their intersection. To this end, we first revisit the basics for tangent and normal cones to the union (or intersection) of finite sets; see~\cite{lewis2008intersectionmanifolds,rockafellar2009variationalanalysis,lee2012manifolds}.
\begin{itemize}[left=0pt]
    \item If $X\in\bigcup_{i=1}^d{\xanifold_i}$, it holds that $\tangent^\bouli_X(\bigcup\nolimits_{i=1}^d{\xanifold_i}) =  \bigcup\nolimits_{i=1}^d \tangent^\bouli_X \xanifold_i$;
    \item If $X\in \xanifold_1\cap \xanifold_2$, it holds that
    \begin{equation}\label{eq:cone_oneside}
        \tangent^\bouli_X\kh{\xanifold_1\cap \xanifold_2} \subseteq  \tangent^\bouli_X \xanifold_1 \cap \tangent^\bouli_X \xanifold_2,\quad \text{and} \quad\ 
        \normal^\frechet_X\kh{\xanifold_1\cap \xanifold_2} \supseteq \normal^\frechet_X \xanifold_1 + \normal^\frechet_X \xanifold_2;
    \end{equation}
    \item If both $\xanifold_1$ and $\xanifold_2$ are smooth manifolds and intersect \emph{transversally}, i.e., for any $X\in\xanifold_1\cap\xanifold_2$, $\tangent_X{\xanifold_1} + \tangent_X{\xanifold_2} = \mbR^{m\times n},\ \text{or equivalently,}\ \normal_X\xanifold_1\cap \normal_X\xanifold_2 = \{0\}$, then $\xanifold_1\cap \xanifold_2$ is also a smooth manifold with
    \begin{equation}\label{eq:cone_tranverse}
    \tangent_X\kh{\xanifold_1\cap \xanifold_2} =  \tangent_X \xanifold_1 \cap \tangent_X \xanifold_2,\quad \text{and} \quad\ 
            \normal_X\kh{\xanifold_1\cap \xanifold_2} = \normal_X \xanifold_1 + \normal_X \xanifold_2.
    \end{equation}
\end{itemize}

Generally, studying the geometry of the constraints in problem \eqref{eq:lowrank_orthinvar} is hindered by the nonsmoothness of $\boundedrank$, for which the rule \eqref{eq:cone_tranverse} fails. To circumvent this, we take advantage of the structure $\boundedrank=\bigcup_{s=0}^{r}\lowrank$, and turn to characterize $\tangent_X\lowrank\cap\tangent_X\hanifold$ as the first step.

\begin{proposition}\label{pro:manifold_s_cap_H}
For $s=0,1,\ldots,r$, given $X\in\lowrank\cap\hanifold$ with the singular value decomposition $X=U\varSigma V^\top=HV^\top$ where $H=U\varSigma$, it holds that
\begin{align}
    \tangent_X\lowrank \cap \tangent_X\hanifold  = \hkh{KV^\top + U J V_\bot^\top:\ K\in\tangent_H\hanifold^s,\,J\in\mbR^{s\times(n-s)}}. \label{eq:manifold_s_cap_H}
\end{align}
\end{proposition}
\begin{proof}
    \revise{We begin by proving the ``$\subseteq$" part.} Given a tangent vector $\eta=KV^\top + UJV_\bot^\top\in\tangent_X\lowrank$ from~\eqref{eq:Tcone_lowrank}. If additionally $\eta\in\tangent_X\hanifold$, i.e., $\diff h_{X}[KV^\top + UJV_\bot^\top]=0$, we have $\diff h_{X}[KV^\top]=0$ since $UJV_\bot^\top\in\kernel(\diff h_X)$ via Corollary~\ref{cor:orth_tangentX}. Applying Lemma~\ref{pro:X_H_fullrankdiff} shows $K\in\kernel(\diff h^s_{H})=\tangent_H\hanifold^s$, which concludes the ``$\subseteq $" part in \eqref{eq:manifold_s_cap_H}. Moreover, the ``$\supseteq$" part results from the expressions \eqref{eq:decompose_TXH} and \eqref{eq:Tcone_lowrank}.
\end{proof}

In fact, if $\lowrank\cap\hanifold$ is nonempty, then it is a smooth manifold. To see this, recall from Lemma~\ref{lem:normal_H} that an element in $\normal_X\hanifold$ has the form of $AX$ for some $A\in\mbR^{m\times m}$. If $AX$ also belongs to $\normal_X\lowrank$, then from~\eqref{eq:Ncone_lowrank} identifying $\normal_X\lowrank$, we can derive $AX=0$, and thus ${\normal_X\lowrank \cap \normal_X\hanifold=\hkh{0}}$, which implies $\lowrank$ and $\hanifold$ \revise{intersect transversally~\cite{lewis2008intersectionmanifolds}.} Therefore, ${\lowrank\cap\hanifold}$ is a smooth manifold with 
\[
\tangent_X\kh{\lowrank\cap\hanifold} = \tangent_X\lowrank \cap \tangent_X\hanifold.
\]
In view of this fact, Proposition~\ref{pro:manifold_s_cap_H} indeed clarifies the calculation rule for the tangent space of $\lowrank\cap\hanifold$, which plays a role in $\tangent_X\boundedrank \cap \tangent_X\hanifold$, as stated by the following lemma.
\begin{lemma}\label{lem:tanbounded_cap_tanH}
Given $X\in\boundedrank\cap\hanifold$ with $\rank\kh{X}=s$, it holds that
\begin{equation}\label{eq:TRcapTH}
    \tangent_X\boundedrank \cap \tangent_X\hanifold = \tangent_X\lowrank \cap \tangent_X\hanifold  + \hkh{N\in\normal_X\lowrank:\rank\kh{N}\le r-s}.
\end{equation}
\end{lemma}
\begin{proof}
    Let $\normal_X(r,s)=\{N\in\normal_X\lowrank:\rank\kh{N}\le r-s\}$. Incorporating formula~\eqref{eq:tangent_cone_boundedrank}, we have
    \begin{align*}
        \tangent_X\boundedrank \cap \tangent_X\hanifold 
         = & \kh{\tangent_X\lowrank + \normal_X(r,s)} \cap \tangent_X\hanifold
        \\
        = &\ \tangent_X\lowrank \cap \tangent_X\hanifold  + \normal_X(r,s),
    \end{align*}
    where the second equality holds since $\tangent_X\hanifold$ is a linear space and $\normal_X(r,s)\subseteq\normal_X\lowrank\subseteq\tangent_X\hanifold$, as implied by Corollary~\ref{cor:orth_tangentX} and expression~\eqref{eq:Ncone_lowrank}. 
\end{proof}

We now proceed to derive the tangent cone of the feasible set.
\begin{theorem}[Tangent cone]\label{the:tangent_intersection_rule}
Given $X\in\boundedrank\cap\hanifold$ with $\rank\kh{X}=s$ and the singular value decomposition $X=U\varSigma V^\top=HV^\top$ where $H=U\varSigma$, it holds that
\begin{equation}\label{eq:characterize_TRH}
    \tangent_X\kh{\boundedrank\cap\hanifold} = \left\{ \begin{array}{c}
	    KV^\top + UJV_\bot^\top + U_\bot RV_\bot^\top:\ K\in\tangent_H\hanifold^s
            \\
           J\in \mbR^{s\times(n-s)},\ R\in\mbR^{(m-s)\times(n-s)},\ \rank(R)\le r-s\\
        \end{array} \right\}. 
\end{equation}
Moreover, $\boundedrank\cap\hanifold$ is geometrically derivable at each of its points.
\end{theorem}
\begin{proof}
    Notice that $\tangent_X(\boundedrank\cap\hanifold)\subseteq\tangent_X\boundedrank \cap \tangent_X\hanifold$ holds true, directly following from~\eqref{eq:cone_oneside}. We then substitute \eqref{eq:manifold_s_cap_H} into \eqref{eq:TRcapTH} and verify that $\tangent_X\boundedrank \cap \tangent_X\hanifold$ equals the set on the right side of \eqref{eq:characterize_TRH}. Therefore, the ``$\subseteq$" part of \eqref{eq:characterize_TRH} is confirmed. 
    
    Conversely, given any $\eta$ belonging to the right side of \eqref{eq:characterize_TRH}, we aim to construct a smooth curve on $\boundedrank\cap\hanifold$ passing through $X$ with the tangent direction $\eta$. Specifically, given $\eta = KV^\top + UJV_\bot^\top + U_\bot R V_{\bot}^\top$, for some $K\in\tangent_{H}\hanifold^s,\,J\in\mbR^{s\times(n-s)}$, and $R\in\mbR^{(m-s)\times(n-s)}$ with $\rank(R)\le r-s$, which implies the matrix $R$ admits a decomposition as follows,
    \begin{equation}\label{eq:R_decompose}
        R=H^{}_RV_{R}^{\top}\ \ \text{with}\ \ H_R\in\mbR^{(m-s)\times(r-s)}\ \ \text{and}\ \ V_R\in\stiefel(n-s,r-s).
    \end{equation}
    Expression~\eqref{eq:Tcone_lowrank} indicates that $\eta_1:=KV^\top + UJV_\bot^\top$ is a tangent vector to the analytic manifold $\lowrank$ at $X$, and thus there exists an analytic curve $\gamma(t)$ on $\lowrank$ with $\gamma(0)=X$ and $\gamma^\prime(0)=\eta_1$. Subsequently, \cite[Theorem 1]{bunse1991analyticSVD} reveals that $\gamma(t)$ has an analytic singular value decomposition, i.e.,
    \[
        \gamma(t)=\zkh{U_\gamma(t)\ U_{\gamma\bot}(t)}\left[ \begin{matrix}
    \varSigma _{\gamma}\left( t \right)&		0\\
    0&		\varSigma _{\gamma\bot}\left( t \right)\\
        \end{matrix} \right] \zkh{V_\gamma(t)\ V_{\gamma \bot}(t)}^\top.    
    \]
    \revise{Without loss of generality, suppose $U_\gamma(0) = U,\ \varSigma_\gamma(0)=\varSigma,\ \text{and}\ V_\gamma(0)=V$. Since $\rank(\gamma(t))\equiv s$ and $\rank(\varSigma_{\revise{\gamma}}(0))=s$, we can find an $\varepsilon>0$ and the interval $\kh{-\varepsilon,+\varepsilon}$ such that $\varSigma_{\revise{{\gamma \bot}}}(t)\equiv 0$ for all $t\in(-\varepsilon,+\varepsilon)$, , which means $\gamma(t)=U_\gamma(t)\varSigma_\gamma(t)V^\top_\gamma(t)$.}
    
    \revise{Let $\rho_\gamma(t)=U_\gamma(t)\varSigma_\gamma(t)$ and then we aim at showing that $\rho^\prime_\gamma(0)\in\tangent_H\hanifold^s$. To this end, consider the differential of $\gamma(t)$ at $t=0$,
    \begin{equation}\label{eq:eta_1rho_gamma}
        \eta_1 = \gamma^\prime(0) \overset{(i)}{=}  \rho_\gamma^\prime(0)V^\top + U\varSigma {V_\gamma^{\prime \top} (0)} \overset{(ii)}{=}  \rho_\gamma^\prime(0)V^\top + U\varSigma \Omega V^\top + BV^\top_\bot,
    \end{equation}    
    for some $\Omega\in\sksym(s)$ and $B\in\mbR^{m\times(n-s)}$, where ($i$) comes from the differentiation rule on $\gamma(t)=\rho_\gamma(t)V^\top_\gamma(t)$, and ($ii$) holds from $V_\gamma(t)\in\stiefel(n,s)$ and $V_\gamma^\prime(0)\in\tangent_{V}\stiefel(n,s)=\{-V\tilde{\Omega}+V_\bot \tilde{B}^\top: \tilde{\Omega}\in\sksym(s),\tilde{B}\in\mbR^{s\times(n-s)}\}$. On the one hand, expression~\eqref{eq:mT_X} reveals that $U\varSigma \Omega V^\top+BV^\top_\bot\in\tangent_X\hanifold$. On the other hand, expression~\eqref{eq:decompose_TXH1} reveals that $\eta_1=KV^\top + UJV^\top_\bot$ with $K\in\tangent_{H}\hanifold^s$ also belongs to $\tangent_{X}\hanifold$. By the linearity of the tangent space and the computation~\eqref{eq:eta_1rho_gamma}, we have $\rho_\gamma^\prime(0)V^\top = \eta_1-(U\varSigma \Omega V^\top+BV^\top_\bot)\in\tangent_X\hanifold$, and thus $\rho^\prime_\gamma(0)\in\tangent_H\hanifold^s$ according to expression~\eqref{eq:decompose_TXH1}.}

    Moreover, note that for the embedded manifold $\hanifold^s$, the projection $\projection_{\hanifold^s}$ is well-defined and smooth in a neighborhood of $H$ and its differential at $H$ coincides with $\projection_{\tangent_H\hanifold^s}$  \cite[Lemma~4]{absil2012projectionlike}. Therefore, defining $\theta_\gamma \left( t \right) :=\projection_{\hanifold^s}(\rho_\gamma(t))\in\mbR^{m\times s}$, we have
    \begin{equation}\label{eq:theta1}
         \eta_1 = \theta_\gamma^\prime(0)V^\top + U\varSigma {V_\gamma^{\prime \top} (0)}.
    \end{equation}
    We then view $\theta_\gamma(t)$ as a curve on $\hanifold^r$ by considering the natural embedding $[\theta_\gamma(t)\ 0^{m\times(r-s)}]$, which concludes $[\theta^\prime_\gamma(0)\ 0^{m\times(r-s)}]\in\tangent_{[H\,0]}\hanifold^r$. \revise{Additionally,
    according to the discussion in section~\ref{sec:inheritance}, a parallel statement of Corollary~\ref{cor:orth_tangentX} also holds for the induced $\hanifold^r=\{X\in\mbR^{m\times r}:h^r(X)=0\}$. That is, by checking that $[H\,0^{m\times (r-s)}]\cdot [0^{m\times s}\ U_\bot H_R]^\top=0$, we have $[0^{m\times s}\ U_\bot H_R]\in\tangent_{[H\,0]}\hanifold^r$, and thus the linear combination $[\theta^\prime_\gamma(0)\ \,U_\bot H_R]\in\tangent_{[H\,0]}\hanifold^r$.} Consequently, there exists a smooth curve $\beta(t):=\zkh{\beta_1(t)\ \beta_2(t)}\in\hanifold^r$ such that
    \begin{equation}\label{eq:curvebeta}
    \zkh{\beta_1(0)\ \beta_2(0)}=\zkh{H\ 0^{m\times (r-s)}}\ \text{and}\  \zkh{\beta^\prime_1(0)\ \beta^\prime_2(0)}=\zkh{\theta^\prime_\gamma(0)\ \,U_\bot H_R}.
    \end{equation}
    Multiplying $\beta(t)$ by the transpose of $\zkh{V_\gamma(t)\ V_{\gamma\bot}(t)V_R}\in\stiefel(n,r)$, we obtain 
    \begin{equation}\label{eq:curvealpha}
         \alpha(t):=\beta(t)\zkh{V_\gamma(t)\ V_{\gamma\bot}(t)V_R}^\top=\beta_1(t)V^\top_\gamma(t) + \beta_2(t)V_R^\top V_{\gamma\bot}^\top(t),
    \end{equation}
    which is a smooth curve on $\boundedrank\cap\hanifold$ by Lemma~\ref{pro:X_H_fullrankdiff}. Note that $\alpha(0)=X$ and consider the differential of $\alpha(t)$ at $t=0$,
    \begin{align*}
        \alpha^\prime(0) \overset{\phantom{(1)}}{=}&\ \beta_1^\prime(0)V_\gamma^\top(0)+\beta_1(0)V_\gamma^{\prime\top}(0) +\beta_2^\prime (0)V_R^\top V_{\gamma\bot}^\top(0) + \beta_2 (0)V_R^\top V_{\gamma\bot}^{\prime\top}(0)
        \\
        \overset{(i)}{=}&\ \theta^\prime_\gamma(0)V^\top + U\varSigma {V_\gamma^{\prime \top} (0)} + U^{}_\bot H^{}_RV^\top_R V_\bot^\top + 0\cdot V_R^\top V_{\gamma\bot}^{\prime\top}(0)
        \\
        \overset{(ii)}{=}&\ \eta_1 + U^{}_\bot R V_\bot^\top
        \\
        \overset{\phantom{(1)}}{=}&\ \eta,
    \end{align*}
    where substituting $\zkh{V_\gamma(0)\ V_{\gamma\bot}(0)}=\zkh{V\ V_\bot}$ and initial values \eqref{eq:curvebeta} for $\beta(t)$ yields~$(i)$, and $(ii)$~follows from equations \eqref{eq:R_decompose} and \eqref{eq:theta1}. 
    
    Finally, we conclude that given any $\eta$ belonging to the right side of \eqref{eq:characterize_TRH}, there exists a smooth curve as~\eqref{eq:curvealpha} on $\boundedrank\cap\hanifold$ passing through $X$ with the tangent direction $\eta$. Therefore, the ``$\supseteq$" part in \eqref{eq:characterize_TRH} holds true, and $\boundedrank\cap\hanifold$ is geometrically derivable at~$X$. 
\end{proof}

Although the proof of Theorem~\ref{the:tangent_intersection_rule} is technical, it gives a clue that the orthogonal invariance of $h$ influences the space with respect to $V^{\top}$ in the tangent cone. By comparing expressions \eqref{eq:tangent_cone_boundedrank} and \eqref{eq:characterize_TRH}, we observe that the difference between the tangent cones to $\boundedrank$ and $\boundedrank\cap\hanifold$ is reflected in the first term~of 
\[
\eta = KV^\top + UJV_\bot^\top + U_\bot RV_\bot^\top,
\]
where $K\in\mbR^{m\times s}$ corresponds to $\tangent_X\boundedrank$  while $K$ is restricted to $\tangent_H\hanifold^s$ for $\tangent_X(\boundedrank\cap\hanifold)$. 

The expression~\eqref{eq:characterize_TRH} provides a closed-form characterization for the tangent cone $\tangent_X(\boundedrank\cap\hanifold)$, which recovers the tangent cone in~\cite[Theorem 6.1]{cason2013iterative} when $\hanifold=\mathrm{S}_{\mathrm{F}}(m,n)$. In addition, we can compute the projection onto the tangent cone as follows,
\begin{equation}\label{eq:projection_TRH}
    \projection_{\tangent_X(\boundedrank\cap\hanifold)}(E) = \kh{\projection_{\tangent_H\hanifold^s}\kh{EV}}V^\top + P_UEP_{V_\bot} + \projection_{\mbR^{m\times n}_{r-s}} \kh{P_{U_\bot}EP_{V_\bot}}.
\end{equation}

In essence, Theorem~\ref{the:tangent_intersection_rule} points to the following calculation rules for tangent and normal cones.

\begin{corollary}[Intersection rule]
Given $X\in\boundedrank\cap\hanifold$. The tangent cone to $\boundedrank\cap\hanifold$ equals to the intersection of tangent cones to $\boundedrank$ and $\hanifold$, i.e.,
\begin{align}
\tangent_X\kh{\boundedrank\cap\hanifold} = \tangent_X\boundedrank \cap \tangent_X\hanifold.     \label{eq:tangent_intersection_rule}
\end{align}
The normal cone to $\boundedrank\cap\hanifold$ decomposes as the direct sum of two orthogonal spaces, i.e.,
\begin{align}
    \normal_X\kh{\boundedrank\cap\hanifold} = \normal_X^\frechet\boundedrank \oplus \normal_X\hanifold.    \label{eq:normal_intersection_rule}
\end{align}
\end{corollary}
\begin{proof}
    The intersection rule~\eqref{eq:tangent_intersection_rule} comes from the proof of Theorem~\ref{the:tangent_intersection_rule}. To see the orthogonal decomposition \eqref{eq:normal_intersection_rule}, let $X\in\lowrank$ admit the singular value decomposition $X=U\varSigma V^\top=HV^\top$ where $H=U\varSigma$, and taking the polar operation on both sides on \eqref{eq:characterize_TRH}, we obtain
    \begin{align}
        \normal_X\kh{\boundedrank\cap\hanifold} = \begin{cases}
        \normal_X\lowrank \oplus \kh{\normal_H\hanifold^s}V^\top,\ \ \text{if}\ s=r,
        \\
        \kh{\normal_H\hanifold^s}V^\top,\quad\quad\quad\quad\quad\ \  \ \text{if}\ s<r.
        \end{cases} \label{eq:N_XRcapH_bypro}
    \end{align}
    where the ``$\oplus$" stems from $\normal_X\lowrank = \{U_\bot BV^\top_\bot:\ B\in\mbR^{\kh{m-s}\times\kh{n-s}}\}$. Incorporating \eqref{eq:decompose_NXH} and \eqref{eq:normal_cone_boundedrank} into \eqref{eq:N_XRcapH_bypro} gives~\eqref{eq:normal_intersection_rule}.
\end{proof}

Note that the normal cone to $\boundedrank\cap\hanifold$ is indeed a linear space, as revealed by~\eqref{eq:normal_intersection_rule}, and the projection onto it boils down to projections onto the two subspaces.

\subsection{First-order optimality} 
We investigate the first-order optimality conditions for~\eqref{eq:lowrank_orthinvar}, enriching the study of bounded-rank optimization and laying a foundation for the analysis in section~\ref{sec:RiemannianOpt}. 
\begin{definition}\label{def:crictical}
    A point $X\in\boundedrank\cap\hanifold$ is called \emph{stationary} for problem \eqref{eq:lowrank_orthinvar} if $\innerp{\nabla f(X),\eta}\ge 0$ for all $\eta\in\tangent_X(\boundedrank\cap\hanifold)$, i.e., $-\nabla f(X)\in\normal_X(\boundedrank\cap\hanifold)$, or equivalently, the projected negative gradient vanishes, i.e., ${\projection_{\tangent_X(\boundedrank\cap\hanifold)}(-\nabla f(X))}=0$.
\end{definition}
The stationary condition is necessary for $X$ to be locally optimal according to \cite[Theorem 6.12]{rockafellar2009variationalanalysis}. In practice, we can evaluate the stationarity measure by incorporating $E=-\nabla f(X)$ into~\eqref{eq:projection_TRH}  once the geometry of $\hanifold$ is known, which involves basic matrix operations. More specifically, when $r \ll \min\{m, n\}$, taking $P_{U_\bot}=I - UU^\top$ and $P_{V_\bot}=I - VV^\top$ is apt to computational efficiency. Furthermore, when the point of interest is not of rank~$r$, by taking advantage of the structure of $\boundedrank\cap\hanifold$, one can consider the following simplified stationarity measure.
\begin{proposition}\label{pro:optcondition}
Given $X\in\boundedrank\cap\hanifold$ with $\rank\kh{X}=s<r$ and the singular value decomposition $X=U\varSigma V^\top=HV^\top$ where $H=U\varSigma$. If $X$ is a stationary point of problem \eqref{eq:lowrank_orthinvar}, it holds that $-\nabla f(X) \in\normal_{X}\hanifold$.
\end{proposition}
\begin{proof}
Recalling Definition~\ref{def:crictical}, \eqref{eq:normal_cone_boundedrank} and~\eqref{eq:normal_intersection_rule} yields the result.
\end{proof}

It is worth noting that Proposition~\ref{pro:optcondition} gives an insight into the landscape of~\eqref{eq:lowrank_orthinvar}: the optimality is only attributed to the geometry of $\hanifold$ when a stationary point is of rank $s<r$. In addition, we observe that when $h$ reduces to the zero mapping, correspondingly $\hanifold=\mbR^{m\times n}$ and $\normal_X\hanifold=\{0\}$, Proposition~\ref{pro:optcondition} recovers the results in~\cite[Corollary 3.4]{schneider2015Lojaconvergence} for optimization on bounded-rank matrices.

\section{A space-decoupling parameterization}\label{sec:M_h}
The developed geometry in section~\ref{sec:OptCondition} enhances the theory of optimization on bounded-rank matrices. However, the nonsmooth structure of $\boundedrank\cap\hanifold$ and the unclear projection onto it still pose impediments to addressing problem \eqref{eq:lowrank_orthinvar}. To this end, borrowing the idea from the desingularization technique~\cite{khrulkov2018desingularization,rebjock2024boundedrank},  we propose to parametrize the feasible set by a smooth manifold. Specifically, we consider the following \emph{space-decoupling parameterization},
\begin{equation}\label{eq:spacedecoupling_para}
    \manifold_{h}=\left\{(X, G) \in \mathbb{R}^{m \times n} \times \grassmann(n, n-r): X G=0,\ h(X)=0\right\},
\end{equation}
where the smooth mapping $\phi:\mbR^{m\times n}\times\sym(n)\to\mbR^{m\times n}:(X,G)\mapsto X$ satisfies $\phi(\manifold_h)=\boundedrank\cap\hanifold$, and the Grassmann manifold $\grassmann(n,s)=\{G\in\sym(n): G^2=G,\,\rank(G)=s\}$ is an embedded submanifold in $\sym(n)$ \cite{bendokat2024grassmann}. The parameterization $\manifold_h$ decouples the feasible region of~\eqref{eq:lowrank_orthinvar} defined by two constraints into two spaces: the rank information $\rank(X)\le r$ and the orthogonally invariant constraint $h(X)=0$ are encoded in $\grassmann(n,n-r)$ and $\mbR^{m\times n}$, respectively. 

\revise{We then investigate the geometry of the proposed $\manifold_h$, and the results developed in sections~\ref{sec:embedgeoMh},~\ref{sec:RiegeoMh},~\ref{sec:riederivMh}, and~\ref{sec:retracMh} extend the analysis in~\cite{khrulkov2018desingularization,rebjock2024boundedrank}, which studied the desingularization 
\begin{equation*}
    \desing:=\{(X, G) \in \mathbb{R}^{m \times n} \times \grassmann(n, n-r): X G=0\}.
\end{equation*} 
It is worth noting that $\manifold_h$ can be regarded as an extension of $\desing$ to accommodate the additional constraint $h(X)=0$, and conversely, $\desing$ serves as a special case of $\manifold_h$ when $h\equiv 0$.}

\subsection{Embedded geometry}\label{sec:embedgeoMh}
In this subsection, we prove that $\manifold_{h}$ is an embedded submanifold in $\mathbb{R}^{m \times n} \times \grassmann(n, n-r)$, and then characterize its tangent space. The rationale behind the proof is illustrated in~\myfig\ref{fig:proof_embedded}.

\begin{figure}[h]
    \newcommand{\ratio}{0.5}
    \newcommand{\spacegap}{1.7}
    \newcommand{\customrectangle}[7]{
            \draw[thick, dashed, draw=#6, fill=#7] 
                (#1-#3/2, #2-#4/2) rectangle 
                (#1+#3/2, #2+#4/2); 
            \node at (#1, #2) {#5};  
        }
    \begin{minipage}{1\textwidth}
    \begin{center}
    \begin{tikzpicture}
        \newcommand{\nodegap}{1.3}
        \tikzset{
        node distance=\nodegap cm,
        post/.style={->,shorten >=2pt,shorten <=2pt,>={Stealth[round]},thick},
        space/.style={
          draw=none, 
          fill=none, 
          inner sep=0pt,
          minimum size=6mm
        },
        rhookarrow/.style={{Hooks[right]}->},
        lhookarrow/.style={{Hooks[left]}->},
        }
        
        \node[space] (Sh) {\normalsize $\sanifold_{h}(n-r)$};;
        \node[space] (Mh) [below=\spacegap*0.6 cm of Sh] {\normalsize $\sanifold_{h}(n-r)/\orth(n-r)$};
        \node[space] (Mh1) [left=0.1 cm of Mh] {\normalsize $\manifold_h\cong$};
        \node[space] (embedding) [right=\spacegap*1.4 cm of Mh] {\normalsize $\mbR^{m\times n} \times \grassmann(n,n-r)$};
        \node[space] (Rst)  at (embedding|-Sh) {\normalsize $\mbR^{m\times n}\times \stiefel(n,n-r)$};

        \draw[post] (Sh) to node[midway, left] {\normalsize $\pi$} (Mh);
        \draw[post] (Rst) to node[midway, right] {\normalsize $\pi$} (embedding);
        \draw[->,rhookarrow,shorten <=3pt,shorten >=4pt, >={Stealth[round]},thick] (Mh) -- (embedding) node[pos=0.5, sloped, above,shift={(0,0.1)}] {Theorem~\ref{the:manifold_h}};
        \draw[->,rhookarrow, >={Stealth[round]},thick,shorten <=3pt,shorten >=4pt,] (Sh) -- (Rst)  node[pos=0.5, sloped, above,shift={(0,0.1)}] {Lemma~\ref{lem:sOb}};

        \coordinate (LemmaS) at ($(Sh)!0.5!(Rst)$);
        \coordinate (PropoM) at ($(Mh)!0.5!(embedding)$);

        \node (lem) at ([xshift=0cm, yshift=-0cm]LemmaS) {};
    
    \end{tikzpicture}
    \end{center}
    \end{minipage}
    \caption{Relationship among the manifolds and the embedded spaces, where $\pi$ is the quotient mapping induced by the group action of $\orth(n-r)$.}
    \label{fig:proof_embedded}
\end{figure}

\begin{lemma}\label{lem:sOb}
    For $s=0,1,\ldots,n$, if the set $ \sanifold_{h}(s) :=\{(X, W) \in \mathbb{R}^{m \times n} \times \stiefel(n, s): X W=0,\ h(X)=0\}$ is nonempty, it is a smooth embedded submanifold in $\mathbb{R}^{m \times n} \times \stiefel(n, s)$ of dimension $m(n-s) + ns - \dim(\sym(s)) - q$.
\end{lemma}
\begin{proof}
    Consider the mapping 
    \[F:\, \mathbb{R}^{m \times n} \times \stiefel(n, s) \rightarrow \mathbb{R}^{m \times s}\times \mbR^{q}:\,(X,W)\mapsto (XW, h(X)).
    \]
    Note that $\sanifold_{h}(s)=F^{-1}(0)$, and the differential $\diff F_{(X,W)}\zkh{\eta,\zeta} = (\eta W+X \zeta, \diff h_X[\eta])$, for $\eta\in\mbR^{m\times n}$ and $\zeta\in\tangent_W\stiefel(n,s)=\{W\Omega+W_\bot B:\,\Omega\in\sksym(s), B\in\mbR^{(n-s)\times s}\}$. We then show that $\diff F$ is full-rank on~$\sanifold_{h}(s)$. 
    
    To this end, given a tangent vector $(A,a)\in \mathbb{R}^{m \times s}\times \mbR^{q}$, one can find an $\tilde{\eta}\in\mbR^{m\times n}$ such that $\diff h_X\zkh{\tilde{\eta}}=a$ since $\diff h_X$ is full-rank at $X\in\hanifold$. Subsequently, by the orthogonality of $W$, choosing $\eta =\tilde{\eta}+( A-\tilde{\eta}W ) W^{\top}$ and $\zeta =0$ yields ${\eta} W+X\zeta =A$. According to Corollary~\ref{cor:orth_tangentX} and $X( ( A-\tilde{\eta}W ) W^{\top} ) ^{\top}=0$, it holds that $(A-\tilde{\eta}W ) W^{\top} \in \kernel(\diff h_X)$, and therefore $\diff h_X[\eta ] = \diff h_X [ \tilde{\eta}+( A-\tilde{\eta}W) W^{\top} ] = \diff h_X[ \tilde{\eta} ] =a$. Consequently, we obtain $\diff F_{(X, W)}\zkh{\eta, \zeta} = \kh{A,a}$, which implies the differential $\diff F_{(X,W)}$ is surjective. By the arbitrariness of $(X,W)\in\sanifold_{h}(s)$, we conclude that $\diff F$ is full-rank on~$\sanifold_{h}(s)=F^{-1}(0)$. It follows from \cite[Corollary 5.14]{lee2012manifolds} that $\sanifold_{h}(s)$ is an embedded submanifold of dimension $m n+\dim(\stiefel(n, s))-(ms+q)$.
\end{proof}

Notably, the orthogonal invariance of $h$ plays a crucial role in the above proof, which allows the constructed mapping $F$ to inherit the full-rankness from $h$, thereby ensuring $\sanifold_{h}(s)$ is a smooth manifold.

\begin{theorem}\label{the:manifold_h}
   The set $\manifold_{h}$ is an embedded submanifold in $\mathbb{R}^{m \times n} \times \grassmann(n, n-r)$ of dimension $(m+n-r)r-q$.
\end{theorem}
\begin{proof}
    Taking $s=n-r$ in Lemma~\ref{lem:sOb}, we have $\sanifold_{h}(n-r)$ as a smooth embedded submanifold in $\mathbb{R}^{m \times n} \times \stiefel(n, n-r)$ of dimension~$\kh{mr+n(n-r)-\dime\sym(n-r)-q}$. Define the group action $\orth(n-r)$ on $\sanifold_{h}(n-r)$ as $Q\cdot(X, W):=(X, W Q)$, which is smooth, free (from the orthogonality of $W\in\stiefel(n,n-r)$), and proper (from the compactness of $\orth(n-r)$). It is deduced from \cite[Theorem 21.10]{lee2012manifolds} that the quotient set $\sanifold_{h}(n-r) / \orth\kh{n-r}$ is a quotient manifold of dimension $(m+n-r)r-q$. Consider the smooth immersion $\iota: \sanifold_{h}(n-r) / \orth\kh{n-r} \rightarrow \wanifold=\mathbb{R}^{m \times n} \times \grassmann(n, n-r):\,(X, W)\mapsto(X, W W^{\top})$. In fact, $\iota$ is a homeomorphism onto $\iota\kh{\sanifold_{h}(n-r) / \orth(n-r)}=\mathcal{M}_{h}$, implying that it is a smooth embedding. Therefore, by \cite[Proposition 5.2]{lee2012manifolds}, $\mathcal{M}_{h}$ turns out to be an embedded submanifold in $\wanifold$ of dimension $(m+n-r)r-q$. 
\end{proof}

Based on Theorem~\ref{the:manifold_h}, viewing $\grassmann(n,n-r)$ as an embedded submanifold in $\sym(n)$ reveals that $\manifold_{h}$ is indeed an embedded submanifold in the Euclidean space $\mbR^{m\times n}\times\sym(n)$. 

\begin{lemma}\label{lem:repreMob}
    Given $(X, G) \in \manifold_{h}$, there exist $H \in \hanifold^r$ and $V \in \stiefel(n, r)$ such that $X= HV^{\top}$, $G=I-V V^{\top}$.
\end{lemma}
\begin{proof}
    The projection matrix $G\in\grassmann(n,n-r)$ can be expressed by $W \in \operatorname{St}(n, r)$ in the sense that $G=I-W W^{\top}$, and thus the orthogonality $XG=0$ implies $X=X W W^{\top}$. Substituting the decomposition $X W= H Q^\top$ with $H\in\mbR^{m\times r}$ and $Q\in\orth(r)$, we obtain $X= H V^{\top}$ and $G=I-V V^{\top}$, where $V=W Q$. Furthermore, Lemma~\ref{pro:X_H_fullrankdiff} yields $H\in \hanifold^r$ since $X=HV^\top\in\hanifold$.
\end{proof}

Lemma~\ref{lem:repreMob} demonstrates that for any point $(X,G)$ in $\manifold_h$, we can find a \emph{representation} $(H,V)$ from $\mbR^{m\times r}\times \stiefel\kh{n,r}$ which satisfies $(X,G)=(HV^\top, I-VV^\top)$. In the next proposition, we identify the tangent space to $\manifold_h$. 

\begin{proposition}[Tangent space]\label{pro:Mh_tangent_space}
    Given $(X, G) \in \manifold_h$ with the representation $(H, V)$, the tangent space to $\manifold_h$ at $(X, G)$ is expressed as follows,
    \begin{equation}\label{eq:Mh_tangent_space}
        \tangent_{(X, G)} \manifold_h=\big\{ (K V^{\top}\!\!+HV_p^{\top},-V_p V^{\top}\!\!-V V_p^{\top})\!: K\in\tangent_H\hanifold^r,\,  V_p\in\mbR^{n\times r},\,V^{\top} V_p=0\big\}.
    \end{equation}
\end{proposition}
\begin{proof}
    As in Proposition~\ref{pro:manifoldHs}, the set $\hanifold^r$ is a manifold of dimension $(mr-q)$. Therefore, given any pair of tangent vectors $(K,V_p)\in\tangent_H \hanifold^r\times \tangent_V\stiefel(n,r)$ with $V^\top V_p=0$, there exist smooth curves $H(t)\in\hanifold^r$ and $V(t)\in\stiefel(n,r)$ such that $H(0)=H,\ H^\prime(0) = K$ and $V(0)=V,\ V^\prime(0)=V_p$. Assembling these two curves produces $( X( t ) ,G( t ) ) :=( H( t ) V^{\top}( t ) ,I-V( t ) V^{\top}( t ) )$ on $\manifold_h$, which satisfies $( X^{\prime}( 0 ) ,P^{\prime}( 0 ) ) =( KV^{\top}+HV_p,-V_pV^{\top}-VV_{p}^{\top} )$. This confirms the ``$\supseteq$" part of~\eqref{eq:Mh_tangent_space}. Moreover, note that the dimension of the linear space on the right side of \eqref{eq:Mh_tangent_space} equals to $\dime\kh{\hanifold^r}+\kh{n-r}r=\kh{m+n-r}r-q$, and it coincides with the dimension of~$\manifold_h$ obtained in Theorem~\ref{the:manifold_h}, which leads to the conclusion.
\end{proof}

The expression~\eqref{eq:Mh_tangent_space} sheds light on representing the tangent space to $\manifold_h$ by lower-dimensional matrices. Specifically, given any $(\eta,\zeta)\in\tangent_{(X,G)}\manifold_h$, there exists $(K,V_p)\in\tangent_H\hanifold^r\times \mathbb{R}^{n \times r}$ with $V^\top V_p=0$ representing it in the sense that 
\begin{equation}\label{eq:represent_eta}
    (\eta,\zeta) = (K V^{\top}+HV_p^{\top},-V_p V^{\top}-V V_p^{\top}),
\end{equation}
which results in a space complexity of $(m+n)r$.

\medskip

\revise{We remark that by resorting to the results from~\cite{khrulkov2018desingularization,rebjock2024boundedrank}, the set $\manifold_h$ can in fact be characterized as a transverse intersection of two smooth manifolds in $\mathbb{R}^{m\times n}\times\sym(n)$. This perspective essentially points to the main results, Theorem~\ref{the:manifold_h} and Proposition~\ref{pro:Mh_tangent_space}. Nevertheless, to keep the exposition self-contained, we implement a direct analysis on $\manifold_h$ in this work.} 

\revise{In the end, for completeness and better understanding of the geometry of $\manifold_h$, we specify the transverse intersection to conclude the discussion. Taking $h$ as the zero mapping, the developed results of $\manifold_h$ reduce to those of $\desing$ in~\cite{rebjock2024boundedrank}, including the formula for the tangent space:
\begin{equation*}
    \tangent_{(X, G)} \desing=\big\{ (K V^{\top}\!\!+HV_p^{\top},-V_p V^{\top}\!\!-V V_p^{\top})\!: K\in\mbR^{m\times r},\,  V_p\in\mbR^{n\times r},\,V^{\top} V_p=0\big\},
\end{equation*}
which can also be obtained by substituting $\hanifold^r=\mbR^{m\times r}$ into~\eqref{eq:Mh_tangent_space}. More generally, for a nontrivial $h:\mbRmn\to\mbR^q$, $\manifold_h$ can be treated as the intersection of $\desing$ and the product manifold $\hanifold\times\sym(n)$ in the ambient space $\mbRmn\times\sym(n)$, i.e.,
\begin{equation*}
    \manifold_h = \desing \cap (\hanifold\times\sym(n)) \subseteq \mbRmn\times \sym(n).
\end{equation*}
We then show that this is a transverse intersection. According to Proposition~\ref{pro:kerDh}, we have $\{BV_{\bot}^{\top}:B\in\mbR^{m\times (n-r)}\} \subseteq \tangent_X\hanifold$. Hence, for any $(E,Z)\in\tangent_{(X,G)}(\mathbb{R}^{m\times n}\times \sym(n))$, one can decompose
\begin{equation*}
    (E,Z)= ((EV)V^\top, 0)  + ((EV_{\bot})V_{\bot}^\top,Z)\in \tangent_{(X,G)}\desing + \tangent_{(X,G)} (\hanifold\times \sym(n)),
\end{equation*}
which indicates that $\desing$ intersects transversally with $\hanifold\times \sym(n)$. Hence, we can also conclude that $\manifold_h$ is a smooth manifold (i.e., Theorem~\ref{the:manifold_h}), and give the calculation rule for the tangent space to $\manifold_h$,
\begin{equation*}
    \tangent_{(X,G)}\manifold_h = \tangent_{(X,G)}\desing \cap \tangent_{(X,G)}(\hanifold\times\sym(n)),
\end{equation*}
which is equivalent to the characterization in Proposition~\ref{pro:Mh_tangent_space}.}

\subsection{Riemannian geometry}\label{sec:RiegeoMh}
As an embedded submanifold, $\manifold_h$ naturally inherits the Riemannian metric from the ambient Euclidean space $\eanifold:=\mbR^{m\times n}\times\sym(n)$, i.e.,
\begin{align}\label{eq:rie_geometry}
    \left< \cdot ,\cdot \right>_\omega:\ \eanifold\times\eanifold \longrightarrow \mathbb{R}:\quad \kh{(E_1,Z_1),(E_2,Z_2)}\longmapsto \innerp{E_1,E_2} + \omega\innerp{Z_1,Z_2}.
\end{align}
Note that the weight $\omega> 0$, and the subscript $\omega$ will be omitted if there is no ambiguity. 

Taking into account the representations---$(H,V)$, $(K_1,V_{p,1})$ and $(K_2,V_{p,2})$---for $(X,G)\in\manifold_h$ and $(\eta_1,\zeta_1),\,(\eta_2,\zeta_2)\in\tangent_{(X,G)}\manifold_h$, the metric can be computed by
\begin{equation}\label{eq:inner_product}
        \innerp{(\eta_1,\zeta_1),(\eta_2,\zeta_2)} = \innerp{K_1,K_2} + \innerp{V_{p,1},V_{p,2}M_{H,\omega}},
\end{equation}
where $M_{H,\omega}:=2\omega I+H^\top H$.

The following proposition gives the projection of $(E,Z)\in\eanifold$ onto the tangent space.

\begin{proposition}\label{pro:projection_to_TMh}
Given $(X,G)\in\manifold_h$ with the representation $(H,V)$, the projection of $(E,Z)\in\eanifold$ onto $\tangent_{(X, G)} \manifold_{h}$ can be represented by
\begin{equation}\label{eq:projection_eq}
    \bar{K}=\projection_{\tangent_H \hanifold^r}\kh{EV},\quad\quad \bar{V}_p= G(E^\top H -2\omega ZV)M_{H,\omega}^{-1}\,.
\end{equation}
\end{proposition}
\begin{proof}
    Computing the projection leads to the following optimization problem,
    \[
    \min_{K\in\tangent_H\hanifold^r,\,V^\top V_p=0}\,c(K,V_p):=\|E-(KV^\top+HV_p^\top)\|_\frob^2 + \omega \|Z+V_pV^\top + VV^\top_p\|_\frob^2.
    \]
    Rearranging the expression of the cost function $c$, we have
    \[
        c(K,V_p)=\Fnorm{E-KV^\top}^2 - 2\langle E,HV^\top_p\rangle + \Fnorm{HV^\top_p}^2 + \omega \Fnorm{Z+V_pV^\top + VV^\top_p}^2.
    \]
    It suffices to choose $\bar{K}=\projection_{\tangent_H\hanifold^r}(EV)$ to minimize the first term. Additionally, taking the partial derivative with respect to $V_p$ and letting it be orthogonal to the space $\{\tilde{V}_p\in\mbR^{n\times r}:\,{V}^\top \tilde{V}_p=0\}$, it holds that $G(-2E^\top H + 2V_pH^\top H + 4\omega (ZV + V_p)) = 0$, which reveals that the solution $\bar{V}_p=G(E^\top H-2\omega ZV)M^{-1}_{H,\omega}$.
\end{proof}

\section{Optimization on the manifold $\manifold_h$}\label{sec:RiemannianOpt} 
By using the parameterization $(\manifold_h,\phi)$ with $\phi(\manifold_h)={\boundedrank\cap\hanifold}$, we consider the minimization of $\bar{f}:=f\circ\phi$ on the manifold~$\manifold_h$, and thus reformulates the nonsmooth constrained optimization problem \eqref{eq:lowrank_orthinvar} as the smooth Riemannian optimization problem~\eqref{eq:G}
\begin{align*}
    \min_{(X,G) \in \manifold_h} \bar{f}(X,G):=f\circ \phi(X,G);
\end{align*}
see~\myfig\ref{fig:diagram} for illustration. Specifically, the two problems share the same optimal value, while the formulation \eqref{eq:G} offers a smooth remedy for \eqref{eq:lowrank_orthinvar}, allowing us to draw upon existing theoretical and algorithmic techniques on Riemannian optimization. 

Riemannian optimization aims to solve problems on smooth manifolds by exploiting the Riemannian geometry of manifolds; see~\cite{absil2008optimization,boumal2023introduction} for an overview. To clarify the discussion, given a curve $\gamma(t)$ on a manifold, we adopt $\ddot{\gamma}(t)$ to denote extrinsic acceleration on the embedding Euclidean space, and $\gamma^{\prime\prime}(t)$ to denote intrinsic acceleration on the embedded Riemannian manifold \cite{boumal2023introduction}. For problem \eqref{eq:G}, a point $(X,G)\in\manifold_h$ is called \emph{first-order stationary} if $\diff \bar{f}_{(X,G)}[\eta,\zeta]= 0$ for all $(\eta,\zeta)\in\tangent_{(X,G)}\manifold_h$; and called \emph{second-order stationary} if it additionally satisfies ${(\bar{f}\circ\gamma)^{\prime\prime}(0)\geq 0}$ for all curves $\gamma(t)$ on $\manifold_h$ with $\gamma(0)=(X,G)$.

In this section, we provide computations for the Riemannian derivatives, retractions, and vector transports, which serve as a cornerstone for implementing Riemannian optimization algorithms on $\manifold_h$. Subsequently, we unveil the relationship between stationary points of~\eqref{eq:lowrank_orthinvar} and \eqref{eq:G}, and give the convergence properties\revise{, drawing on and extending the analysis in~\cite{levin2023remedy,levin2024effectlift}.}

\subsection{Riemannian derivatives on $\manifold_h$}\label{sec:riederivMh}
The developed geometry of $\manifold_h$ enables computing Riemannian gradients and Hessians, which assists in choosing appropriate directions in an algorithm. For clarity, we retain the symbol $\nabla$ to denote the Euclidean derivative. When referring to the Riemannian counterparts, we adopt an additional subscript, e.g., $\nabla_{\manifold_h}$ and $\nabla_\hanifold$ for derivatives on $\manifold_h$ and $\hanifold$, respectively. Furthermore, projection onto the tangent space $\tangent_{(X,G)}\manifold_h$ is abbreviated as $\projection_{(X,G)}$.
\begin{proposition}[Riemannian gradient]\label{pro:1st_Riegrad}
    Given $(X,G)\in\manifold_h$ with the representation $(H,V)$, the Riemannian gradient of $\bar{f}$, $\nabla_{\manifold_h} \bar{f}(X,G)$, can be represented by
\begin{equation}\label{eq:1st_Riegrad}
    \bar{K}=\projection_{\tangent_H \hanifold^r}\kh{\nabla f(X) V},\quad\quad \bar{V}_p= G\nabla f(X) ^\top H M_{H,\omega}^{-1}\,.
\end{equation}
\end{proposition}
\begin{proof}
    It suffices to take $(E,Z)=\nabla \bar{f}(X,G)=(\nabla f(X), 0)$ in Proposition~\ref{pro:projection_to_TMh}.
\end{proof}

An ensuing product is the first-order optimality condition of problem~\eqref{eq:G}.
\begin{corollary}\label{cor:1st_Mh}
    Given $(X,G)\in\manifold_h$ with the representation $(H,V)$, then it is a first-order stationary point if and only if ${\nabla f(X) V}\in\normal_H\hanifold^r$ and $G\nabla f(X) ^\top H = 0$.
\end{corollary}
\begin{proof}
    Letting $\bar{K}=0$ and $\bar{V}_p=0$ in Proposition~\ref{pro:1st_Riegrad} obtains the result.
\end{proof}

Proposition~\ref{pro:1st_Riegrad} obtains a closed-form expression for $\nabla_{\manifold_h}\bar{f}$. In order to exploit second-order information, we delve into the calculation of the Riemannian Hessian $\nabla^2_{\manifold_h}\bar{f}$, which appears more involved. Specifically, given $(\eta,\zeta)\in\tangent_{(X,G)}\manifold_h$, it is formulated in \cite[\S5.11]{boumal2023introduction} that $\nabla^2_{\manifold_h} \bar{f}(X,G)[\eta,\zeta]$ can be decomposed into two terms:
\begin{equation}\label{eq:RieHess_2terms}
    \nabla^2_{\manifold_h} \bar{f}(X,G)[\eta,\zeta] = \projection_{(X,G)}\kh{\mathfrak{D}(\nabla \bar{f}(X,G))} + \projection_{(X,G)}\kh{\nabla^2 \bar{f}(X,G)[\eta,\zeta]},
\end{equation}
where the associated mapping $\mathfrak{D}$ is defined as
\begin{equation*}
    \mathfrak{D}:\quad \eanifold\longrightarrow\eanifold:\quad (E,Z)\longmapsto \diff\kh{(\tilde{X},\tilde{G})\mapsto \projection_{(\tilde{X},\tilde{G})}(E,Z)}(X,G)[\eta,\zeta].
\end{equation*}
After the detailed calculations, we obtain the following proposition for computing the Riemannian Hessian.

\begin{proposition}[Riemannian Hessian]\label{pro:2rd_hessian}
    Given $(X,G)\in\manifold_h$ and $(\eta,\zeta)\in\tangent_{(X,G)}\manifold_h$ with representations $(H,V)$ and $(K,V_p)$, respectively, then $\nabla^2_{\manifold_h} \bar{f}(X,G)[\eta,\zeta]$ can be represented by
\begin{equation}\label{eq:hessian_expression}
    \begin{aligned}
        &\bar{K} = \nabla^2_\hanifold f(X)[\eta]V + \projection_{\tangent_H{\hanifold^r}}\kh{W_{H,\omega}\nabla f(X)V_p},
        \\
        &\bar{V}_p= G\kh{-V_p[\projection_{\normal_H\hanifold^r}(\nabla f(X)V)]^\top H+(\nabla^2 f(X)[\eta])^\top H + \nabla f(X)^\top W_{H,\omega} K}M^{-1}_{H,\omega},
    \end{aligned}
\end{equation}
where $W_{H,\omega}:=I-HM^{-1}_{H,\omega}H^\top$.
\end{proposition}
\begin{proof}
    See Appendix~\ref{app: Riemannian_hessian}.
\end{proof}

As revealed by~\eqref{eq:hessian_expression}, the Riemannian Hessian on $\manifold_h$ can be built from that on $\hanifold$ and the Euclidean derivatives, along with basic matrix operations. Consequently, the second-order optimality condition for problem \eqref{eq:G} is deduced.

\begin{corollary}
    Given $(X,G)\in\manifold_h$ and $(\eta,\zeta)\in\tangent_{(X,G)}\manifold_h$ with representations $(H,V)$ and $(K,V_p)$, respectively, then the second-order optimality condition is that for all $(\eta,\zeta)\in\tangent_{(X,G)}\manifold_h$
    \[
        \begin{cases}
            {\nabla f(X) V}\in\normal_H\hanifold^r,\ G\nabla f(X) ^\top H = 0,
            \\
             \langle K, \nabla^2_\hanifold f(X)[\eta]V\rangle+\langle H, \nabla^2 f(X)[\eta]V_p\rangle 
             \\
             \quad\quad\quad\quad\quad \quad\quad \quad\,-\ \langle V_p,V_pV^\top \nabla f(X)^\top H\rangle + 2\innerp{K, \nabla f(X)V_p} \ge 0.
        \end{cases}
    \]
\end{corollary}
\begin{proof}
    Represent $\nabla^2_{\manifold_h} \bar{f}(X,G)[\eta,\zeta]$ with $(\bar{K},\bar{V}_p)$ as in \eqref{eq:hessian_expression}. Substituting $(K,V_p)$ and $(\bar{K},\bar{V}_p)$ into~\eqref{eq:inner_product} yields
    \[
    \begin{aligned}
        \innerp{\!(\eta,\zeta),\nabla^2_{\manifold_h} \bar{f}(X,G)[\eta,\zeta]\!}\!=& \innerp{\!K,\nabla^2_\hanifold f(X)[\eta]V + \projection_{\tangent_H{\hanifold^r}}(W_{H,\omega}\nabla f(X)V_p)\!} 
        \\
        &-\innerp{V_p,V_p[\projection_{\normal_H\hanifold^r}(\nabla f(X)V)]^\top H}
        \\        
        &+ \innerp{V_p, {(\nabla^2 f(X)[\eta])^\top H + \nabla f(X)^\top W_{H,\omega}K}}
        \\
        =& \innerp{ K, \nabla^2_\hanifold f(X)[\eta]V}+\innerp{ H, \nabla^2 f(X)[\eta]V_p}
        \\
        &-\innerp{V_p,V_pV^\top \nabla f(X)^\top H} + 2\innerp{K,W_{H,\omega}\nabla f(X)V_p}.
        \\
        \revise{=}&\ \revise{\innerp{ K, \nabla^2_\hanifold f(X)[\eta]V}+\innerp{ H, \nabla^2 f(X)[\eta]V_p}}
        \\
        &\revise{-\innerp{V_p,V_pV^\top \nabla f(X)^\top H} + 2\innerp{K,\nabla f(X)V_p},}
    \end{aligned}
    \]
    \revise{where the last equality comes from $H^\top\nabla f(X)V_p=0$ implied by the first-order optimality in Corollary~\ref{cor:1st_Mh}, $GV_p=V_p$, $W_{H,\omega}=I-HM^{-1}_{H,\omega}H^\top$, and the computation
    \begin{equation*}
        \innerp{K,W_{H,\omega}\nabla f(X)V_p}\!=\!\innerp{K, \nabla f(X)V_p} - \Big\langle K, HM^{-1}_{H,\omega}\underbrace{H^\top\nabla f(X)GV_p}_{0} \Big\rangle\!=\!\innerp{K, \nabla f(X)V_p}.
        \vspace{-2mm}
    \end{equation*}
    Consequently, applying \cite[Proposition 6.3]{boumal2023introduction} concludes the second-order optimality condition.}
\end{proof}

\subsection{Retractions on $\manifold_h$}\label{sec:retracMh}
The next step is to investigate the {retraction} on $\manifold_h$, a geometric tool guiding the movement from the current point along a tangent vector. Specifically, a smooth mapping $\retrac: \tangent\manifold \rightarrow \manifold$, defined on the tangent bundle, is called a \emph{retraction} on the manifold $\manifold$ if for any $X\in\manifold,\xi\in\tangent_Y\manifold$, the curve $\gamma(t):=\retrac_X(t\xi)$ satisfies $\gamma(0)=X$ and $\gamma^\prime(0)=\xi$, where $\retrac_X$ denotes the restriction of $\retrac$ on $\tangent_X\manifold$.

\begin{definition}\label{def:orth_homo}
    Let $\manifold\subseteq\mbR^{m\times n}$ be a manifold with the property that if $X\in\manifold$, then $XQ\in\manifold$ for all $Q\in \orth(n)$. The mapping $\retrac:\,\tangent\manifold \rightarrow\manifold$ is \emph{orthogonally homogeneous} if $\retrac_X (\xi)Q = \retrac_{XQ} (\xi Q)$ for any $Q\in \orth(n)$ and $(X,\xi)\in\tangent\manifold$.
\end{definition}

The motivations for introducing this homogeneity are two-fold: 1) it generalizes the ``consistency requirement" for retractions on the Stiefel manifold (see \cite{boumal2015Grassmann}). Specifically, note that both the Cayley transform \cite{nishimori2005cayley} and the polar retraction \cite{absil2008optimization} are orthogonally homogeneous; 2) the projection-like retraction on $\hanifold$ is orthogonally homogeneous, that is, $[\projection_\hanifold(X+\eta)]Q=\projection_\hanifold(XQ+\eta Q)$. Extending these spirits, we demonstrate that homogeneous retractions on $\stiefel(n,r)$ and $\hanifold^r$ can produce a retraction on~$\manifold_h$.

\begin{proposition}[First-order retraction]\label{pro:1st_retrac}
   Suppose $\retrac^{\mathrm{St}}:\,\tangent\stiefel(n,r)\rightarrow \stiefel(n,r)$ is a retraction on $\stiefel(n,r)$ and $\retrac^{\hanifold}:\,\tangent\hanifold^r \rightarrow \hanifold^r$ is a retraction on $\hanifold^r$. If both $\retrac^{\mathrm{St}}$ and $\retrac^{\hanifold}$ are orthogonally homogeneous, then the assembled mapping $\retrac$ defined as follows is a retraction on $\manifold_h$
    \begin{equation}\label{eq:1st_retrac}
        \retrac_{(X,G)}(\eta,\zeta) := \kh{\retrac^{\hanifold}_H(K)(\retrac^{\mathrm{St}}_V(V_p))^\top, I-\retrac^{\mathrm{St}}_V(V_p)(\retrac^{\mathrm{St}}_V(V_p))^\top},
    \end{equation}
    where $(H,V)$ and $\kh{K,V_p}$ are representations of $(X,G)$ and $(\eta,\zeta)$, respectively.
\end{proposition}
\begin{proof}
    First, we show that $\retrac_{(X,G)}$ is well-defined. Given another representation $\kh{H_0,V_0}$, we have $\kh{H_0,V_0}=(HQ,VQ)$ for some $Q\in\orth(r)$ since $G=I-VV^\top=I-V_0V_0^\top\ \text{and}\  X=HV^\top=H_0V_0^\top$. Therefore, based on the representation $\kh{H_0,V_0}$, the tangent vector $(\eta,\zeta)$ is parameterized by $(K_0,V_{p,0})=(KQ,V_pQ)$, which together with the orthogonally homogeneous property of $\retrac^{\mathrm{St}}$ and $\retrac^{\hanifold}$ leads to
    \begin{align*}
    \retrac^{\hanifold}_H(K)(\retrac^{\mathrm{St}}_V(V_p))^\top &= \retrac^{\hanifold}_{H_0}(K_0)(\retrac^{\mathrm{St}}_{V_0}(V_{p,0}))^\top, 
    \\
    \retrac^{\mathrm{St}}_V(V_p)(\retrac^{\mathrm{St}}_V(V_p))^\top &= \retrac^{\mathrm{St}}_{V_0}(V_{p,0})(\retrac^{\mathrm{St}}_{V_0}(V_{p,0}))^\top.
    \end{align*}   
    This means the value of $\retrac_{(X,G)}(\eta,\zeta)$ is independent of the choice of representations, i.e., the retraction \eqref{eq:1st_retrac} is well-defined. 
    
    Next, consider the following curve 
    \[
        \gamma(t) := \retrac_{(X,G)}(t\eta,t\zeta) = \kh{\retrac^{\hanifold}_H(tK)(\retrac^{\mathrm{St}}_V(tV_p))^\top, I-\retrac^{\mathrm{St}}_V(tV_p)(\retrac^{\mathrm{St}}_V(tV_p)^\top}.    
    \]
    It immediately holds that $\gamma(0)=(HV^\top,I-VV^\top)=(X,G)$. Regarding the first-order derivative, we derive $\gamma^\prime(0)=(KV^\top+HV^\top_p,-V_pV^\top-VV_p^\top)=(\eta,\zeta)$ since $\retrac^{\mathrm{St}}$ and $\retrac^{\hanifold}$ are retractions on $\stiefel(n,r)$ and $\hanifold^r$ respectively. Therefore, we confirm that $\retrac$ defined in \eqref{eq:1st_retrac} is a retraction on $\manifold_h$.
\end{proof}

Proposition~\ref{pro:1st_retrac} decouples the construction of the retraction on $\manifold_h$ into two components: $\retrac^{\mathrm{St}}$ and $\retrac^{\hanifold}$, facilitating the implementation of first-order optimization algorithms on $\manifold_h$ such as the Riemannian gradient descent method. 

Furthermore, to unlock the potential of second-order methods, e.g., the Riemannian Newton method and the Riemannian trust-region method, it is advantageous to explore the \emph{second-order retraction} on $\manifold_h$. Specifically, a retraction $\retrac: \tangent\manifold \rightarrow \manifold$ on the manifold $\manifold$ is second-order if for any $X\in\manifold,\ \xi\in\tangent_Y\manifold$, the curve $\gamma(t):=\retrac_X(t\xi)$ satisfies $\gamma^{\prime\prime}(0)=0$.

According to \cite{absil2012projectionlike}, a second-order retraction can be built by taking a tangent step in the embedding Euclidean space and then projecting it onto the embedded manifold. Therefore, the space-decoupling structure of $\manifold_h$ encourages us to assign the retraction computation to projections onto $\hanifold$ and $\grassmann(n,n-r)$. However, directly applying respective projections to construct a retraction on $\manifold_h$ will break the relation $XG=0$ for $(X,G)\in\manifold_h$ in general, and as a result,
\[\kh{\projection_\hanifold(X+\eta), \projection_{\grassmann(n,n-r)}(G+\zeta)}\notin\manifold_h.
\]
To circumvent this obstacle while preserving the benign properties of the projection, we exploit the specific structure of $\hanifold$ and construct a second-order retraction tailored for $\manifold_h$.

\begin{proposition}[Second-order retraction]\label{pro:2rd_retrac}
    Given $(X,G)\in\manifold_h$ and tangent vector $(\eta,\zeta)\in\tangent_{(X,G)}\manifold_h$ with representations $(H,V)$ and $(K,V_p)$. The mapping $\retrac$ defined as follows is a second-order retraction on $\manifold_h$,
    \begin{equation}\label{eq:2rd_retrac}
    \retrac_{(X,G)}\kh{\eta,\zeta}:=\kh{\zkh{\projection_{\hanifold^r}\kh{(X+\eta)W}}W^\top,I-WW^\top}\,\text{with}\ W=L(L^\top L)^{-\frac{1}{2}},
    \end{equation}
    where $L=V+V_p(I-K^\top HM^{-1}_{H,\omega})$.
\end{proposition}
\begin{proof}
    See Appendix~\ref{app:2rd_retrac}.
\end{proof}

Proposition~\ref{pro:2rd_retrac} demonstrates that combining the projection onto $\hanifold^r$ and the polar retraction on $\stiefel(n,r)$ produces a second-order retraction on $\manifold_h$. In practice, when the manifold $\hanifold$ defined by $h$ is the oblique manifold or the Frobenius sphere, the projection $\projection_{\hanifold^r}$ corresponds to the row-wise normalization or the $\mathrm{F}$-norm normalization, respectively. For a general~$h$, it admits a neighborhood around $H$ in which $\projection_{\hanifold^r}$ is well-defined and smooth; and thus the retraction \eqref{eq:2rd_retrac} satisfies the local smoothness around the point $(X,G)$.

\subsection{Vector transports on $\manifold_h$}
Vector transport provides an approach for transporting tangent vectors from one tangent space to another. It enables the comparison of first-order information at distinct points on the manifold, and thus stands at the core of many algorithms, including the Riemannian BFGS method and the Riemannian conjugate gradient method. 

Specifically, with $\tangent\manifold\oplus\tangent\manifold:=\{(X,\tilde{\xi},\xi):X\in\manifold,\,\tilde{\xi},\xi\in\tangent_X\manifold\}$, a smooth mapping 
\[
\tangent\manifold\oplus\tangent\manifold\to\tangent\manifold:\, (X,\tilde{\xi},\xi)\mapsto \transport_{\tilde{\xi}}(\xi),
\]
is called a \emph{vector transport} on the manifold $\manifold$ if there exists an associated retraction $\retrac$ such that $\transport_{\tilde{\xi}}(\xi)\in\tangent_{\retrac_X(\tilde{\xi})}\manifold$ for any point $X\in\manifold$ and the tangent vectors $\tilde{\xi},\xi\in\tangent_X\manifold$; and it holds that $\transport_{0}(\xi)=\xi$, and the mapping $\xi\mapsto\transport_{\tilde{\xi}}(\xi)$ is linear. Note that a natural construction for embedded submanifolds is the \emph{projection-based transport}, i.e., $\transport_{\tilde{\xi}}(\xi):=\projection_{\retrac_{X}(\tilde{\xi})}(\xi)$. Moreover, $\transport$ is called \emph{isometric} if 
\[
\big\langle{\transport_{\tilde{\xi}}(\xi),\transport_{\tilde{\xi}}(\xi)}\big\rangle_{\retrac_X(\tilde{\xi})}= \innerp{\xi,\xi}_X\ \ \text{for any}\ \ (X,\tilde{\xi},\xi)\in\tangent\manifold\oplus\tangent\manifold.
\]

A projection-based transport on $\manifold_h$ can be constructed as follows.

\begin{proposition}\label{pro:vector_transport_1}
    Given a retraction $\retrac$ on $\manifold_h$, and $\kh{(X,G),(\tilde{\eta},\tilde{\zeta}),(\eta,\zeta)}\in\manifold_h\oplus\manifold_h$, $(H,V)$ and $(K,V_p)$ are representations of  $(X,G)$ and $(\eta,\zeta)$ respectively. The vector transport $\transport_{(\tilde{\eta},\tilde{\zeta})}(\eta,\zeta)$ can be computed by
    \begin{equation}\label{eq:vector_transport_1}
        \begin{aligned}
            &\bar{K} = \projection_{\tangent_{\tilde{H}}\hanifold^r}\kh{KV^\top \tilde{V}+HV_p^\top\tilde{V}},
            \\
            &\bar{V}_p = (I-\tilde{V}\tilde{V}^\top)\kh{VK^\top\tilde{H}+V_pH^\top\tilde{H}+2\omega(VV_p^\top\tilde{V}+V_pV^\top\tilde{V})}M^{-1}_{\tilde{H},\omega},
        \end{aligned}        
    \end{equation}
    where we use $(\tilde{H},\tilde{V})$ to represent $\retrac_{(X,G)}(\tilde{\eta},\tilde{\zeta})$.
\end{proposition}
\begin{proof}
    Let $\tilde{\xi}=(\tilde{\eta},\tilde{\zeta})$, $\xi=(\eta,\zeta)$, and $(\tilde{X},\tilde{G})=\retrac_{(X,G)}(\tilde{\xi})$. It is revealed from \cite[section~8.1.3]{absil2008optimization} that $\retrac$ is associated with the transport $\transport$ by defining $\transport_{\tilde{\xi}}(\xi)$ as the projection of $\xi$ onto the tangent space to $\manifold_h$ at $(\tilde{X},\tilde{G})$. To this end, we recover the tangent vector $\xi = (K V^{\top}+HV_p^{\top},-V_p V^{\top}-V V_p^{\top})$ and employ Proposition~\ref{pro:projection_to_TMh} to project it onto $\tangent_{(\tilde{X},\tilde{G})}\manifold_h$.
\end{proof}

Thanks to the low-rank structure of $\manifold_h$, computing the vector transport~\eqref{eq:vector_transport_1} has a complexity of $\complexity((m+n)r^2)$. However, expression \eqref{eq:vector_transport_1} still remains involved, as it requires the evaluation of several terms including $V^\top\tilde{V}$, $V_p^\top \tilde{V}$ and $H^\top\tilde{H}$. To circumvent this, we extend the ``space-decoupling" spirit, and obtain a vector transport on $\manifold_h$. 

In preparation, we introduce the following property resonating with Definition~\ref{def:orth_homo}. Let $\manifold\subseteq\mbR^{m\times n}$ be a manifold with the property that if $X\in\manifold$, then $XQ\in\manifold$ for all $Q\in \orth(n)$. Given a vector transport $\transport$ on $\manifold$ and its associated retraction $\retrac$ which is orthogonally homogeneous, then they are said to be \emph{compatible} if for any $Q\in\orth(n)$, $X\in\manifold$, $\tilde{\xi},\xi\in\tangent_X\manifold$, and $\tilde{\xi}Q,\xi Q\in\tangent_{XQ}\manifold$, it holds that $\transport_{\tilde{\xi}}(\xi)Q=\transport_{\tilde{\xi}Q}(\xi Q)$. In fact, this compatibility requirement is commonly satisfied on the manifold defined as the level set of an orthogonally invariant mapping, e.g., the Stiefel manifold and the more general $\hanifold$. For instance, given an orthogonally homogeneous retraction $\retrac$, it can be verified that $\transport$ constructed as the projection-based transport, or via the differentiated retraction (see \cite[\S8.1]{absil2008optimization} for details) is compatible with $\retrac$.

\begin{proposition}[Vector transport]\label{pro:vector_transport_2}
    Suppose the orthogonally homogeneous retractions $\retrac^{\hanifold}$ and $\retrac^{\mathrm{St}}$ are associated with their compatible vector transports $\transport^\hanifold$ and $\transport^{\mathrm{St}}$, and the retraction $\retrac$ on $\manifold_h$ is given as in \eqref{eq:1st_retrac}. For $\kh{(X,G),(\tilde{\eta},\tilde{\zeta}),(\eta,\zeta)}\in\manifold_h\oplus\manifold_h$, with representations $(H,V)$, $(\tilde{K},\tilde{V}_p)$, and $(K,V_p)$, respectively, let $\tilde{H}=\retrac^\hanifold_H(\tilde{K})$ and $\tilde{V}=\retrac^{\mathrm{St}}_{V}(\tilde{V_p})$. Then, the vector transport $\transport_{(\tilde{\eta},\tilde{\zeta})}(\eta,\zeta)$ on $\manifold_h$ can be built by
    \begin{equation}\label{eq:vector_transport_2}
        \bar{K}=\transport^{\hanifold}_{\tilde{K}}(K)\ \ \text{and}\ \ \bar{V}_p=(I-\tilde{V}\tilde{V}^\top)\transport^{\mathrm{St}}_{\tilde{V}_p}(V_p).
    \end{equation}
\end{proposition}
\begin{proof}
    The compatibility requirement for $(\retrac^\hanifold,\transport^\hanifold)$ and $(\retrac^{\mathrm{St}},\transport^{\mathrm{St}})$ plays a role in checking that $\transport$ is well-defined, with the proof analogous to that of Proposition~\ref{pro:1st_retrac}. To show $\transport$ is a vector transport, Proposition~\ref{pro:Mh_tangent_space} reveals that $(\bar{K},\bar{V}_p)$ indeed gives a vector tangent to $\manifold_h$ at the point $\retrac_{(X,G)}(\tilde{\eta},\tilde{\zeta})$ represented by $(\tilde{H},\tilde{V})$. Moreover, fixing the direction $\tilde{\xi}:=(\tilde{\eta},\tilde{\zeta})$, the mapping $\xi\mapsto\transport_{\tilde{\xi}}(\xi)$ is linear, and it reduces to the identity on $\tangent_{(X,G)}\manifold_h$ when $\tilde{\xi}=0$ by the definition~\eqref{eq:vector_transport_2}.
\end{proof}

Consequently, transporting vectors on $\manifold_h$ boils down to assembling transports on $\hanifold^r$ and $\stiefel(n,r)$.
Interestingly, taking $V^\top\tilde{V}\approx I$, $V_p^\top \tilde{V}\approx 0$, and $H^\top\tilde{H}\approx \tilde{H}^\top\tilde{H}$, the expression~\eqref{eq:vector_transport_1} turns out to be
\[
    \bar{K}=\projection_{\tangent_{\tilde{H}}\hanifold^r}(K)\ \ \text{and}\ \ \bar{V}_p=(I-\tilde{V}\tilde{V}^\top)(V_p),
\]
which coincides with~\eqref{eq:vector_transport_2} when $\transport^\hanifold$ and $\transport^{\mathrm{St}}$ are chosen as projection-based transports on the respective manifolds.
In view of this, Proposition~\ref{pro:vector_transport_2} indeed simplifies the tedious computation in~\eqref{eq:vector_transport_1}, providing a decent construction of the vector transport on $\manifold_h$ via the ``space-decoupling" principle.

More importantly, expression~\eqref{eq:vector_transport_2} inspires us to develop an isometric $\transport$ on~$\manifold_h$, which is crucial for theoretical analysis of the Riemannian methods \cite{ring2012RCG_BFGS}. We begin by revisiting an isometric transport on the Stiefel manifold~\cite{zhu2017stifelcayley}. For $V\in\stiefel(n,r)$ and $V_p\in\tangent_V\stiefel(n,r)$, let $W(V,V_p)=(I-\frac{1}{2}VV^\top)V_pV^\top-VV_p^\top(I-\frac{1}{2}VV^\top)$ and $\varPi(V,V_p)=\kh{I-\frac{1}{2}W(V,V_p)}^{-1}\kh{I+\frac{1}{2}W(V,V_p)}$. The Cayley transform, which is an orthogonally homogeneous retraction on $\stiefel(n,r)$, and its associated vector isometric transport are given by
\begin{equation}\label{eq:Cayley_pair}
    \retrac^{\mathrm{St}}_V(\tilde{V}_p):=\varPi(V,\tilde{V}_p)V\ \ \text{and}\ \ \transport^{\mathrm{St}}_{\tilde{V}_p}(V_p):=\varPi(V,\tilde{V}_p)V_p.
\end{equation}
It can be verified that $\retrac^{\mathrm{St}}$ and $\transport^{\mathrm{St}}$ are compatible.

\begin{proposition}\label{pro:vector_transport_3}
    Suppose the orthogonally homogeneous retraction $\retrac^{\hanifold}$ is associated with an isometric vector transport $\transport^\hanifold$ on $\hanifold^r$, and they are compatible. For the element $\kh{(X,G),(\tilde{\eta},\tilde{\zeta}),(\eta,\zeta)}\in\manifold_h\oplus\manifold_h$, with representations $(H,V)$, $(\tilde{K},\tilde{V}_p)$, and $(K,V_p)$, respectively, let $\tilde{H}=\retrac^\hanifold_H(\tilde{K})$ and $\tilde{V}=\retrac^{\mathrm{St}}_{V}(\tilde{V_p})$. Then, an isometric vector transport $\transport_{(\tilde{\eta},\tilde{\zeta})}(\eta,\zeta)$ on $\manifold_h$ can be built by
    \begin{equation}\label{eq:vector_transport_3}
        \bar{K}=\transport^{\hanifold}_{\tilde{K}}(K)\ \ \text{and}\ \ \bar{V}^{}_p=\transport^{\mathrm{St}}_{\tilde{V}_p}(V_pM^{\frac{1}{2}}_{H,\omega})M_{\tilde{H},\omega}^{-\frac{1}{2}},
    \end{equation}
    where the $(\retrac^{\mathrm{St}},\transport^{\mathrm{St}})$ is given by \eqref{eq:Cayley_pair}.
\end{proposition}
\begin{proof}
    Note that by the formula~\eqref{eq:Cayley_pair}, it can be verified that $V_p^\top V=0$ indicates $\bar{V}_p^\top \tilde{V}=0$. Similar to Proposition~\ref{pro:vector_transport_2}, we can prove the constructed $\transport$ is well-defined. Additionally, considering the Riemannian metric~\eqref{eq:inner_product} on $\manifold_h$ leads to
    \begin{align*}
        \innerp{\bar{K},\bar{K}} + \innerp{\bar{V}_p,\bar{V}_pM_{\tilde{H},\omega}} =&\ \innerp{\transport^{\hanifold}_{\tilde{K}}(K),\transport^{\hanifold}_{\tilde{K}}(K)} + \innerp{\transport^{\mathrm{St}}_{\tilde{V}_p}(V_pM^{\frac{1}{2}}_{H,\omega}),\transport^{\mathrm{St}}_{\tilde{V}_p}(V_pM^{\frac{1}{2}}_{H,\omega})}
        \\
        =&\innerp{K,K} + \innerp{V_pM^{\frac{1}{2}}_{H,\omega},V_pM^{\frac{1}{2}}_{H,\omega}}
        \\
        =& \innerp{K,K} + \innerp{V_p,V_pM_{H,\omega}},
    \end{align*}
    which implies the $\transport$ is an isometric vector transport on $\manifold_h$.
\end{proof}

In summary, we have provided important ingredients to tackle problem \eqref{eq:G} via Riemannian optimization. Notably, if the geometry of $\hanifold$, e.g., $\nabla_\hanifold^2$ in \eqref{eq:hessian_expression} and~$\projection_{\hanifold^r}$ in \eqref{eq:2rd_retrac}, is understood a priori, then implementing Riemannian optimization algorithms on $\manifold_h$ becomes convenient.

\subsection{Relationship of stationary points} 
As pointed out in \cite{levin2024effectlift}, nonlinear parameterizations may introduce spurious stationary points, which means, in our context, the image of a stationary point for the reformulated problem \eqref{eq:G} through the mapping $\phi$ is not stationary for the original problem \eqref{eq:lowrank_orthinvar}. Therefore, to establish the equivalence between the two problems, it is worth investigating the relationships of their stationary points. We say that the parameterization $(\manifold_h,\phi)$ satisfies ``${k\Rightarrow 1}$" $(k=1,2)$ at $(X,G)$, if for any objective function $f$, $(X,G)$ being a $k$-th-order stationary point for problem \eqref{eq:G} implies that $\phi(X,G)$ is a first-order stationary point for problem~\eqref{eq:lowrank_orthinvar}.

In the following theorem, we give an analysis for the ``$k\Rightarrow 1$'' property of $(\manifold_h,\phi)$. Specifically, based on the developed geometry in Theorem~\ref{the:tangent_intersection_rule}, we verify the sufficient and necessary condition for ``$1\Rightarrow 1$'', and then confirm that $(\manifold_h,\phi)$ satisfies ``$2\Rightarrow 1$''.

\begin{theorem}\label{pro:11_21_Mh}
   The space-decoupling parameterization $(\manifold_h,\phi)$ satisfies
   \begin{enumerate}
      \item[\emph{(i)}] ``\,$1\Rightarrow 1$" holds at $\kh{X,G}$ if and only if $\rank\kh{X}=r$;
      \item[\emph{(ii)}] ``\,$2\Rightarrow 1$" holds everywhere on $\manifold_h$.
    \end{enumerate}
\end{theorem}
\begin{proof}
    Given $(X,G)\in\manifold_h$ with $\rank\kh{X}=s\le r$ and the SVD $X=U\varSigma V^\top=HV^\top$ where $H=U\varSigma$, and we take $U\in\stiefel(m,r)$, $\varSigma\in\mbR^{r\times r}$, and $V\in\stiefel(n,r)$.

    (i) Consider the image of $\tangent_{(X,G)}\manifold_h$ under the mapping $\diff \phi_{(X,G)}$, i.e., 
    \begin{align}
    \ima(\diff \phi_{(X,G)}):=&\ \diff\phi_{(X,G)}[\tangent_{(X,G)}\manifold_h] \nonumber
    \\
    =&\ \{ K V^{\top}+HV_p^{\top} : K\in\tangent_H \hanifold^r,\,V_p\in\mbR^{n\times r},\,  V^{\top} V_p=0\},  \label{eq:image_Dphi}
    \end{align}
    which results from \eqref{eq:Mh_tangent_space}. As revealed by \cite[Theorem 2.4]{levin2024effectlift}, $\ima(\diff \phi_{(X,G)})=\tangent_X(\boundedrank\cap\hanifold)$ is the sufficient and necessary condition for ``$1\Rightarrow 1$".
    
    If $s=r$, Theorem~\ref{the:tangent_intersection_rule} shows $\tangent_X(\boundedrank\cap\hanifold) = \{KV^\top + UJV_\bot^\top :\ K\in\tangent_{H}\hanifold^r,\ J\in \mbR^{r\times(n-r)}\}$. Comparing it with~\eqref{eq:image_Dphi}, we have $\ima(\diff \phi_{(X,G)})=\tangent_X(\boundedrank\cap\hanifold)$ since $H$ is full-rank in this case, and thus ``$1\Rightarrow 1$'' holds. 
    
    On the other hand, when $s<r$, the expression \eqref{eq:characterize_TRH} implies the tangent cone to $\boundedrank\cap\hanifold$ at $X$ is not a linear space, which concludes $\ima(\diff \phi_{(X,G)})\neq\tangent_X(\boundedrank\cap\hanifold)$, and thus ``$1\Rightarrow 1$'' fails.
    
    (ii) Define a smooth mapping $l:\mbR^{m\times r}\times \mbR^{n\times r} \rightarrow \mbR^q\times \sym(r):(\tilde{H},\tilde{V})\mapsto (h^r(\tilde{H}), \tilde{V}^\top \tilde{V} -I)$. We note that the embedded manifold $\hanifold^r \times \stiefel(n, r)=l^{-1}(0)\subseteq\mbR^{m\times r}\times \mbR^{n\times r}$ and consider the smooth mapping $\varphi:\hanifold^r\times \stiefel(n, r) \rightarrow \manifold_h:(\tilde{H},\tilde{V})\mapsto (\tilde{H}\tilde{V}^\top, I-\tilde{V}\tilde{V}^\top)$, which is a surjection onto $\manifold_h$. Subsequently, we introduce the composition \revise{${\psi}:=\phi\circ\varphi$}, and slightly abuse the notation of ${\psi}$ and its extension to the whole Euclidean spaces, i.e., ${\psi}:\,\mbR^{m\times r}\times \mbR^{n\times r} \rightarrow \mbR^{m\times n}:\, (\tilde{H},\tilde{V})\mapsto \tilde{H}\tilde{V}^\top$. We aim to show the ``$2\Rightarrow 1$'' property of $(\hanifold^r \times \stiefel(n, r),\psi)$,\footnote{Given a parameterization $(\manifold,{\psi})$ such that $\psi(\manifold)=\boundedrank\cap\hanifold$, the ``$k\Rightarrow 1$" properties are generalized similarly based on the stationarity of the smooth problem  $\min_{x\in\manifold}f\circ \psi(x)$ and the original problem~\eqref{eq:lowrank_orthinvar}.} which implies the ``$2\Rightarrow 1$'' property of $(\manifold_h,\phi)$, as supported by \cite[Proposition~3.27]{levin2024effectlift}. 
    
    Recall that the considered $X=HV^\top$ and let the tangent vector of the product manifold at the point $(H,V)$ be
        \begin{align}
        v:=\kh{K,\beta}\in&\ \tangent_{\kh{H,V}}\kh{\hanifold^r\times \stiefel(n,r)}    \nonumber
        \\
        =& \hkh{(\tilde{K},V\Omega+V_\bot \tilde{B}):\ \tilde{K}\in\tangent_H\hanifold^r, \Omega\in\sksym(r), \tilde{B}\in\mbR^{(n-r)\times r}}.       \label{eq:tangent_psi}
        \end{align}
        We then give the associated $u_v:=(u_K,u_\beta)\in\mbR^{m\times r}\times \mbR^{n\times r}$ such that $\diff^2 l_{(H,V)}[v,v] + \diff l_{(H,V)}[u_v] = 0$, which, by calculating the derivatives of $l$, can be specified as
        \begin{equation} \label{eq:tm2_condition}
            ( \diff h^r_H[ u_K ] ,V^{\top}u_\beta+u_{\beta}^{\top}V ) = -( D^2 h^r_H[ K,K ] ,2\beta ^{\top}\beta ).
        \end{equation}
         Subsequently, mappings $\tanL_{(H,V)},\tanQ_{(H,V)}:\,\tangent_{\kh{H,V}}\kh{\hanifold^r\times \stiefel(n,r)} \to \tangent_{X}(\boundedrank\cap\hanifold)$ are defined as follows,
        \begin{align}
            &\tanL_{\kh{H,V}}(v) := \diff{\psi}_{(H,V)}\zkh{v} = KV^\top + H\beta^\top, \label{eq:tangentL}
            \\
            &\tanQ_{\kh{H,V}}(v) := \diff^2{\psi}_{(H,V)}[v,v] + \diff{\psi}_{(H,V)}[u_v] = 2K\beta^\top + u_KV^\top + Hu_\beta^\top,   \label{eq:tangentQ}
        \end{align}
        for $(u_K,u_\beta)$ satisfying \eqref{eq:tm2_condition}.\footnote{Different choices of $(u_K, u_\beta)$ yield equivalent $\tanQ_{\kh{H,V}}(v)$, modulo $\ima(\tanL_{(H,V)})$, ensuring no ambiguity in the definition of $\banifold_{(H,V)}$.}
        
       Next, we turn to prove that the set 
        \begin{equation}\label{eq:identify_By}
            \banifold_{(H,V)} := \bigcup_{\kh{v_i}_{i\ge 1}:\tanL_{(H,V)}(v_i)\to 0}\lim_{i\to\infty}\kh{\tanQ_{(H,V)}(v_i) + \ima(\tanL_{(H,V)})}
        \end{equation}
        satisfies $ \banifold_{(H,V)}^\circ\subseteq(\tangent_X(\boundedrank\cap\hanifold))^\circ$, which is sufficient for the ``$2\Rightarrow 1$'' property of $(\hanifold^r\times \stiefel(n,r),\psi)$ \cite[Theorem~3.23]{levin2024effectlift}.
        
        Let $\vecspan(M)$ denote the linear space spanned by the columns of the matrix~$M$. We first note that $\vecspan(H)$ admits an orthogonal basis $Q_H\in\mbR^{m\times s}$ since $\rank(H)=\rank(X)=s<r$, and correspondingly, $Q_{H_\bot}\in\mbR^{m\times (m-s)}$ denotes the orthogonal complement of $Q_H$. In this view, taking into account the definition~\eqref{eq:tangentL} of $\tanL_{(H,V)}$ and its domain \eqref{eq:tangent_psi} yields
        \begin{equation}\label{eq:imageL}
            \hkh{\tilde{K}V^\top+Q_HBV^\top_\bot:\ \tilde{K}\in\tangent_H\hanifold^r,B\in\mbR^{s\times(n-r)}} \subseteq \ima\tanL_{\kh{H,V}}.
        \end{equation}
         
        Next, given arbitrary vectors $\theta$ and $w$ in $\vecspan(Q_{H_\bot})$ and $\vecspan(V_\bot)$, respectively, we aim to construct a sequence $\{v_\varepsilon\}_{\epsilon>0}$ such that $ \tanL_{\kh{H,V}}\kh{v_\varepsilon}\to 0$ while $\tanQ_{\kh{H,V}}(v_\varepsilon)\to \theta w^\top$. To this end, choose a vector $z\in\mbR^{r}$ such that $\norm{z}_2=1$ and $Hz=0$. Setting $v_\varepsilon = (K_\varepsilon,\beta _{\varepsilon})=(\frac{\varepsilon}{2}\theta z^{\top}, \frac{1}{\varepsilon}wz^{\top})\in\tangent_{(H,V)}(\hanifold^r\times \stiefel(n,r))$ for $\varepsilon>0$ produces
        \begin{equation}\label{eq:limittanL}
            \tanL_{\kh{H,V}}\kh{v_\varepsilon}=\frac{\varepsilon}{2}\theta z^{\top}V^{\top}+\frac{1}{\varepsilon}Hzw^{\top} = \frac{\varepsilon}{2}\theta z^{\top}V^{\top} \xlongrightarrow{\ \ \varepsilon \to 0\ \ } 0.
        \end{equation}
        Regarding the mapping $\tanQ_{(H,V)}$, we incorporate $(K_\varepsilon,\beta_\varepsilon)$ on the right side of~\eqref{eq:tm2_condition} to find an $u_{\varepsilon} := (u_{K\varepsilon},u_{\beta\varepsilon})$ as a solution for the following equation
        \begin{align}
            ( \diff h^r_H[ u_{K\varepsilon} ] ,V^{\top}u_{\beta\varepsilon}+u_{\beta\varepsilon}^{\top}V ) = -( \diff^2h^r_H[ K_\varepsilon,K_\varepsilon ] ,2\beta_\varepsilon^\top\beta_\varepsilon ).      \label{eq:equation_uv}
        \end{align}
        Specifically when $\varepsilon=1$, there exits a smooth curve $\gamma_{K1}(t)$ on $\hanifold^r$ with $\gamma_{K1}(0)=H$ and $\gamma_{K1}^\prime(0)=K_1$. Twice differentiating $h^r(\gamma_{K1}(t))$ at $t=0$ yields $\diff h^r_H[ u_{K1}] + \diff^2 h^r_H[ K_1,K_1] =0$, where we take $u_{K1}=\ddot{\gamma}_{K1}(0)$. Then, let $u_{\beta 1}=-\| w \| ^2Vzz^{\top}$, and thus it follows that $V_{}^{\top}u^{}_{\beta 1}+u_{\beta 1}^{\top}V^{}_{} = -2\| w \| ^2zz^{\top}=-2\beta_1^\top\beta_1$. Therefore, we obtain $u_{\varepsilon} = (u_{K\varepsilon},u_{\beta\varepsilon}) = (\frac{\varepsilon^2}{4}u_{K1}, \frac{2}{\varepsilon^2}u_{\beta 1})$ satisfying \eqref{eq:equation_uv}. Substituting the constructed $v_\varepsilon$ and $u_{\varepsilon}$ into \eqref{eq:tangentQ} produces
        \begin{align}
            \tanQ_{\kh{H,V}}\kh{v_\varepsilon} = 2K_\varepsilon\beta_{\varepsilon}^\top + u_{K\varepsilon}V^\top + Hu_{\beta\varepsilon}^\top    =\theta w^\top + \frac{\varepsilon^2}{4}u_{K1}V^\top \xrightarrow{\ \varepsilon \to 0\ } \theta w^\top. \label{eq:limittanQ}
        \end{align}
        Combining \eqref{eq:limittanL} with \eqref{eq:limittanQ}, we conclude that for any $\theta\in\vecspan\kh{Q_{H_\bot}}$ and $w\in\vecspan\kh{V_\bot}$, there exist $\hkh{v_\varepsilon}_{\varepsilon>0}$ and associated $\hkh{u_{\varepsilon}}_{\varepsilon>0}$ such that
        \[
        \lim_{\varepsilon\to 0}\tanL_{\kh{H,V}}\kh{v_\varepsilon} = 0.\ \ \text{and}\ \ 
        \lim_{\varepsilon\to 0}\kh{\tanQ_{\kh{H,V}}(v_\varepsilon) + \ima\tanL_{\kh{H,V}}} = \theta w^\top + \ima\tanL_{\kh{H,V}}.
        \]
        Therefore, taking the above result and the estimation \eqref{eq:imageL} into the definition~\eqref{eq:identify_By} of $\banifold_{\kh{H,V}}$ reveals that the following set
        \begin{equation}\label{eq:mathcalA}
            \begin{aligned}
                \aanifold_{\kh{H,V}} := \bigg \{ &\tilde{K}V^\top+Q_HBV^\top_\bot + \theta w^\top:\ \tilde{K}\in\tangent_H\hanifold^r,
                \\
                 &B\in\mbR^{s\times(n-r)}, \,\theta\in\vecspan\kh{Q_{H_\bot}}, \,w\in\vecspan\kh{V_\bot} \bigg\}
            \end{aligned}            
        \end{equation}
        satisfies $\aanifold_{(H,V)} \subseteq \banifold_{(H,V)}$. Given any $Y\in\tangent_X\hanifold$, Proposition~\ref{pro:decompose_geo_H} implies $YV\in \tangent_H\hanifold^r$, and we have the following decomposition of $Y$ according to the orthogonality of $[V\,\,V_\bot]$ and $[Q_H\,\,Q_{H_\bot}]$,
        \[
            Y = (YV) V^\top + Q^{}_H(Q_H^\top Y^{}{V_\bot})V^\top_\bot + Q^{}_{H_\bot}(Q_{H_\bot}^\top Y^{}{V_\bot})V^\top_\bot.
        \]
        Comparing the above expression with \eqref{eq:mathcalA} reveals $Y\in\conv(\aanifold_{(H,W)})$, where $\conv(\cdot)$ denotes the convex hull of a set. By the arbitrariness of $Y\in\tangent_X\hanifold$, it holds that $\tangent_X\hanifold \subseteq \conv(\aanifold_{(H,W)})$. Consequently,
        \[
            \banifold_{\kh{H,V}}^\circ \subseteq \aanifold_{\kh{H,V}}^\circ = \kh{\conv(\aanifold_{(H,V)})}^\circ \subseteq (\tangent_X \hanifold)^\circ \subseteq (\tangent_{X}(\boundedrank\cap\hanifold))^\circ,
        \]
        where the last ``$\subseteq$'' comes from \eqref{eq:normal_intersection_rule}.
        \end{proof}
        
    Theorem~\ref{pro:11_21_Mh} relates the landscapes of the proposed problem \eqref{eq:G} and the original problem \eqref{eq:lowrank_orthinvar}. The result shows that finding a second-order stationary point for \eqref{eq:G} is sufficient to obtain a first-order stationary point for \eqref{eq:lowrank_orthinvar}. Specifically, when the point of interest is of rank $r$, the first-order stationarity for \eqref{eq:G} is adequate.

\subsection{Convergence properties} 
This subsection develops the convergence results for Riemannian algorithms on the proposed $\manifold_h$. Specifically, monotone algorithms admit at least one accumulation point, and more importantly, iterates generated by second-order algorithms can find first-order stationary points for problem \eqref{eq:lowrank_orthinvar}, which illustrates that the proposed parameterization $(\manifold_h,\phi)$ circumvents the undesirable apocalypse. 

First, we show that $\manifold_h$ is complete and $\phi$ is proper.

\begin{proposition}\label{pro:complete}
    The manifold $\manifold_h$ is complete.
\end{proposition}
\begin{proof}
    The closeness of $\boundedrank\cap\hanifold$ and the continuity of $\phi$ show that $\manifold_h=\phi^{-1}(\boundedrank\cap\hanifold)$ is closed in $\eanifold=\mbR^{m\times n}\times\sym(n)$. Therefore, endowed with the Riemannian metric \eqref{eq:rie_geometry}, $\manifold_h$ is a closed submanifold of the Euclidean space $\eanifold$, which reveals the completeness of $\manifold_h$.
\end{proof}

\begin{proposition}\label{pro:proper}
    The mapping $\phi$ is proper, i.e., the preimage of any compact set~$\canifold\subseteq\boundedrank\cap\hanifold$ is compact.
\end{proposition}
\begin{proof}
    By the continuity of $\phi$, the preimage $\phi^{-1}(\canifold)$ is a closed subset of $\canifold\times\grassmann(n,n-r)$ which is compact in $\mbR^{m\times n}\times \mbR^{n\times n}$. Therefore, $\phi^{-1}(\canifold)$ is also compact.
\end{proof}

Proposition~\ref{pro:proper} reveals that the reformulated problem \eqref{eq:G} preserves the compactness of the sublevel sets, that is, for any $c\in\mbR$, if $f^{-1}((-\infty,c])$ is compact, then $\bar{f}^{-1}((-\infty,c])=\phi^{-1}\circ f^{-1}((-\infty,c])$ is also compact. 

Subsequently, we can derive the general convergence results for algorithms on $\manifold_h$ as follows\revise{, which extends the results of~\cite[Theorem~1.1]{levin2023remedy}.}

\begin{theorem}\label{the:convergence}
    Given an initialization $(X_0,G_0)\in\manifold_h$ with the sublevel set $\{X:\,f(X)\le f(X_0)\}$ being compact, the sequence $\{(X_k,G_k)\}$ generated by any monotone algorithms, i.e., $\bar{f}(X_{k+1},G_{k+1})\le \bar{f}(X_{k},G_{k})$ has at least one accumulation point. Moreover, if the accumulation point $(X^*,G^*)$ is second-order stationary for $\eqref{eq:G}$, then $X^*$ is first-order stationary for \eqref{eq:lowrank_orthinvar}.
\end{theorem}
\begin{proof}
    By the monotone property, the sequence lies in the set $\{(X,G):\,\bar{f}(X,G)\le\bar{f}(X_0,G_0)\}$, which is compact due to the compactness of $\{X:\,f(X)\le f(X_0)\}$. Therefore, the sequence admits at least one accumulation point. Additionally, applying Theorem~\ref{pro:11_21_Mh} implies $X^*$ is first-order stationary for \eqref{eq:lowrank_orthinvar}.
\end{proof}

\section{Numerical experiments}\label{sec:experiments}
Recasting the coupled-constrained problem \eqref{eq:lowrank_orthinvar} into the Riemannian optimization problem \eqref{eq:G}, we aim to numerically validate the performance of Riemannian algorithms on the proposed manifold $\manifold_h$. 

Specifically, we solve \eqref{eq:G} using the Riemannian gradient descent (RGD) and Riemannian trust-region (RTR) methods to obtain $x\in\manifold_h$, and then map it to $X=\phi(x)\in\boundedrank\cap\hanifold$ as a solution for \eqref{eq:lowrank_orthinvar}. In the implementations, RGD utilizes the first-order retraction~\eqref{eq:1st_retrac} with $\retrac^{\mathrm{St}}$ as the polar retraction on Stiefel manifold, and RTR employs the second-order retraction~\eqref{eq:2rd_retrac}. For experiments in sections~\ref{sec:LR_obliquedata}-\ref{sec:LR_Markov}, algorithms on $\manifold_h$ are abbreviated in the fashion of ``$\manifold_h$-RGD" and ``$\manifold_h$-RTR"; while for reinforcement learning and deep learning experiments (sections~\ref{sec:LRRL} and \ref{sec:LRNN}), our approach is named with task-specific terms to additionally identify the contributions to the modeling for respective applications. 

The experiments are produced on a workstation that consists of two Intel(R) Xeon(R) Gold 6330 CPUs (at $2.00$GHz$\times28$, $42$M Cache), 512GB RAM, and one NVIDIA A800 (80GB memory) GPU. The reinforcement learning and deep learning experiments (sections~\ref{sec:LRRL} and \ref{sec:LRNN}) run in Python (Release 3.8.10) on the GPU, while all other experiments (sections~\ref{sec:LR_obliquedata}-\ref{sec:LR_Markov}) are carried out in MATLAB (Release {9.7.0}) on the CPUs, adopting the \texttt{Manopt} toolbox~\cite{boumal2014manopt}. The codes of the proposed framework are available at \href{https://github.com/UCAS-YanYang}{https://github.com/UCAS-YanYang}.

\subsection{Low-rank approximation of spherical data}\label{sec:LR_obliquedata}
In this experiment, we test RGD and RTR on the proposed $\manifold_h$ with the associated $h(X)=XX^\top-{\bm 1}$ to approximate data points on the sphere. Reformulating~\eqref{eq:sphericalfitting} and additionally introducing a sampling scheme for broader applicability, we concentrate on the following model
\begin{equation}\label{eq:sphericalfitting_Mh}
    \min_{x\in\manifold_h}\ \ \ \bar{f}(x):=\frac{1}{2} \norm{\projection_\Omega (\phi(x)-A)}_\frob^2,
\end{equation}
where $A\in\oblique(m,n)$ represents $m$ data points in $\mbR^n$, $\Omega\subseteq \{1,2,\ldots,m\}\times\{1,2,\ldots,n\}$ is an index set, and $\projection_\Omega$ is the projection operator onto $\Omega$, i.e., $\projection_\Omega(X)(i,j)=X(i,j)$ if $(i,j)\in\Omega$, otherwise $\projection_\Omega(X)(i,j)=0$. Given the rank $r^*$, we generate a synthetic low-rank data matrix $A\in\oblique(m,n)$ by $A=\projection_{\oblique(m,r^*)}(U^*\varSigma^*)(V^*)^\top$,
where $U^*\in\stiefel(m,r^*),\,V^*\in\stiefel(n,r^*)$ are sampled from the standard normal distribution $\nanifold(0,1)$ and then orthogonalized by the QR factorization; and the nonzero elements of the diagonal matrix $\varSigma^*\in\mbR^{r^*\times r^*}$ are sampled from the standard uniform distribution on $(0,1)$. \revise{The \emph{oversampling factor} ($\mathrm{OS}$) for a rank-$r^*$ matrix is defined as follows, $$\mathrm{OS}:=\frac{|\varOmega|}{r^*\times (m+n-r^*)}.$$}

We configure the experiment by $(m, n) = (5000, 6000)$ and $r^* = 6$, and test algorithms with rank parameters $r\ge r^*$. For the initial guess $x_0\in\manifold_h$ represented by $(H_0,V_0)$, the orthogonal component $V_0$ is generated in the same manner as $V^*$, while $H_0$ is initialized as $H_0=\projection_{\oblique(m,n)}(\tilde{H}_0)$, where $\tilde{H}_0$ is sampled from $\nanifold(0,1)$ when $r=r^*$, and is constructed by randomly selecting $r$ columns from $A$ when $r>r^*$.

We examine the performance of Riemannian algorithms with varying metric parameters $\omega$ in \eqref{eq:rie_geometry}, and the algorithm is terminated if 1) the Riemannian gradient satisfies $\|\nabla_{\manifold_h} \bar{f}(x_k)\|\le 10^{-10}$ for RGD or $\|\nabla_{\manifold_h} \bar{f}(x_k)\|\le 10^{-13}$ for RTR; 2) it reaches the maximum iteration $500$; 3) runtime exceeds $200$ seconds. The performance of algorithms is assessed by the test error ${\|\projection_\varGamma(\phi(x)-A)\|_\frob}/{\|\projection_\varGamma(A)\|_\frob}$
with the test set $\varGamma$ ($|\varGamma|=|\Omega|$).

\begin{figure}[htbp]
\hspace{1mm}
\begin{minipage}{1\textwidth}
    \centering
    \includegraphics[width=1.1\linewidth]{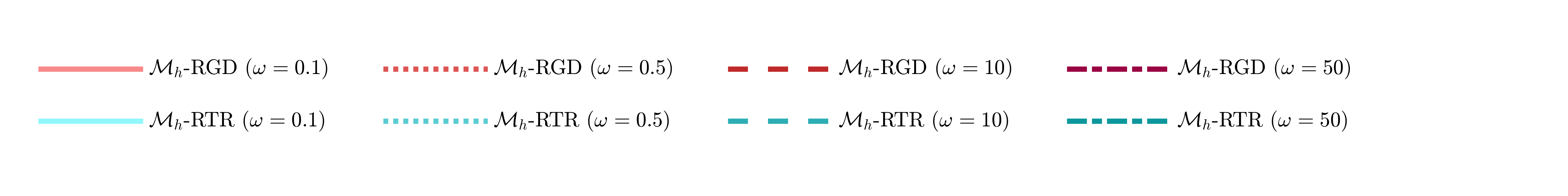}
\end{minipage}
\\[-8.5mm]
\begin{center}
\begin{minipage}{0.48\textwidth}
    \centering
    \includegraphics[width=1\linewidth]{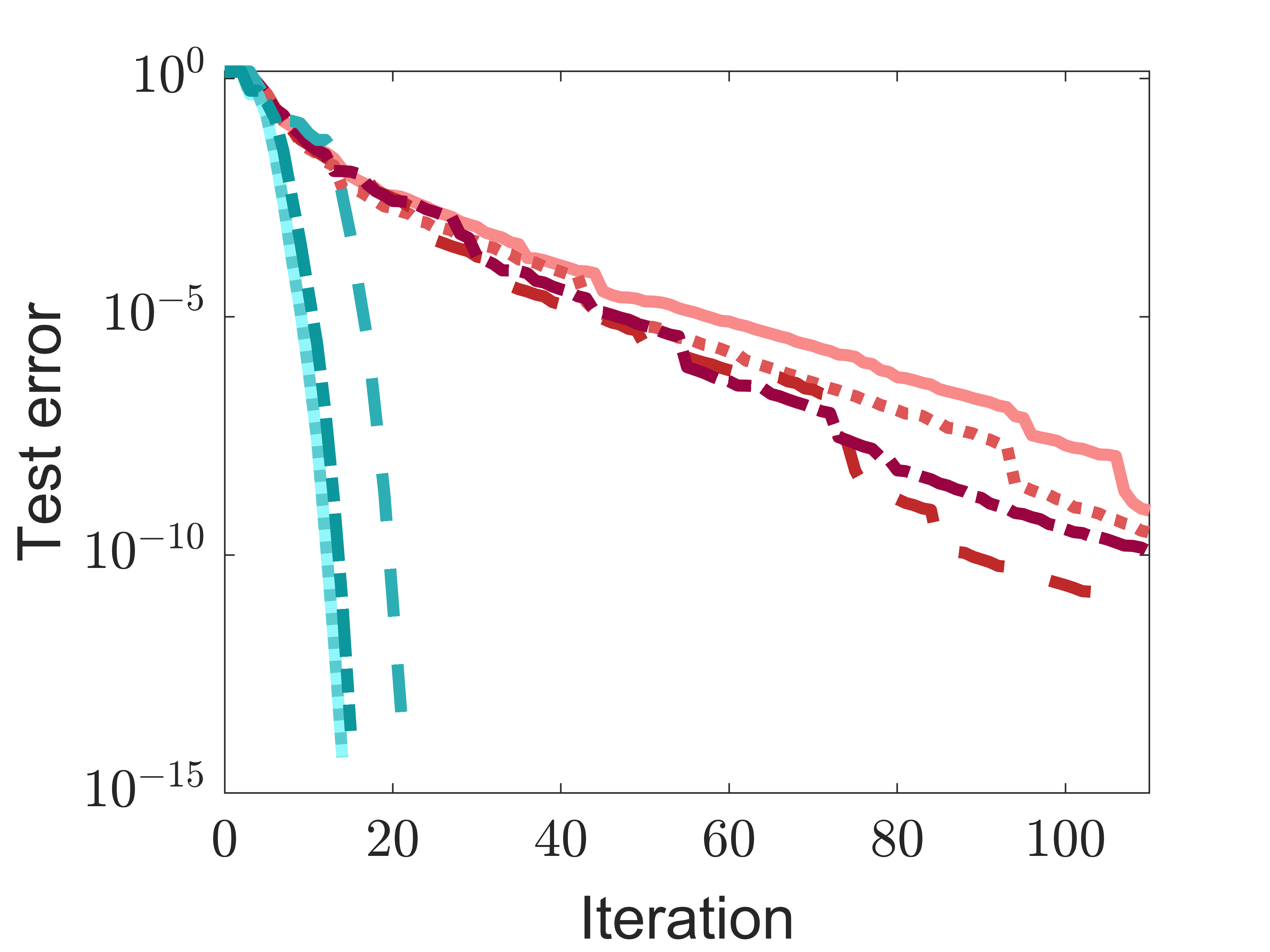}
\end{minipage}
\!\!\!
\begin{minipage}{0.48\textwidth}
    \centering
    \includegraphics[width=1\linewidth]{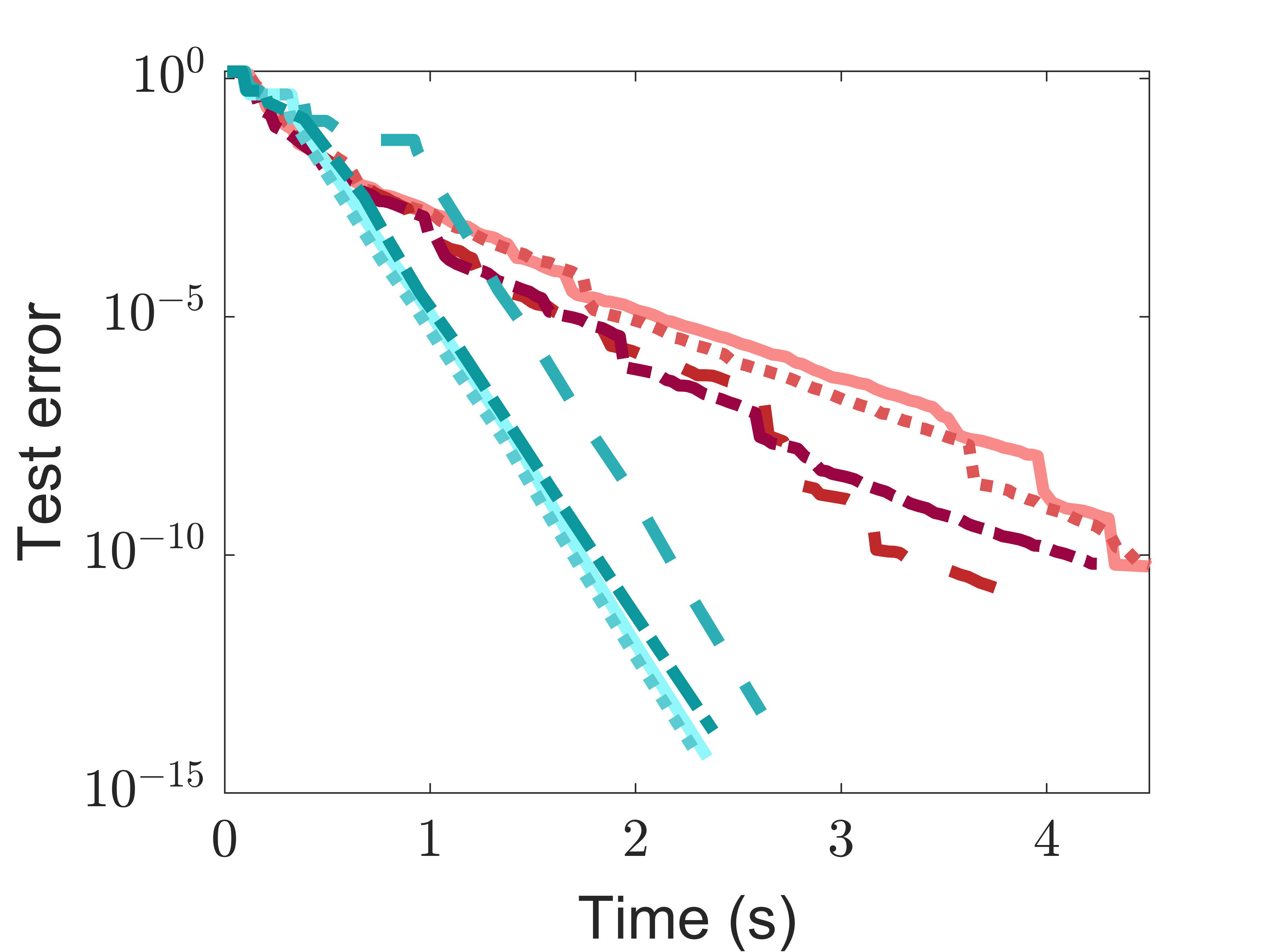}
\end{minipage}
\end{center}
\begingroup
\color{revisecolor} 
\caption{\revise{Spherical data fitting problem with the oversampling factor $\mathrm{OS}=5$ and the unbiased rank parameter $r=r^*=6$.}}
\label{fig:exact_fitting}
\endgroup
\end{figure}

\revise{
\emph{Test with unbiased rank parameter.}
The algorithms are evaluated with the unbiased rank parameter $r = r^*$. Test errors under the oversampling factor $\mathrm{OS}=5$ are reported in \myfig\ref{fig:exact_fitting}. Both RGD and RTR methods show robustness concerning the parameter $\omega$ by successfully recovering the underlying low-rank matrix $A\in\oblique(m,n)$, as evidenced by the test errors. 
}

\revise{
\emph{Test with over-estimated rank parameter.} The algorithms are evaluated under over-estimated rank parameters $r =7,8,9,10>r^*$ and the oversampling factor $\mathrm{OS}=5$. After conducting initial tests, we observe that $\omega=0.5$ for RGD and $\omega=10$ for RTR provide robust performance, for which they will serve as the default settings for all the remaining experiments across this paper. Table~\ref{tab:overestimate_fitting} reveals that the algorithms reconstruct the true data matrix~$A$ across all the selected rank parameters.
}

\begin{table*}[htbp]
\begingroup
\color{revisecolor} 
\caption{\revise{Spherical data fitting problem with the oversampling factor $\mathrm{OS}=5$ and over-estimated rank parameters $r=7,8,9,10>r^*$.}}
\begin{center}
\setlength{\tabcolsep}{4pt}
\begin{tabular}{lcccccccc}
\toprule
\multicolumn{1}{l}{\multirow{2}{*}{Algorithm}} & \multicolumn{2}{c}{$r=7$} & \multicolumn{2}{c}{$r=8$} & \multicolumn{2}{c}{$r=9$}& \multicolumn{2}{c}{$r=10$}\\ \cmidrule(l{2pt}r{2pt}){2-3} \cmidrule(l{2pt}r{2pt}){4-5} \cmidrule(l{2pt}r{2pt}){6-7} \cmidrule(l{2pt}r{2pt}){8-9}
& \multicolumn{1}{c}{Test err.} & \multicolumn{1}{c}{Time} & \multicolumn{1}{c}{Test err.} & \multicolumn{1}{c}{Time}& \multicolumn{1}{c}{Test err.} & \multicolumn{1}{c}{Time}& \multicolumn{1}{c}{Test err.} & \multicolumn{1}{c}{Time} \\
\midrule
$\manifold_h$-RGD & $\expnumber{1.97}{-11}$ & $10.23$ & $\expnumber{1.69}{-10}$ & $11.22$ & $\expnumber{1.49}{-11}$ & $9.56$ &$\expnumber{1.08}{-9}$ & $\phantom{0}7.30$\\
$\manifold_h$-RTR & $\expnumber{2.81}{-15}$  & $\phantom{0}6.03$ & $\expnumber{3.97}{-15}$ & $15.34$ & $\expnumber{2.41}{-14}$ & $4.51$ &$\expnumber{6.77}{-14}$ & $17.08$ \\
\bottomrule
\end{tabular}
\end{center}
\label{tab:overestimate_fitting}
\endgroup
\end{table*}


\subsection{Low-rank approximation of graph similarity matrices}\label{sec:LR_graph}
Given two graphs $G_A$ and $G_B$ with $m$ and $n$ nodes, respectively, we denote their adjacency matrices as $A\in\mbR^{m\times m}$ and $B\in\mbR^{n\times n}$. For $G_A$, $A(i,j)=1$ indicates the existence of an edge from node $i$ to node $j$, while $A(i,j) = 0$ otherwise. Moreover, the sets of children and parents of node $i$ in $G_A$ are denoted by $C_A(i)$ and $P_A(i)$, respectively. The same notation applies for the graph $G_B$. 

\revise{Given a matrix $X\in \mathrm{S}_{\mathrm{F}}(m,n)$, adopt $X(i,j)$ to estimate the similarity between node $i$ in $G_A$ and node $j$ in $G_B$, and the basic principle is that $i$ and $j$ should be considered similar if their respective neighbors are also similar. To aggregate the similarity between neighbors, Blondel et al.~\cite{blondel2004measuresimilarity} introduced a linear operator $\lanifold:\,\mbR^{m\times n}\to\mbR^{m\times n}$ defined by
$${[\lanifold(X)](i,j):=\sum_{\substack{s \in C_A(i) \\ t\in C_B(j)}}X(s,t) + \sum_{\substack{s \in P_A(i) \\ t\in P_B(j)}}X(s,t),}$$
for $i=1,2,\ldots,m$ and $j=1,2,\ldots,n$, and proposed a update scheme as follows,
\begin{equation}\label{eq:graph_full_rule}
X_{k+1}=\frac{\lanifold(X_k)}{\|\lanifold(X_k)\|_\frob}.  
\end{equation}
It was proved in \cite[\S3]{blondel2004measuresimilarity} that given any initialization $X_0>0$, where the ``$>$'' is understood component-wise, the subsequences $\{X_{2k}\}$ and $\{X_{2k+1}\}$ converge, respectively. Consequently, the \emph{similarity matrix} is defined by $X^*:=\lim_{k\to+\infty} X_{2k}$ with $X_0={{\bm 1}^{}_m {\bm 1}^\top_n}/{\|{\bm 1}^{}_m {\bm 1}^\top_n\|_\frob}$, where ${\bm 1}_{m}\in\mbR^m$ denotes the all-ones vector. Furthermore, Cason et al. \cite{cason2013iterative} showed that solving problem \eqref{eq:similarmeasure} yields a matrix of rank at most $r$ to approximate the similarity matrix between $G_A$ and~$G_B$.}

In light of the developments in this paper, we address the following problem,
\begin{equation}\label{eq:similarmeasure_Mh}
    \min_{x\in\manifold_h}\ \ -\trace\kh{\phi(x)^\top \lanifold\circ\lanifold(\phi(x))},
\end{equation}
with the rank parameter $r$ and the associated manifold $\hanifold= \mathrm{S}_{\mathrm{F}}(m,n)$; and then map the solution $x\in\manifold_h$ to $\phi(x)\in\mbR^{m\times n}$ as the approximate similarity matrix.

We carry out RGD and RTR on the proposed $\manifold_h$, and implement the ``Iterative method" \cite{cason2013iterative} by ourselves as the compared method. We initialize $x_0\in\manifold_h$ for our methods and $X_0\in\mbR^{m\times n}$ for the ``Iterative method" such that $\phi(x_0)=X_0={{\bm 1}^{}_m {\bm 1}^\top_n}/{\|{\bm 1}^{}_m {\bm 1}^\top_n\|_\frob}$. The performance is assessed based on the relative errors, ${\|\phi(x)-X^*\|_\frob}/{\|X^*\|_\frob}$ and ${\|X-X^*\|_\frob}/{\|X^*\|_\frob}$, 
where $X^*$ is the ground-truth similarity matrix. The algorithm is terminated if the relative error achieves $10^{-6}$ or the iteration exceeds $200$. We consider two scenarios: the solution $X^*$ is low-rank or full-rank.

\emph{Test on low-rank solution.} We construct $G_A$ of $m=2000$ vertices, which form a single cycle with all the edges oriented in the same direction. The graph $G_B$ is generated following the \emph{binomial random graph model} where each edge is included in the graph with the probability $p=0.0005$. In this manner, the similarity matrix $X^*$ has the rank $r^*=1$, as proved by~\cite{blondel2004measuresimilarity}. We test the method with various rank parameters $r=1,10,50,100$. The results in \myfig~\ref{fig:cycle_graph} show that our method, which is proposed for solving the general problem~\eqref{eq:lowrank_orthinvar}, successfully recovers the true similarity matrix across all the settings, exhibiting robustness to the rank parameter. Furthermore, its performance is comparable to the ``Iterative method".

\begin{figure}[htbp]
\begin{center}
\begin{minipage}{0.47\textwidth}
    \centering
    \includegraphics[width=1\linewidth]{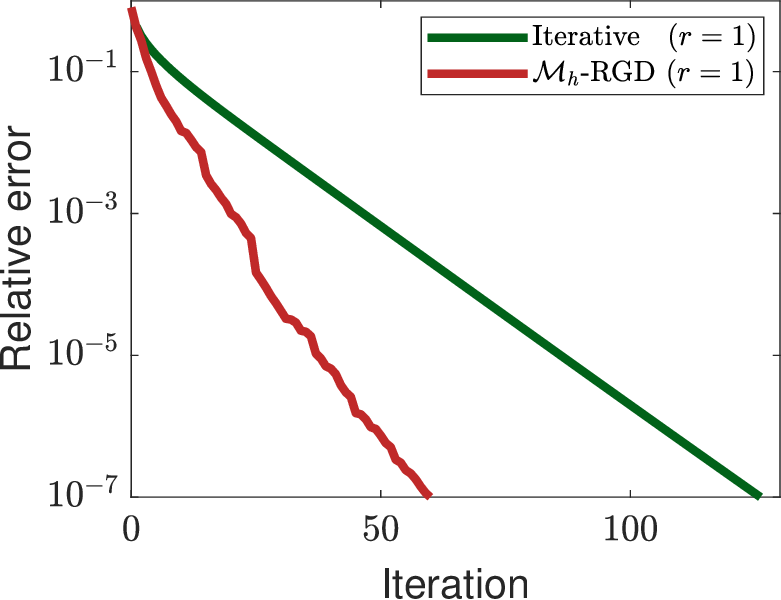}
\end{minipage}
\ \ 
\begin{minipage}{0.47\textwidth}
    \centering
    \includegraphics[width=1\linewidth]{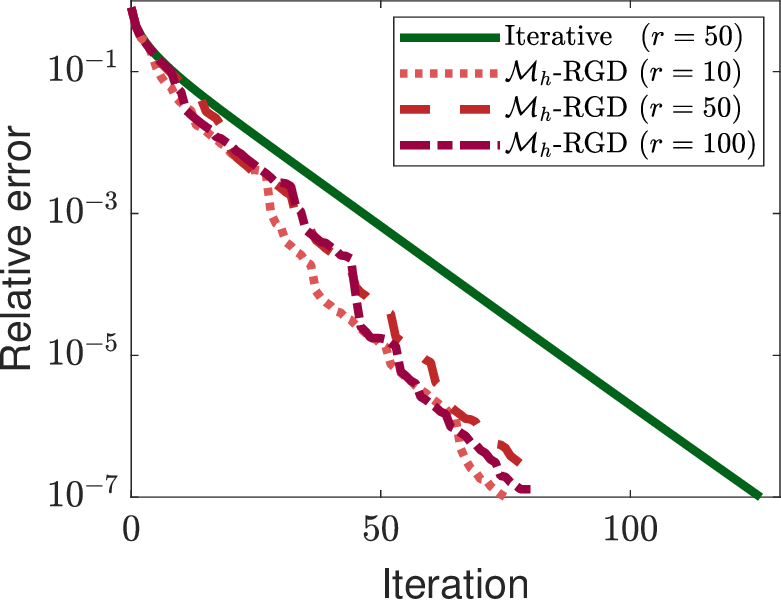}
\end{minipage}
\end{center}
\caption{Graph similarity measuring problem when the solution is low-rank. {\bf Left:} test with the unbiased rank parameter $r=r^*=1$; {\bf Right:} test with over-estimated rank parameters $r=10,50,100>r^*$.}\
\label{fig:cycle_graph}
\end{figure}

\emph{Test on full-rank solution.} We set $m=n$ and generate both $G_A$ and $G_B$ following the binomial random graph model with $p=10/m$, which means the average number of outgoing edges of a node is $10$. It is observed that the solution $X^*$ is full-rank. Numerical results of ``$\manifold_h$-RTR" for different problem sizes and rank parameters are presented in Table~\ref{tab:two_graph}. Specifically, given a small rank $r$, the method produces a decent low-rank approximation of $X^*$, and it recovers the true similarity matrix effectively as $r$ approaches the full rank.

\begin{table*}[htbp]
\caption{Graph similarity measuring problem when the solution is full-rank.}
\begin{center}
\setlength{\tabcolsep}{4pt}
\begin{tabular}{ccccccccc}
\toprule
\multicolumn{1}{l}{\multirow{2}{*}{Dimension $m$}} & \multicolumn{2}{c}{$r=10\%\times m$} & \multicolumn{2}{c}{$r=50\%\times m$} & \multicolumn{2}{c}{$r=75\%\times m$}& \multicolumn{2}{c}{$r= 100\%\times m$}\\ \cmidrule(l{2pt}r{2pt}){2-3} \cmidrule(l{2pt}r{2pt}){4-5} \cmidrule(l{2pt}r{2pt}){6-7} \cmidrule(l{2pt}r{2pt}){8-9}
& \multicolumn{1}{c}{Rel. err.} & \multicolumn{1}{c}{Iter.} & \multicolumn{1}{c}{Rel. err.} & \multicolumn{1}{c}{Iter.} & \multicolumn{1}{c}{Rel. err.} & \multicolumn{1}{c}{Iter.} & \multicolumn{1}{c}{Rel. err.} & \multicolumn{1}{c}{Iter.} \\
\midrule
$200$ & $\expnumber{7.82}{-7}$ & $38$ & $\expnumber{1.25}{-7}$ & $31$ & $\expnumber{4.76}{-8}$ & $26$ &$\expnumber{1.43}{-11}$ & $4$\\
$400$ & $\expnumber{8.23}{-7}$  & $32$ & $\expnumber{1.85}{-7}$ & $30$ & $\expnumber{5.92}{-8}$ & $26$ &$\expnumber{1.07}{-11}$ & $5$ \\
$600$ & $\expnumber{7.61}{-8}$ & $39$ & $\expnumber{5.73}{-7}$ & $27$ & $\expnumber{3.12}{-7}$ & $28$ &$\expnumber{8.71}{-12}$ & $5$\\
$800$ & $\expnumber{4.82}{-7}$  & $37$ & $\expnumber{5.79}{-7}$ & $28$ & $\expnumber{1.45}{-7}$ & $21$ &$\expnumber{9.31}{-12}$ & $5$ \\
$1000$ & $\expnumber{9.98}{-7}$ & $37$ & $\expnumber{7.13}{-7}$ & $31$ & $\expnumber{3.85}{-7}$ & $28$ &$\expnumber{8.02}{-12}$ & $5$\\
\bottomrule
\end{tabular}
\end{center}
\label{tab:two_graph}
\end{table*}

\subsection{Synchronization problem}\label{sec:LR_synchron}
Synchronization refers to the problem of determining absolute rotations in $\so(3):=\{R\in\orth(3):\,\det(R)=1\}$ with respect to a shared coordinate system, employing the relative rotation measurements. Concretely, given $n$ cameras and a set of relative rotations $\{\hat{R}_{ij}:(i,j)\in\mathscr{E}\}$, the element $\hat{R}_{ij}$ measures the orientation difference between the $i$-th and $j$-th cameras sharing overlapping view fields. The objective is to reconstruct the absolute rotations $\{R_i\}_{i=1}^n$ defining the individual orientations of the cameras such that $\hat{R}_{ij}\approx R_jR_i^\top$.

The synchronization problem enjoys an SDP relaxation \cite{wang2013rotationsynchronization}, and in view of the reformulation \eqref{eq:lowrankSDP} in this paper, we consider
\begin{equation}\label{eq:lowrank_Syn_Mh}
    \min_{x\in\manifold_h} \bar{f}(x)=\langle C, \phi(x)\phi(x)^\top \rangle,
\end{equation}
where the associated $\hanifold=\stiefel(3n,3)^n$, the rank parameter for $\manifold_h$ is $r=3$, and the measurement matrix $C\in\mbR^{3n\times 3n}$ is defined by $C_{ij}=\hat{R}_{ij}^\top$ for $(i,j)\in\mathscr{E}$, and $C_{ij}=0$ otherwise. To recover rotations $\{R_i\}_{i=1}^n$ from the solution $x\in\manifold_h$ of problem \eqref{eq:lowrank_Syn_Mh}, we utilize the representation $(H,V)$. In detail, noticing that $H=[H^{1}\,H^2\,\ldots\,H^n]^\top\in\stiefel(3n,3)$, which implies all the $3\times 3$ blocks $H^i$ are orthogonal, we can extract them as the desired rotations, i.e., letting $R_i=(H^i)^\top$, provided their determinants are positive. Note that $\so(3)$ is a connected component of $\orth(3)$ and the adopted retractions on $\manifold_h$ are continuous. Therefore, algorithms on $\manifold_h$ initialized with $H^i\in\so(3)$ $(i=1,2,\ldots,n)$ will always generate iterates $x_k$ capable of producing $n$ rotations.

The ``Stanford bunny" and the ``Spotted cow" datasets\footnote{Available from The Stanford 3D Scanning Repository at \url{https://graphics.stanford.edu/data/3Dscanrep/} and Keenan's 3D Model Repository at \url{https://www.cs.cmu.edu/~kmcrane/Projects/ModelRepository/}.} are employed for our test; see \myfig\ref{fig:bunny_cow}. In preparation, we randomly generate $360$ camera positions around the original mesh. For each camera, points in the visible portion are sampled to simulate point cloud scanning. Running the automatic Iterative Closest Point algorithm~\cite{rusinkiewicz2001ICP}, we obtain $1858$ and $1883$ relative rotations for ``Stanford bunny" and ``Spotted cow", respectively. The algorithm ``$\manifold_h$-RTR" with $\omega=10$ is carried out and it is terminated if the Riemannian gradient satisfies $\|\nabla_{\manifold_h} \bar{f}(x_k)\|\le 10^{-13}$. For the initial guess $x_0\in\manifold_h$ represented by $(H_0,V_0)$, $H_0=[H^{1}_0\,H^2_0\,\ldots\,H^n_0]^\top$ is generated by randomly sampling $H^i_0$ on $\so(3)$, and $V_0$ is randomly generated on~$\stiefel(3n,3)$.

The reconstructed point clouds of the two 3D models are visualized in \myfig\ref{fig:bunny_cow}, and the errors are quantified in \myfig\ref{fig:reconstructed_errors}. The results demonstrate a decent reconstruction quality both visually and numerically, validating the effectiveness of the proposed approach on $\manifold_h$ for solving problem \eqref{eq:lowrank_Syn_Mh}.

\begin{figure*}[htbp]
\vspace{-6mm}
\begin{center}
        \begin{minipage}{0.28\textwidth}
		\centering
		\includegraphics[width=1\linewidth]{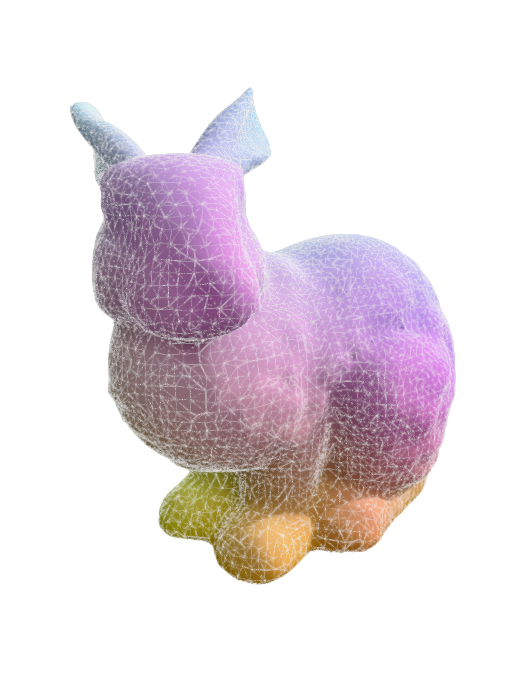}
	\end{minipage}
        \ \ \ \ 
	\begin{minipage}{0.35\textwidth}
		\centering
            \vspace{4mm}
		\includegraphics[width=1\linewidth]{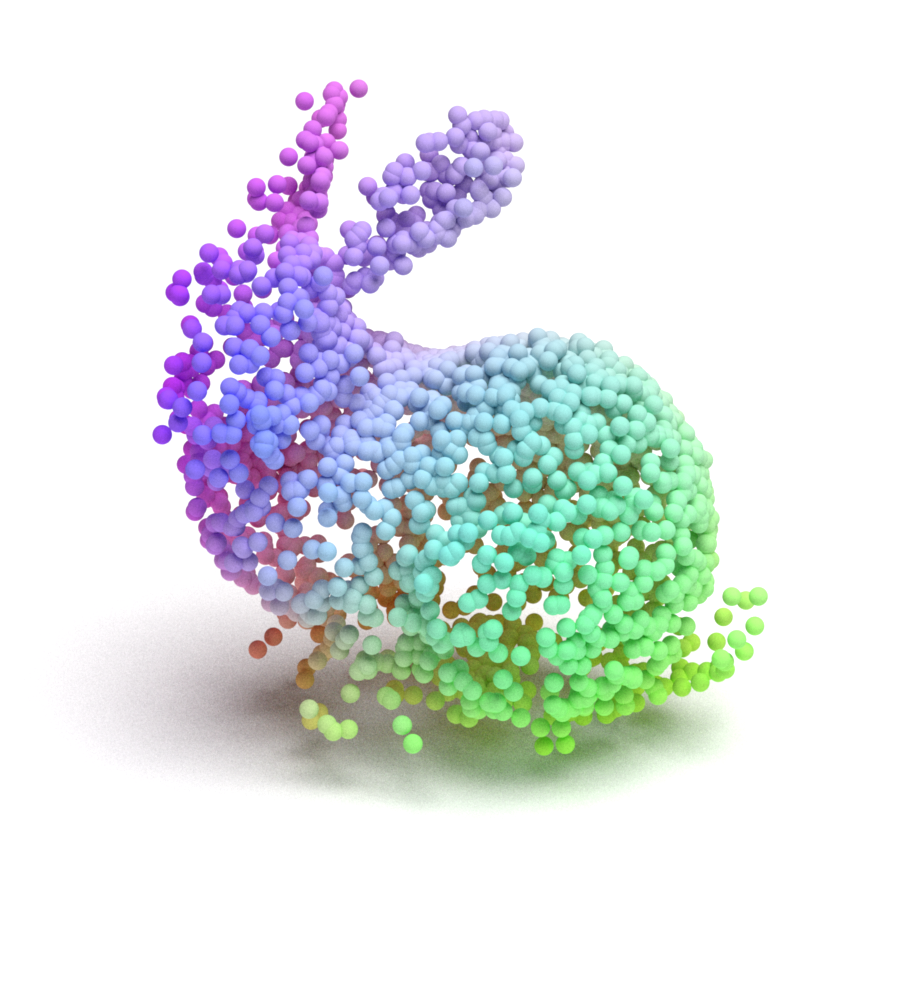}
	\end{minipage}
        \!\!\!\!\!\!\!\!
	\begin{minipage}{0.35\textwidth}
		\centering
            \vspace{3mm}
		\includegraphics[width=1\linewidth]{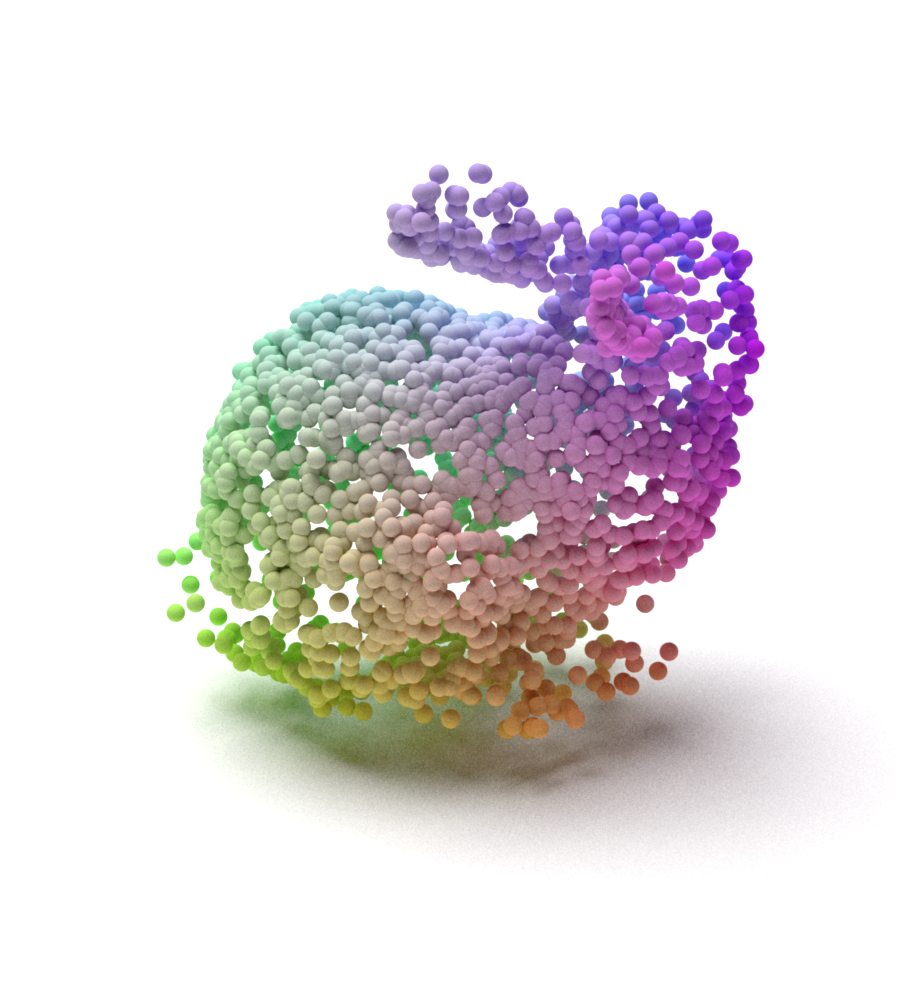}
	\end{minipage}
	\end{center}
        \vspace{-12mm}
        \begin{center}
        \begin{minipage}{0.29\textwidth}
		\centering
		\includegraphics[width=1\linewidth]{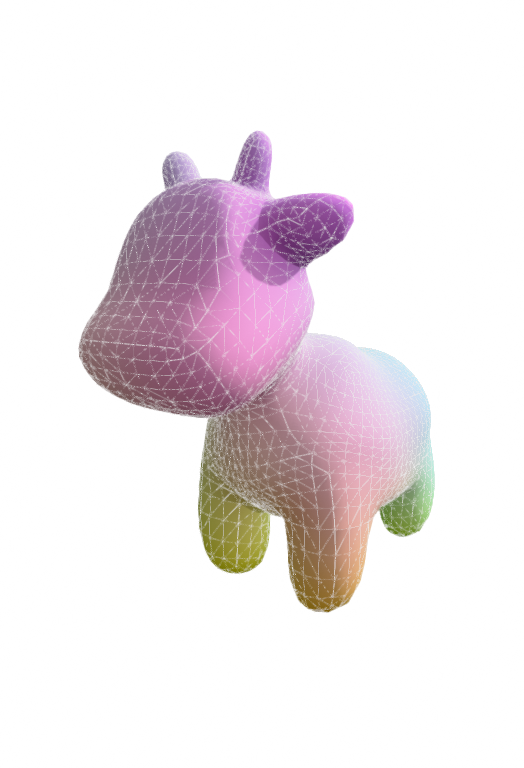}
	\end{minipage}
        \ \ \ \ 
	\begin{minipage}{0.33\textwidth}
		\centering
		\includegraphics[width=1\linewidth]{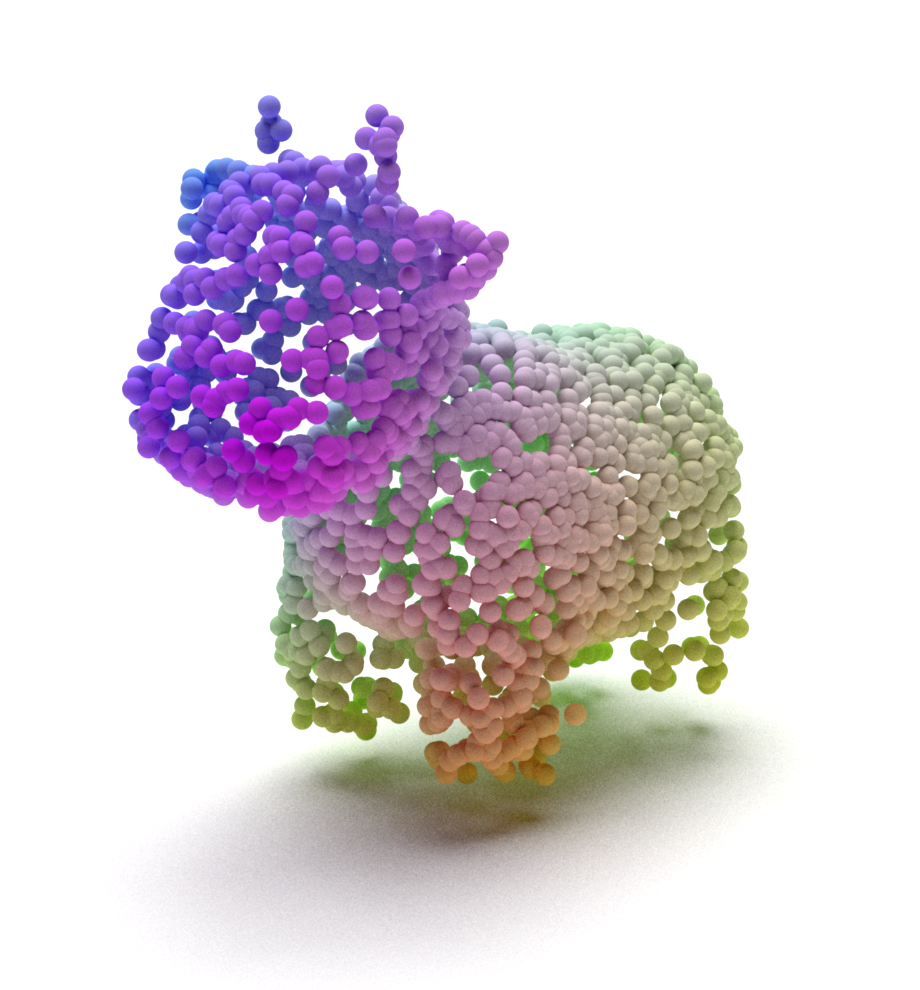}
	\end{minipage}
        \!\!\!\!
	\begin{minipage}{0.33\textwidth}
		\centering
            \vspace{3mm}
		\includegraphics[width=1\linewidth]{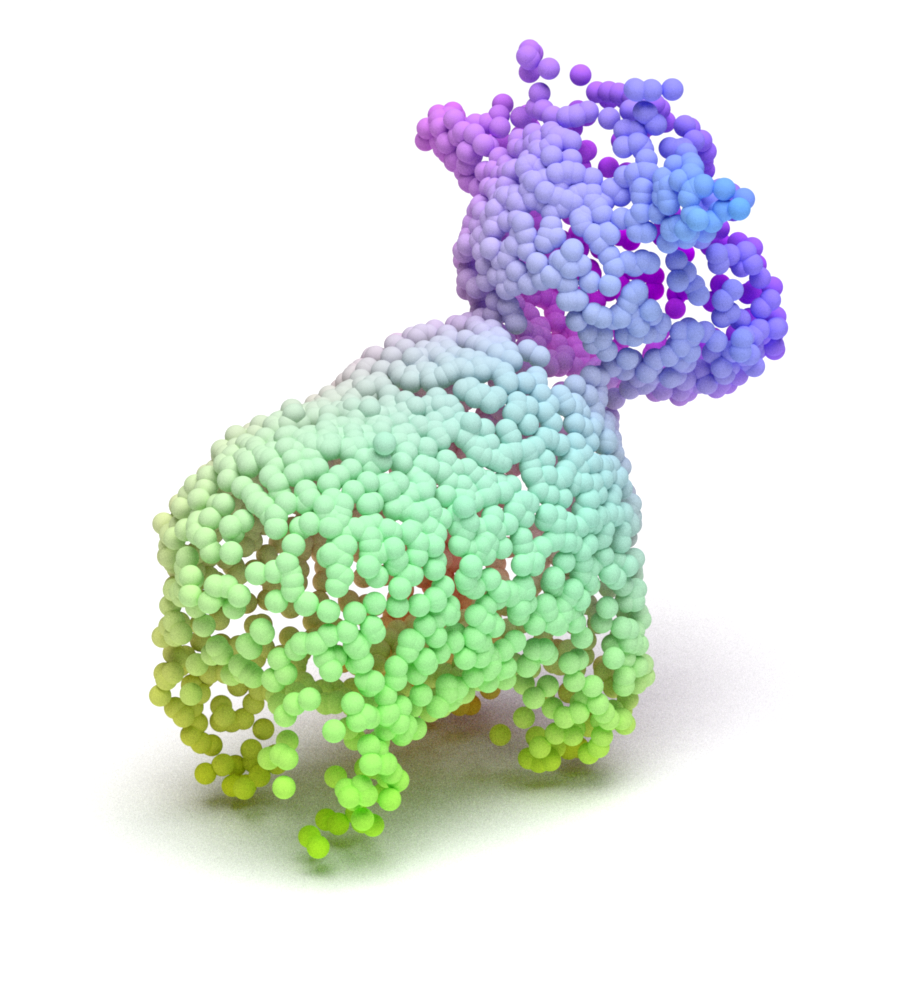}
	\end{minipage}
	\caption{Visualization of ``Stanford bunny" and ``Spotted cow" in the synchronization problem. \textbf{Left:} original meshes; \textbf{Middle:} front view of the reconstructed point clouds; \textbf{Right:} back view of the reconstructed point clouds.}
	\label{fig:bunny_cow}
\end{center}
\end{figure*}
\begin{figure*}[h]
    \vspace{-8mm}
    \begin{center}
	\begin{minipage}{0.43\textwidth}
		\centering
		\includegraphics[width=1\linewidth]{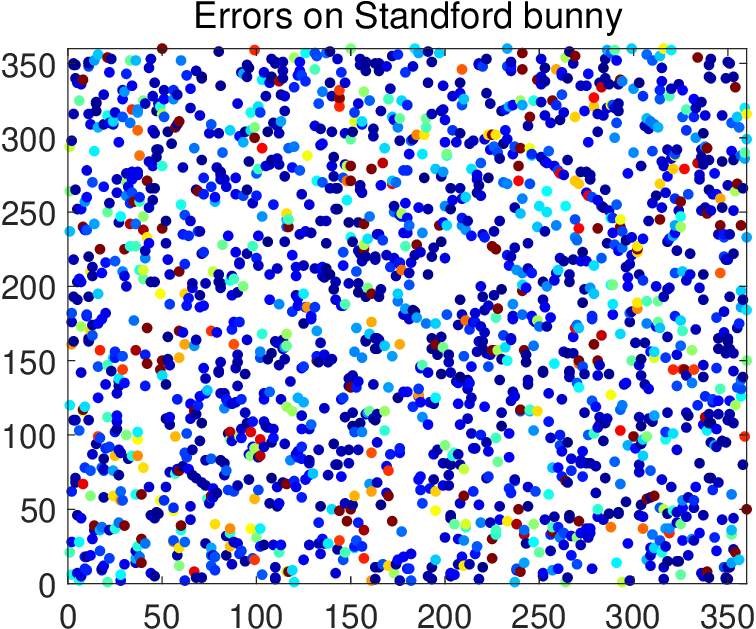}
	\end{minipage}
        \ \ \ 
	\begin{minipage}{0.43\textwidth}
		\centering
		\includegraphics[width=1\linewidth]{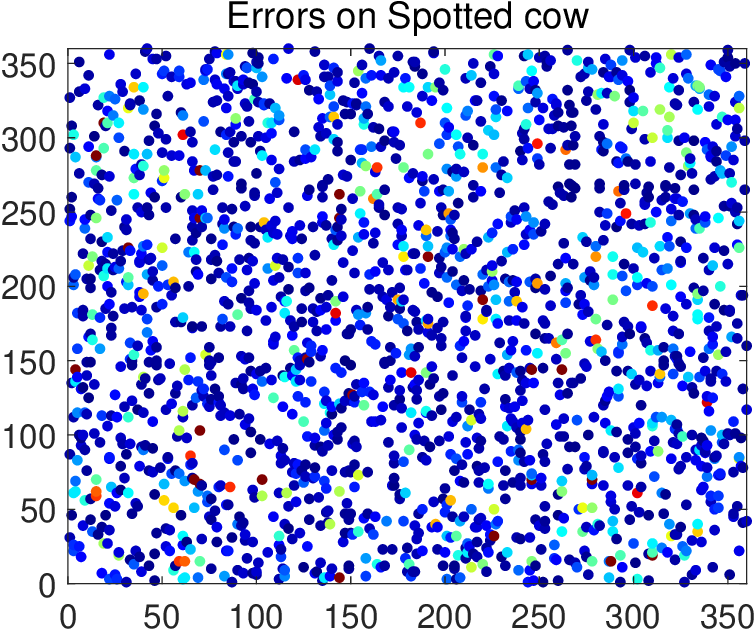}
	\end{minipage}
        \ \ \ 
	\begin{minipage}{0.08\textwidth}
		\centering
            \vspace{2.5mm}
		\includegraphics[width=1\linewidth]{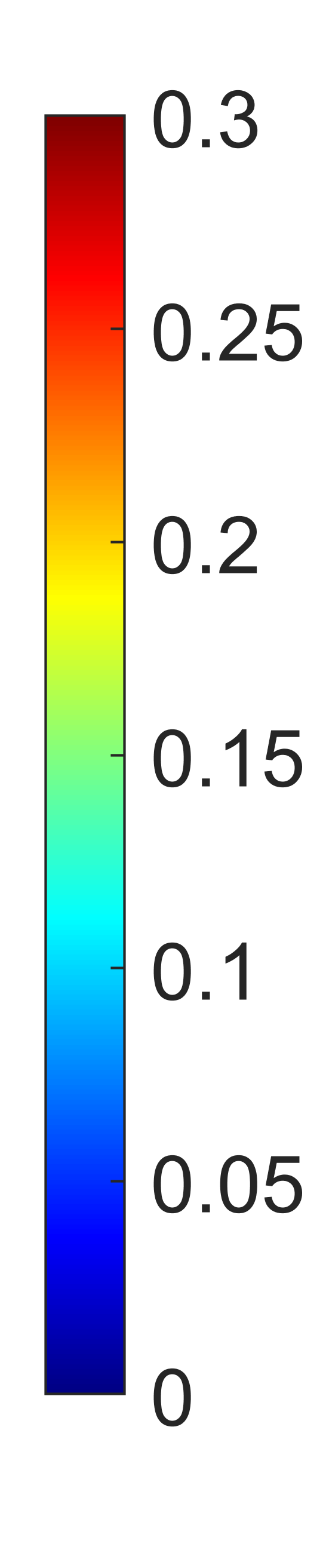}
	\end{minipage}
        \caption{Reconstructed errors in the synchronization problem, which are evaluated by $\|R_iR_j^\top-\hat{R}_{ij}\|_\frob$ for $(i,j)$ pairs. The measurements are sparse since only around $1800$ relative rotations are provided for $360$ three-dimensional scans.}
	\label{fig:reconstructed_errors}        
    \end{center}
    \vspace{-5mm}
\end{figure*} 

\newpage
\subsection{State compression of a city-wide Markov process}\label{sec:LR_Markov}
We consider the following problem with the associated $\hanifold=\oblique(\abs{\mdps},\abs{\mdps})$,
\begin{equation}\label{eq:markovaggregation_Mh}
    \min_{x\in\manifold_h}\ \ \frac{1}{2}\|\phi(x)\odot \phi(x)-\hat{\probmatrix}\|^2_\frob
\end{equation}
as a parameterized counterpart of \eqref{eq:markovaggregation}, to examine the low-rankness of the Manhattan transportation network. The experiment is based on a real-life dataset of $1.1\times10^7$ NYC Yellow cab trips in January 2016 \cite{TLC2017}, where each entry includes the pick-up and drop-off locations of one trip. We assume that taxi transitions are nearly memoryless and, therefore, characterize the transportation dynamics as a Markov process. In this view, the map is discretized such that locations in the same grid cell are aggregated into a single state, and the recorded trips are seen as transitions between states.

The grid is of the size $0.001^\circ\times 0.001^\circ$, and states falling outside the region of $80.07^\circ\mathrm{W}$ to $60.92^\circ\mathrm{W}$ and $30.69^\circ\mathrm{N}$ to $50.85^\circ\mathrm{N}$ or those with an occurrence frequency below $10^{-4}$ are excluded. Consequently, we collect $2017$ valid states $\mdps=\{1,2,\ldots,{2017}\}$, and about $7.5\times 10^6$ transition indexed by $\kh{s^{\text{pickup}}_t,s^{\text{dropoff}}_t}$ with $t=1,2,\ldots,T$. The empirical probability matrix is constructed by
\begin{equation*}
    \hat{\probmatrix}(i,j) = \frac{\sum_{t=1}^T \mathbb{I}\kh{s^{\text{pickup}}_t=i,s^{\text{dropoff}}_t=j}}{\sum_{t}^T \mathbb{I}\kh{s^{\text{pickup}}_t=i}},\ \ \text{for}\ \  i,j\in\mdps,
\end{equation*}
where $\mathbb{I}$ is the indicator function of an event, $1$ for happening and $0$ for otherwise.

We implement ``$\manifold_h$-RGD" with $\omega=0.5$ to address \eqref{eq:markovaggregation_Mh}, running for $K=500$ iterations to obtain $x_K$ represented by $(H_K,V_K)\in \oblique(m,r)\times\stiefel(n,r)$. Each row of $H_K$ is then treated as a feature vector corresponding to a state, and the MATLAB function \texttt{kmeans} is used to cluster $\mdps$ based on these features, where the number of clusters is set equal to the rank parameter $r$. In \myfig\ref{fig:Markov_lowrank}, the clustering results are visualized via Google Maps, showing that the partition based on our approach agrees well with the geometry of this region.

\begin{figure}[htbp]
\begin{center}
\begin{minipage}{0.23\textwidth}
    \centering
    \includegraphics[width=1\linewidth]{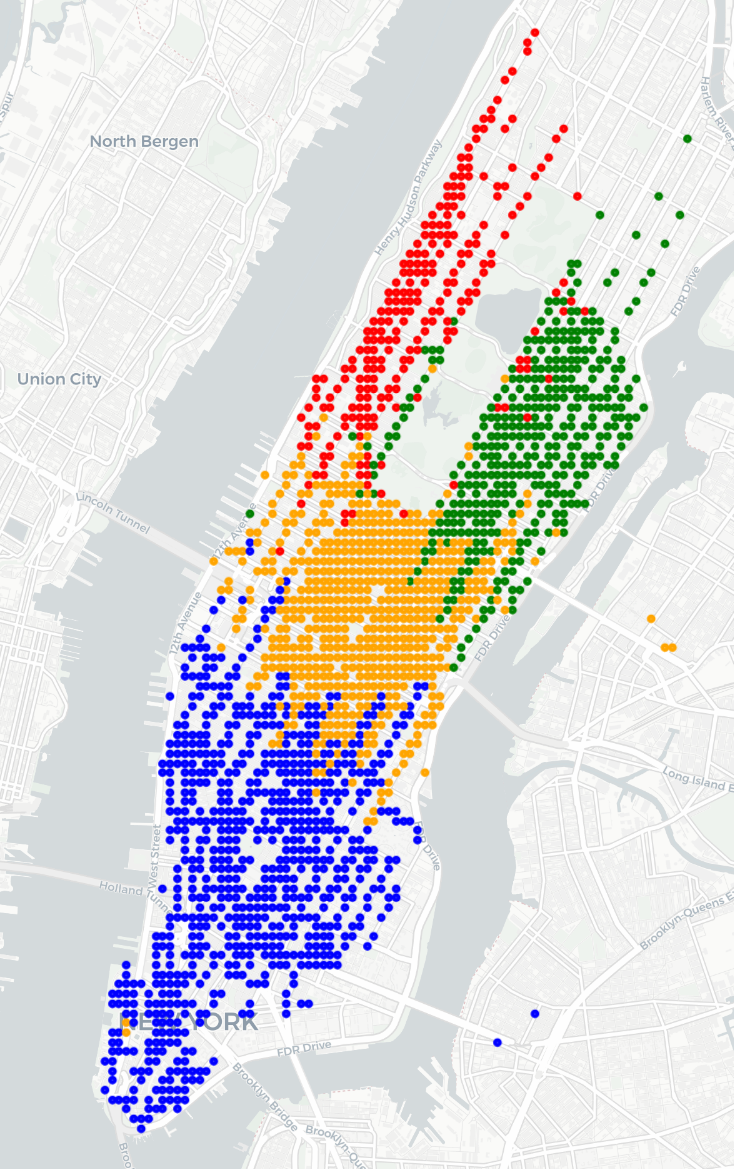}
    \vspace{-4mm}
    \begin{center}
        $r=4$
    \end{center}
\end{minipage}
\,
\begin{minipage}{0.23\textwidth}
    \centering
    \includegraphics[width=1\linewidth]{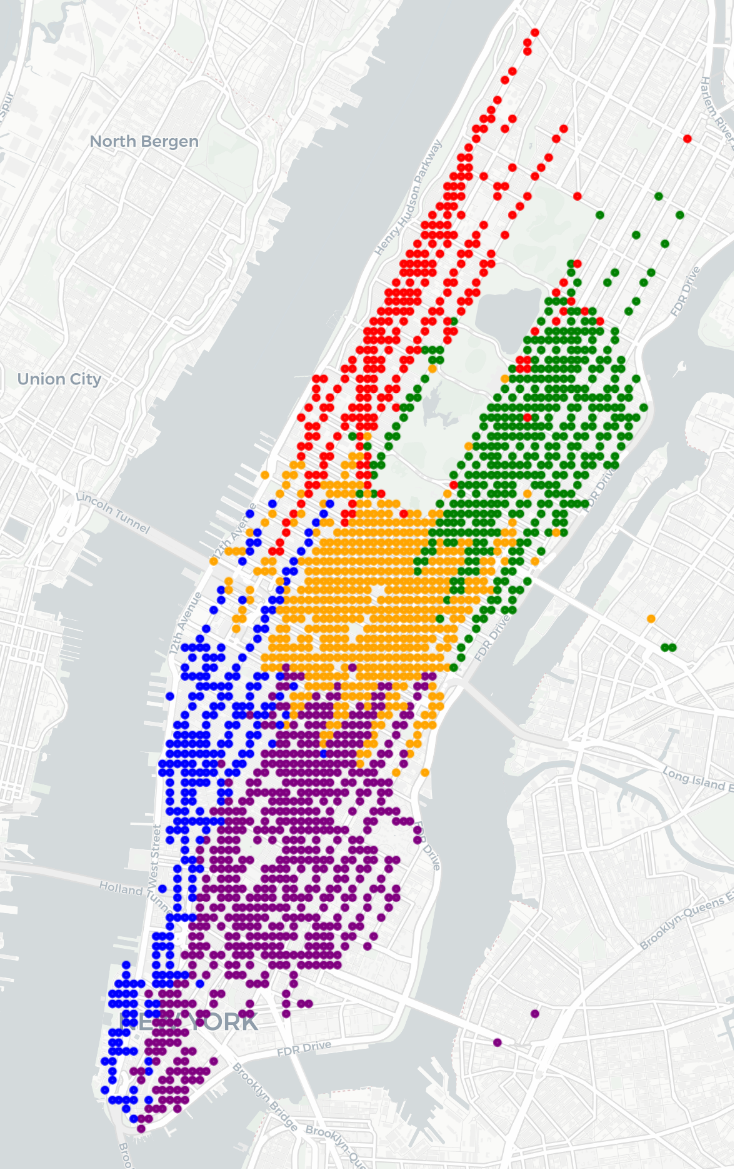}
    \vspace{-4mm}
    \begin{center}
        $r=5$
    \end{center}
\end{minipage}
\,
\begin{minipage}{0.23\textwidth}
    \centering
    \includegraphics[width=1\linewidth]{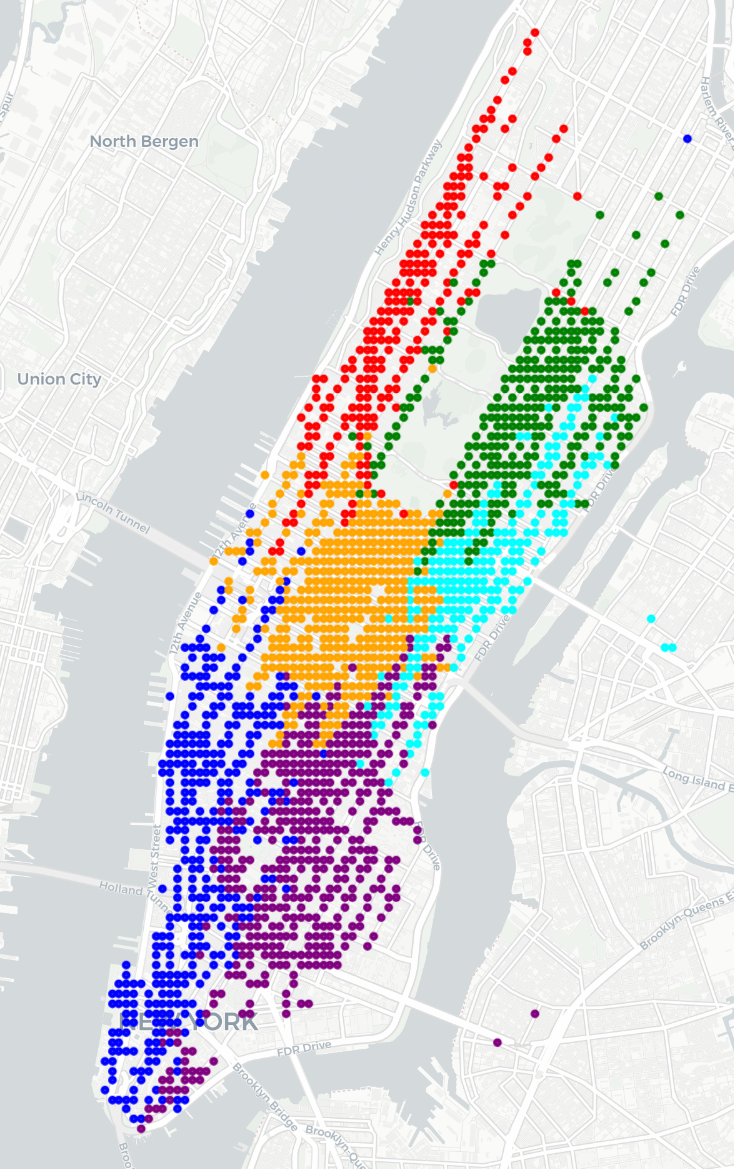}
    \vspace{-4mm}
    \begin{center}
        $r=6$
    \end{center}
\end{minipage}
\,
\begin{minipage}{0.23\textwidth}
    \centering
    \includegraphics[width=1\linewidth]{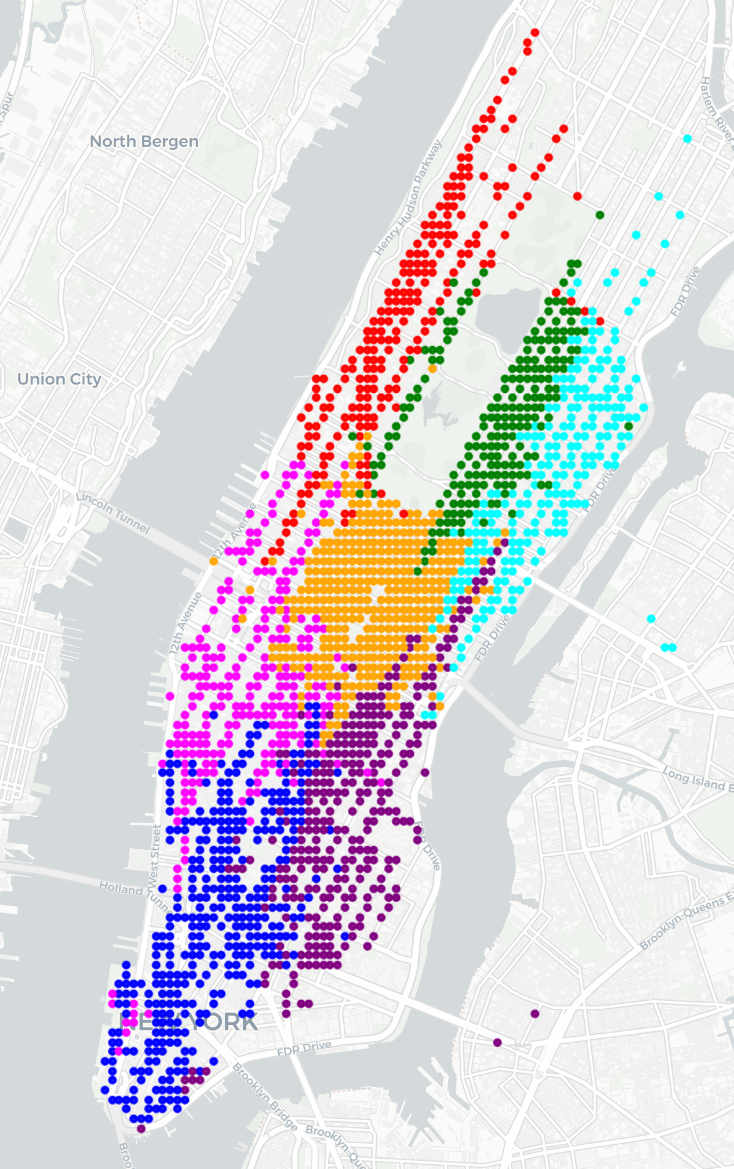}
    \vspace{-4mm}
    \begin{center}
        $r=7$
    \end{center}
\end{minipage}
\end{center}
\caption{State compression of the Markov process: partition of the Manhattan transportation network.}
\label{fig:Markov_lowrank}
\end{figure}

\subsection{Low-rank reinforcement learning}\label{sec:LRRL}
Recalling the proposed model \eqref{eq:lowrankRL}, we next turn to the following problem
\begin{equation}\label{eq:lowrankRL_Mh}
    \min_{x\in\manifold_h} -J(\pi(\phi(x))),
\end{equation}
which seeks an optimal low-rank policy. In this problem, $\manifold_h$ is the associated parameterization of $\mbR_{\le r}^{\abs{\mdps}\times\abs{\mdpa}}\cap\oblique(\abs{\mdps},\abs{\mdpa})$. The policy gradient $\nabla J$ can be derived via existing RL theory \cite{sutton2018RL}, and its Riemannian counterpart on $\manifold_h$ follows from the formula \eqref{eq:1st_Riegrad}. Implementing RGD to solve \eqref{eq:lowrankRL_Mh} is thus termed the \emph{low-rank Riemannian policy gradient} (LRRPG) method. 

We compare the proposed LRRPG with Q-learning and REINFORCE \cite{sutton2018RL}, both of which address RL tasks without exploiting low-rank structures in the environment. Specifically, REINFORCE updates the policy $\pi\in\mbR^{\abs{\mdps}\times\abs{\mdpa}}$ based on the policy gradient. Moreover, Q-learning maintains a variable $Q\in\mbR^{\abs{\mdps}\times\abs{\mdpa}}$ to estimate the optimal expected accumulative reward conditioned on state-action pairs, and then recover the policy by greedily picking the action with the maximal reward in each state. As summarized in Table~\ref{tab:RL_compare}, LRRPG requires far fewer parameters than the compared methods, thereby mitigating both the memory and computational burdens.

We test on two RL environments of the toolkit OpenAI Gym \cite{brockman2016Gym}---``pendulum" and ``mountain car"---illustrated in \myfig\ref{fig:illu_RL}. The evaluation is based on the \emph{timestep}, which counts the agent's interaction with the environment, and we report the cumulative reward and the episode's length as performance metrics. In the implementation, the learning rates for each algorithm are optimally chosen from $\{5\times10^i: i=-3,-2,-1,0,1\}$ in ``pendulum" and $\{3\times10^i: i=-4,-3,-2,-1,0\}$ in ``mountain car''. Further details and numerical results are discussed below.

\begin{figure*}[htbp]
\begin{center}
	\begin{minipage}{0.6\textwidth}
        \vspace{-2mm}
        \hspace{-8mm}
	\begin{tikzpicture}
            \pgfmathsetmacro{\ratio}{0.6/0.4} 
		
		\node[rounded corners=2mm, inner sep=5pt] (box) at (0, 0) {
            \begin{tikzpicture}
            \node (pen) at (0,0){\begin{minipage}{0.3\textwidth}
                \centering
                \includegraphics[width=1\linewidth]{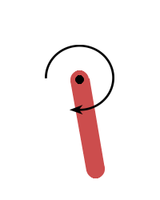}
            \end{minipage}
            };
            \node[right=1.5cm of pen] (car){
            \begin{minipage}{0.6\textwidth}
                \centering
                \includegraphics[width=1\linewidth]{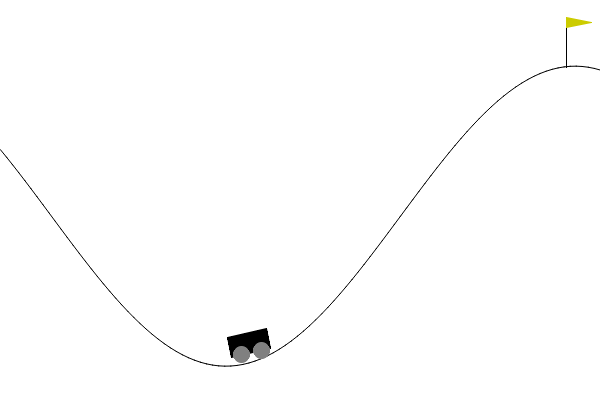}
            \end{minipage}
		};

            \node[above=-5pt of pen] {\textbf{Pendulum}};
            
            \node[above=-7pt of car] {\textbf{Mountain car}};

            \draw[-stealth] ([xshift=-\ratio*1.9cm, yshift=-\ratio*0.7cm]car.east) -- ([xshift=-\ratio*1.5cm, yshift=-\ratio*0.3cm]car.east);

            \end{tikzpicture}
            };
	\end{tikzpicture}
        \vspace{-5mm}
	\end{minipage}
	\caption{Illustration of RL environments. \boldt{Left:} a pendulum attached at one end to a fixed point, with the other end being free; \boldt{Right}: a car stuck in the sinusoidal valley, with the flag on top of the right hill being the destination.}
    \label{fig:illu_RL}
\end{center}
\end{figure*}
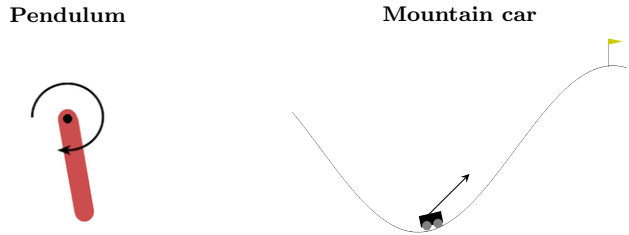

{\emph{Pendulum.}} In this scenario, the agent endeavors to keep an inverted pendulum upright by applying torque to its free end. The problem has two continuous state coordinates (angle and angular velocity) and one continuous action coordinate (torque), necessitating discretization of the state and action spaces. Specifically, by uniformly discretizing the Cartesian product of state coordinates, we generate a finite state space containing $\abs{\mdps}=2121$ elements. Similar discretization on the action coordinate yields $\abs{\mdpa}=41$. Moreover, a multi-objective $\text{reward}=-(\text{angle}^2+0.1\times\text{velocity}^2+0.001\times\text{torque}^2)$ is used, which promotes the pendulum to be upright and encourages small exerted action. The discount factor is set to $\gamma=0.9$. The episode is truncated if the pendulum deviates far away from the upright position, and the maximal episodic life is $100$ timesteps. We adopt the learning rate of $0.05$ for REINFORCE and Q-learning, and $0.005$ for LRRPG.

Parameter efficiency and convergence curves are reported in Table~\ref{tab:RL_compare} and \myfig\ref{fig:pendulum}, showing that in this case, LRRPG achieves the optimal reward and the maximal balanced time faster than REINFORCE. Additionally, the performance of LRRPG is on par with the baseline Q-learning. Increasing the rank enhances initial training, but the influence on the final result is small.

\begin{figure*}[htbp]
	\centering
        \begin{minipage}{0.9\textwidth}
		\centering
		\includegraphics[width=1\linewidth]{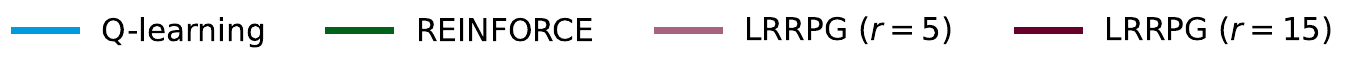}
	\end{minipage}
        \\[-2mm]
	\begin{minipage}{0.49\textwidth}
		\centering
		\includegraphics[width=1\linewidth]{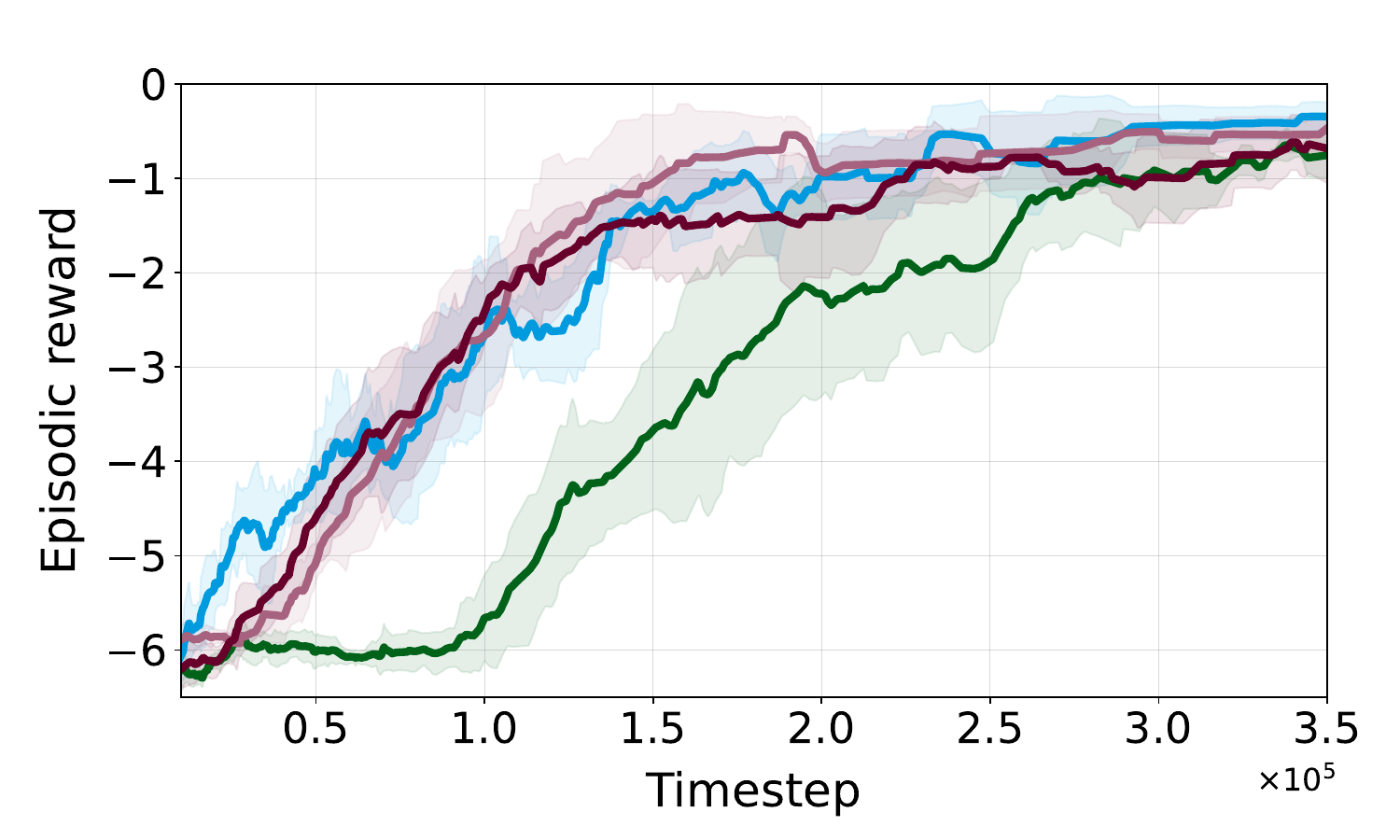}
	\end{minipage}
	\begin{minipage}{0.49\textwidth}
		\centering
		\includegraphics[width=1\linewidth]{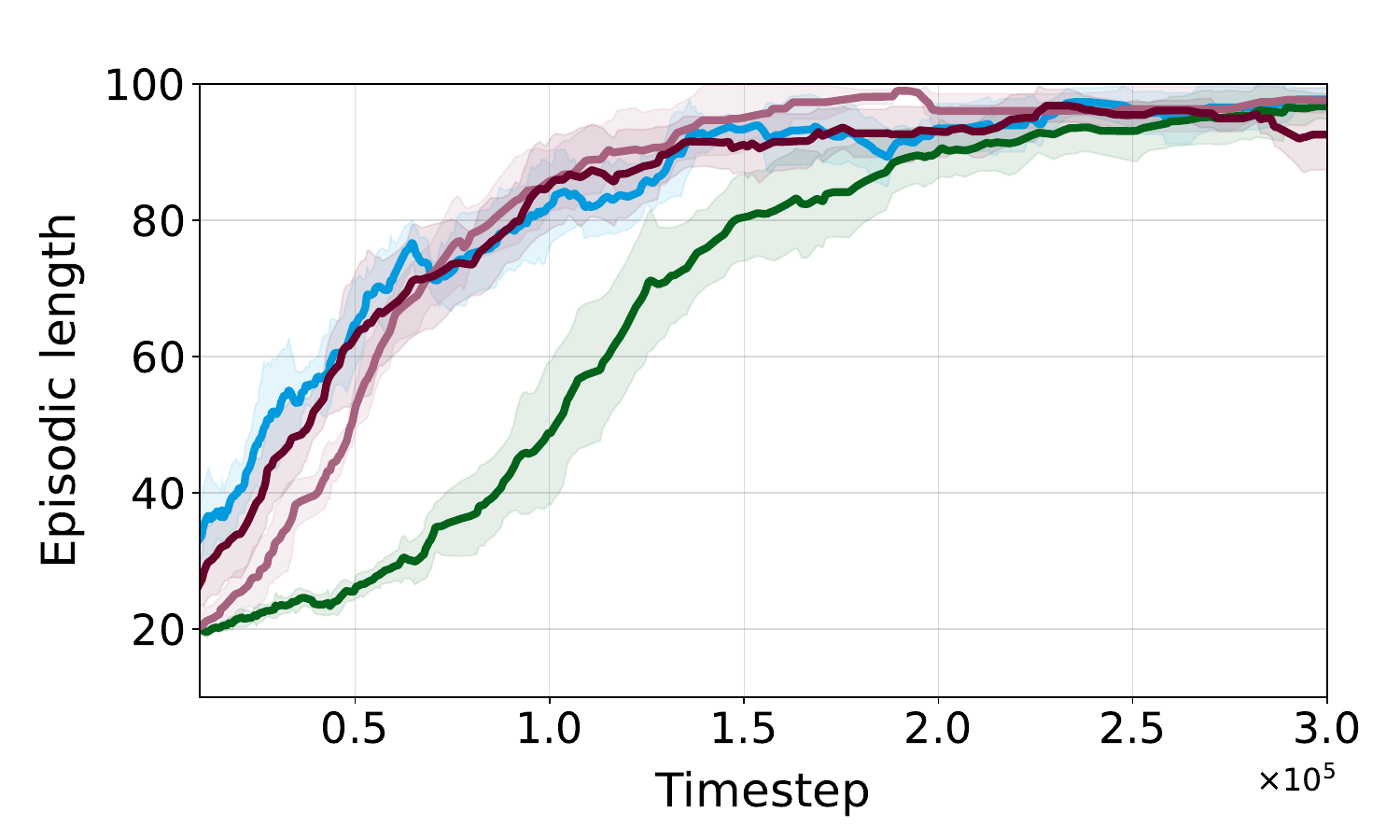}
	\end{minipage}
        \caption{Comparison of the RL algorithms in ``pendulum". The running average over 10 consecutive episodes is adopted for the presentation, and results are averaged over $10$ seeds. A~longer episodic length suggests that the agent is more proficient at balancing the pendulum.}
	\label{fig:pendulum}
\end{figure*}

\begin{table*}[htbp]
\caption{Parameter efficiency of the proposed method across two environments. The ratio of storage is defined as $\rho_{\mathrm{storage}}:=\text{Parameters}/\abs{\mdps}\abs{\mdpa}$. Episodic rewards are averaged over $10$ runs. The results are reported after $0.3$M and $12$M timesteps in the ``pendulum" and ``mountain car", respectively.}
\begin{center}
\setlength{\tabcolsep}{4pt}
\begin{tabular}{lcccccc}
\toprule
\multicolumn{1}{l}{\multirow{2}{*}{Algorithm}} & \multicolumn{3}{c}{Pendulum} & \multicolumn{3}{c}{Mountain car}\\ \cmidrule(l{2pt}r{2pt}){2-4} \cmidrule(l{2pt}r{2pt}){5-7}
& \multicolumn{1}{c}{Parameters} & \multicolumn{1}{c}{$\rho_{\mathrm{storage}}$} & \multicolumn{1}{c}{Reward} & \multicolumn{1}{c}{Parameters}& \multicolumn{1}{c}{$\rho_{\mathrm{storage}}$} & \multicolumn{1}{c}{Reward} \\
\midrule
REINFORCE   & $86,961$ & $1.00$ & $-0.94$ &  $582,096$ & $1.00$ & $96.99$ \\
Q-learning  & $86,961$ & $1.00$ & $-0.75$ &  $582,096$ & $1.00$ & $83.95$ \\
LRRPG ($r=5$)   & $10,810$ & $0.12$ & $-0.73$ &  $\phantom{0}15,485$ & $0.03$  & $88.31$ \\
LRRPG ($r=10$)   & $21,620$ & $0.25$ & $-0.58$ &  $\phantom{0}30,970$ & $0.05$ & $86.44$ \\
LRRPG ($r=15$)   & $32,430$ & $0.37$ & $-0.83$ &  $\phantom{0}46,455$ & $0.08$ & $78.23$ \\
\bottomrule
\end{tabular}
\end{center}
\label{tab:RL_compare}
\end{table*}

{\emph{{Mountain car.}}} In this environment, the decision process involves applying an appropriate directional force to help the car, initially stuck at the bottom of the valley, reach the goal state on top of the right mountain. A negative reward of $-0.1\times \text{force}^2$ is received at each timestep to penalize for taking actions of large magnitude. As an incentive, a positive reward of $+100$ will be added if the car reaches the goal. We discretize the original continuous state and action spaces to obtain finite ones with $\abs{\mdps}=2896$ and $\abs{\mdps}=201$. The discount factor is set to $\gamma=0.99$, and the maximal episodic life is $1000$ timesteps. The learning rates are chosen as $0.3$ for Q-learning, $0.03$ for REINFORCE, and $0.0003$ for LRRPG.

Although the low-rank method slightly lowers the final reward, it significantly improves storage efficiency while maintaining overall performance comparable to Q-learning and REINFORCE, as illustrated by \myfig\ref{fig:car}. Specifically, with ${r=5}$, LRRPG uses parameters amounting to only $2.66\%$ of the environment size $\abs{\mdps}\times\abs{\mdpa}$, yet it still learns a low-rank policy that successfully drives the car to the mountaintop; see Table~\ref{tab:RL_compare} for detailed comparison across different rank parameters.

\begin{figure*}[htbp]
	\centering
        \begin{minipage}{0.9\textwidth}
		\centering
		\includegraphics[width=1\linewidth]{fig/car_legend.pdf}
	\end{minipage}
        \\[-2mm]
	\begin{minipage}{0.49\textwidth}
		\centering
		\includegraphics[width=1\linewidth]{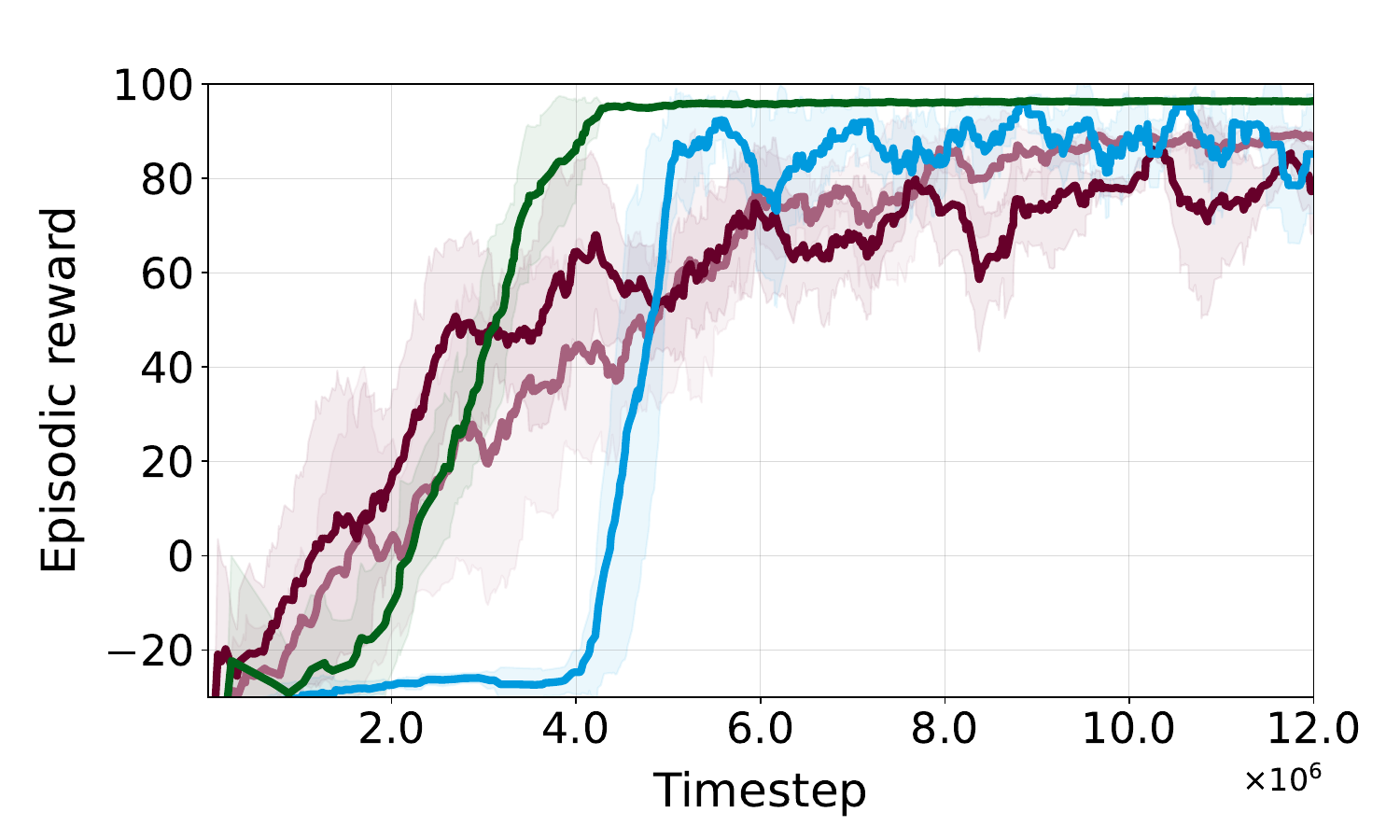}
	\end{minipage}
	\begin{minipage}{0.49\textwidth}
		\centering
		\includegraphics[width=1\linewidth]{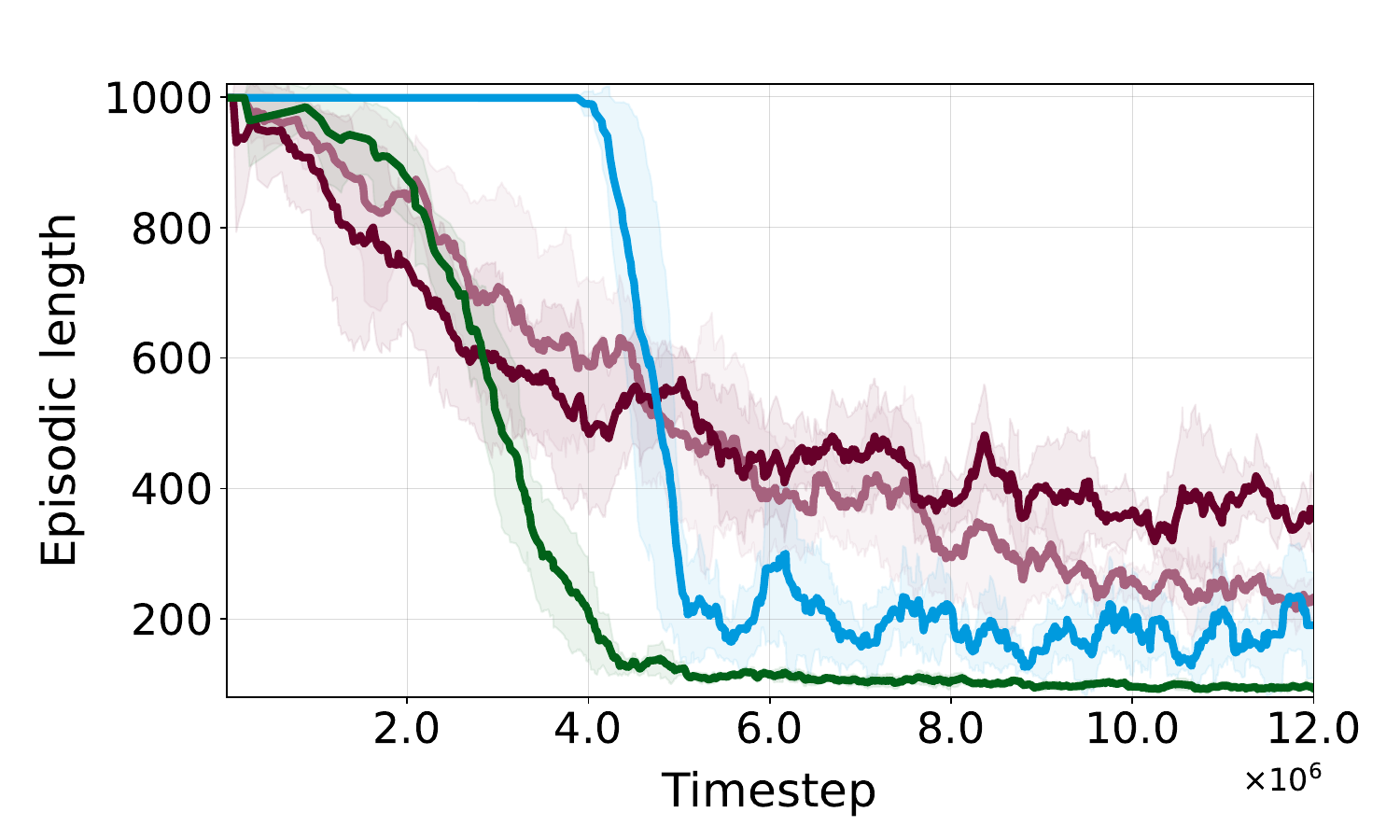}
	\end{minipage}
        \caption{Comparison of the RL algorithms in ``mountain car". The running average over 10 consecutive episodes is adopted for the presentation, and results are averaged over $10$ seeds. Shorter episodes mean the agent drives the car to its destination faster.}
	\label{fig:car}
\end{figure*}

\subsection{Low-rank neural network with weight normalization} \label{sec:LRNN}
To validate the effectiveness of combining the weight normalization technique (WN) and low-rank compression (LR) on neural networks, we solve the following parameterized version of \eqref{eq:lowrankNN},
\begin{equation}\label{eq:lowrankNN_Mh}
    \min_{x_i\in\manifold_{h_i},\,g_i}\ f(\{\phi(x_i),g_i\}_{i=1}^l).
\end{equation}
The proposed approach is tested on two benchmark classification datasets, MNIST \cite{lecun1998mnist} and CIFAR-10 \cite{krizhevsky2009CIFAR}.

\begin{table*}[htbp]
    \setlength{\tabcolsep}{5.5pt}
    \centering
    {
    \caption{Neural network training with different parameterizations.}
    \label{tab:NN_models}
    \begin{tabular}{lcccll}
        \toprule
        { Model\!\!\!} & {\!\!\!Weight norm.\!\!\!} & {\!\!Low rank\!} & \!Net. param.\! & \!Search space\! & Remark\\
        \midrule
        Vanilla\!\!& - & - & $\{W_i\}$ & $\mbR^{m_i\times n_i}$ & \ \ \ \ \ - \\
        WN & \checkmark & - & $\{X_i,g_i\}$ & $\oblique(m_i,n_i)\!\times\!\mbR^{m_i}$& \ \ \ \ \ -\\
        LR & - & \checkmark & $\{x_i,g_i\}$ & $\manifold_{h_i}\!\times\!\mbR^{m_i}\!\!\!$ & $h_i(\cdot)\equiv 0$\\
        WN+LR\!\!\! & \checkmark & \checkmark & $\{x_i,g_i\}$ & $\manifold_{h_i}\times \mbR^{m_i}$ & $h_i(X_i)=X_i X_i^\top\!-\!\mathbf{1}$
        \\
        \bottomrule	
    \end{tabular}
    }
\end{table*}

To assess the respective contributions of WN and LR, as well as their combined effect, we conduct an ablation study that compares the baseline neural network, the network with WN, the network with LR, and the network incorporating both WN and LR. These four models employ different parameterizations of the neural network, as presented in Table~\ref{tab:NN_models}.  Note that if only the low-rank compression is imposed, the model resembles \eqref{eq:lowrankNN_Mh}, but the search space simplifies to $\manifold_{h_i}$ with the associated mapping $h_i(\cdot)\equiv 0$. In the implementation, we use the stochastic (Riemannian) gradient descent method in the corresponding search space, assisted with an appropriate momentum. Implementation details and hyper-parameter specifications are provided in Appendix~\ref{app:network}.

Numerical outcomes are reported in \myfig\ref{fig:MNIST} and Table~\ref{tab:cifar10}, where the metric \emph{flops} is defined as floating point operations. Regarding the fully parameterized models, the results verify that imposing WN enhances the performance over the vanilla network. As a counterpart, for the low-rank models, ``WN+LR" achieves faster convergence on MNIST and delivers higher accuracy on CIFAR-10 than ``LR". Moreover, as shown in Table~\ref{tab:cifar10}, although the proposed ``WN+LR" trades a small amount of accuracy when classifying CIFAR-10, it offers considerable parameter efficiency and inference acceleration compared to fully parameterized models. In conclusion, the findings reveal that combining weight normalization and low-rank compression presents a promising direction for designing neural networks. 

\begin{figure*}[htbp]
	\centering
        \begin{minipage}{0.9\textwidth}
		\centering
		\includegraphics[width=1\linewidth]{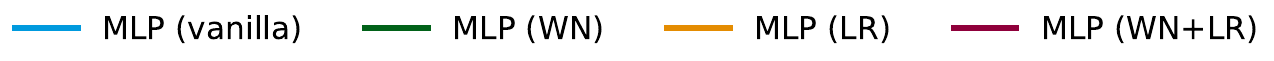}
	\end{minipage}
        \\[1mm]
	\begin{minipage}{0.48\textwidth}
		\centering
		\includegraphics[width=1\linewidth]{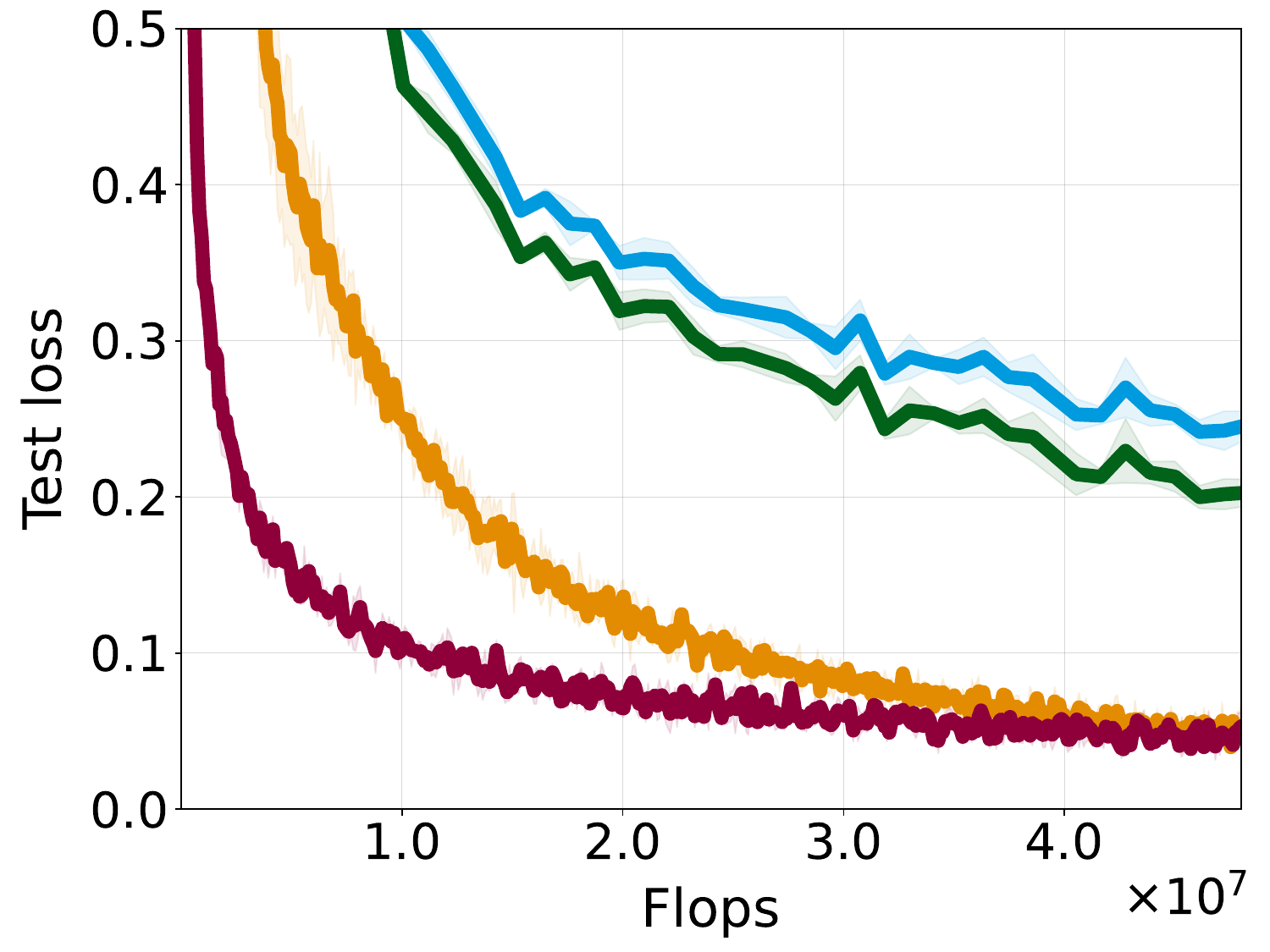}
	\end{minipage}
        \ 
	\begin{minipage}{0.48\textwidth}
		\centering
		\includegraphics[width=1\linewidth]{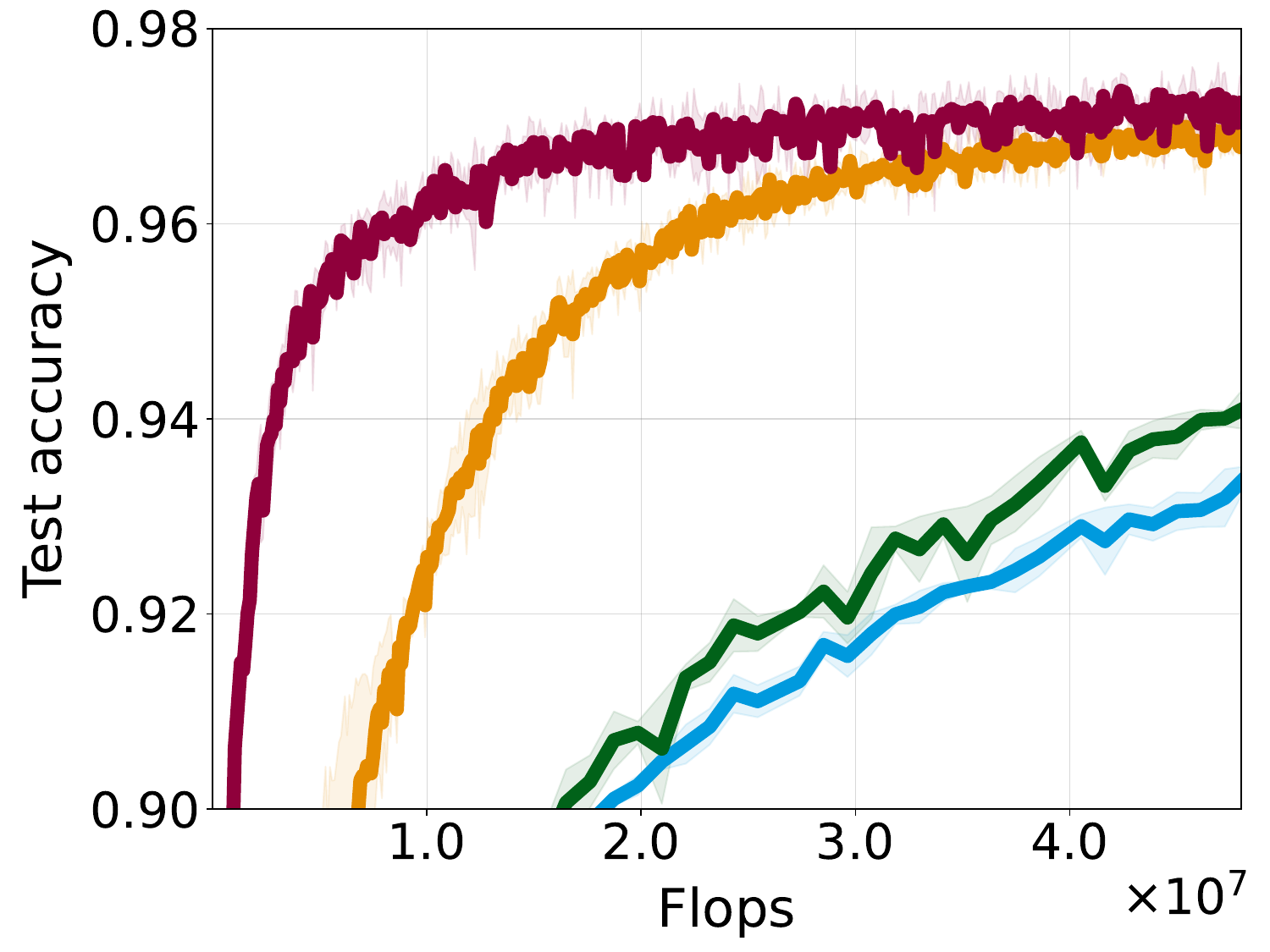}
	\end{minipage}
	\caption{Comparison of the network models on MNIST classification.}
	\label{fig:MNIST}
\end{figure*}

\begin{table*}[htbp]
    \setlength{\tabcolsep}{3.5pt}
    \centering
    {
    \caption{Comparison among the network models on CIFAR-10 classification. Reported metrics include the number of parameters, storage ratio, flops per inference and their ratio, and the top-1 training/test accuracy.}
    \label{tab:cifar10}
    \begin{tabular}{lcccccc}
        \toprule
        { Algorithm } & {Parameters} & { $\rho_{\mathrm{storage}}$ } & {Infer. flops} & {$\rho_{\mathrm{infer}}$} & {Train. acc. ($\%$)} & {Test acc. ($\%$)} \\
        \midrule
        CNN (vanilla) & $2.62$M  & $1.00$ & $565.50$M & $1.00$ & $\phantom{0}96.67$ & $87.80$ \\
        CNN (WN) & $2.62$M   & $1.00$ & $565.50$M & $1.00$ & $100.00$ & $90.31$ \\
        CNN (LR) & $0.38$M  & $0.14$ & $179.02$M & $0.32$ & $\phantom{0}90.01$ & $83.07$ \\
        CNN (WN+LR) & $0.38$M  & $0.14$ & $179.02$M & $0.32$ & $\phantom{0}95.39$ & $87.30$ \\
        \bottomrule	
    \end{tabular}
    }
\end{table*}

\section{Conclusions and perspectives}\label{sec:conclusion}
In this paper, we propose a space-decoupling framework for optimization problems on bounded-rank matrices with orthogonally invariant constraints. Specifically, we identify the tangent and normal cones of the coupled feasible set. Then, a smooth space-decoupling parameterization $\manifold_h$ is introduced, which reformulates the original nonsmooth problem as a smooth Riemannian optimization problem. Geometric tools for implementing Riemannian algorithms are developed, and our analysis demonstrates the equivalence between the reformulated and original problems. We conclude by offering several comments and outlining potential future directions inspired by this work.

{\emph{Generalization from right to left invariance.}} The proposed space-decoupling framework rests on the right orthogonally invariant property of $h$, i.e., $h(XQ)=h(X)$ for all $Q\in\orth(n)$. In essence, all results can be generalized to the case when $h(QX)=h(X)$ following from a parallel analysis. For example, the focused decomposition shifts from $X = HV^\top$ to $X = UH$, and the constraint in the definition of $\manifold_h$ (see~\eqref{eq:spacedecoupling_para}) is reversed from $XG=0$ to $GX = 0$.

{\emph{Discussion on the blanket assumption.}} The smoothness assumption for the mapping $h:\,\mathbb{R}^{m\times n}\rightarrow \mathbb{R}^{q}$ can be relaxed to twice continuous differentiability, while all the corresponding results in this paper remain valid. In fact, the full-rank condition in Assumption~\ref{assu:h} corresponds to the {linear independence constraint qualification} (LICQ) for the constraint $h(X)=0$.

{\emph{Another parameterization for the feasible region.}} It is worth noting that by defining the mapping $\psi: \mbR^{m\times r}\times \mbR^{n\times r}:(H,V)\mapsto HV^\top$, we can derive an alternative parameterization $(\hanifold^r\times \stiefel(n,r),\psi)$ in the sense that $\psi(\hanifold^r\times \stiefel(n,r))=\boundedrank\cap\hanifold$, and thus obtain an reformulated problem on $\hanifold^r\times \stiefel(n,r)$. Essentially, the equivalence between this Riemannian problem with the original problem~\eqref{eq:lowrank_orthinvar} has been answered by the proof of Theorem~\ref{pro:11_21_Mh}---similar ``$k\Rightarrow 1$" ($k=1,2$) properties hold true. 
The construction of geometric tools on the new parameterization mirrors the study of $\manifold_h$, which builds on the theoretical results developed in sections~\ref{sec:H} and~\ref{sec:OptCondition}.

{\emph{Extension of geometric methods on the bounded-rank variety.}} Theorem~\ref{the:tangent_intersection_rule} provides a closed-form characterization for the tangent cone to $\boundedrank\cap\hanifold$, which exhibits a parallel structure as that to $\boundedrank$. Notably, $\tangent_X\boundedrank$ plays a pivotal role in existing retraction-free algorithms \cite{schneider2015Lojaconvergence,olikier2023RFDR,olikier2024ERFDR} and the rank-adaptive mechanism \cite{gao2022Rieadap} for optimization on the bounded-rank set. Therefore, these techniques have the potential to be adapted to tackle problem~\eqref{eq:lowrank_orthinvar}, based on the results developed in this work.

\appendix
\section{Riemannian Hessian on $\manifold_h$}\label{app: Riemannian_hessian}
In this section, we present the computation of the Riemannian Hessian on $\manifold_h$. The following {corollary} will facilitate the computations in Proposition~\ref{pro:2rd_hessian}.
\begin{corollary}
    Given $X\in\mbR^{m\times n}$ with a decomposition $X=HV^\top$ where $H \in \mbR^{m\times s}$ and $V \in \stiefel(n, s)$, it holds that, for any $E\in\mbR^{m\times n}$ and $Y\in\mbR^{m\times s}$,
     \begin{align}
          \projection_{\tangent_X\hanifold}(EVV^\top)&=\projection_{\tangent_X\hanifold}(E)VV^\top,  \label{eq:projection_VVT}
          \\
           \projection_{\tangent_X\hanifold}(YV^\top)&=\zkh{\projection_{\tangent_H\hanifold^s}\kh{Y}}V^\top. \label{eq:projection_transform}
     \end{align}
\end{corollary}
\begin{proof}
    Applying \eqref{eq:proj_TXH} gives the first equality. Moreover, taking $E=YV^\top$ in \eqref{eq:proj_TXH} confirms the second equality.
\end{proof}

According to the computation of the Riemannian Hessian in~\eqref{eq:RieHess_2terms}, with $(H,V)$ representing $(X,G)\in\manifold_h$, we outline the subsequent quantities associated with $X$,
\begin{equation}\label{eq:4a}
    \begin{array}{ll}
     a_0(X)= \kh{2\omega I+ XX^\top}^{-1}, &\ \  a_1(X)= XX^\top a_0(X),
     \\
     a_2(X)= a_0(X)X,    &\ \   a_3(X)=X^\top X(2\omega I+X^\top X)^{-1}.
\end{array}
\end{equation}
Equivalently, these quantities can be computed via the representation $(H,V)$.
\begin{equation}\label{eq:4a_express}
    \begin{array}{ll}
     a_0(X):= \kh{2\omega I+ HH^\top}^{-1}, &\ \ a_1(X):= HM_{H,\omega}^{-1}H^\top,
     \\
     a_2(X):= HM_{H,\omega}^{-1}V^\top,    &\ \   a_3(X):=VM_{H,\omega}^{-1}V^\top+\frac{1}{2\omega}(I-VV^\top).
\end{array}
\end{equation}
To give $\nabla^2_{\manifold_h} \bar{f}(X,G)[\eta,\zeta]$ explicitly, we need to identify $\mathfrak{D}(E,Z)$ as the first step. 

\begin{lemma}\label{lem:compute_D}
    Given $(X,G)\in\manifold_h$ and $(\eta,\zeta)\in\tangent_{(X,G)}\manifold_h$ with representations $(H,V)$ and $(K,V_p)$, respectively, we denote $\projection_{(X,G)}(E,Z)=(P_1,P_2)$ and $\mathfrak{D}(E,Z)=(D_1,D_2)$. Then, they can be computed as follows,
    \[
        \begin{cases}
            P_1 = \projection_{\tangent_X\hanifold}\kh{E(I-G)} + a_1(X)EG - 2\omega a_2(X)ZG,
    	\\
            P_2 = 2\symac\kh{-GE^\top a_2(X) + 2\omega GZa_3(X) - GZG},
        \end{cases}
    \]
    \[
        \begin{cases}
            D_1 = \diff \kh{\tilde{X}\mapsto \projection_{\tangent_{\tilde{X}}\hanifold}(E(I-G))}(X)[\eta] - \projection_{\tangent_X\hanifold}(E\zeta) 
            \\
            \quad \quad\,\,\,\,+\,\diff a_1(X)[\eta]EG + a_1(X)E\zeta - 2\omega \diff a_2(X)[\eta]ZG - 2\omega a_2(X)Z\zeta,
            \\
            D_2 = 2\symac\big(-\zeta E^\top a_2(X)- GE^\top\diff a_2(X)[\eta]+2\omega \zeta Za_3(X)  
            \\
            \quad\quad\quad\quad\quad\ \ + 2\omega GZ\diff a_3(X)[\eta] - \zeta ZG - GZ\zeta\big).
        \end{cases} 
    \]
\end{lemma}
\begin{proof}
    Incorporate the projection \eqref{eq:projection_eq} into the formula~\eqref{eq:represent_eta},
    \begin{align*}
        P_1 =&\ \projection_{\tangent_H\hanifold^r}(EV)V^\top + HM_{H,\omega}^{-1}(H^\top E-2\omega V^\top Z)G
        \\
        \overset{(i)}{=}&\ \projection_{\tangent_X\hanifold}\kh{E(I-G)} + a_1(X)EG - 2\omega a_2(X)ZG,
        \\
        P_2 =&-2\symac\kh{G(E^\top H -2\omega ZV)M_{H,\omega}^{-1}V^\top}
        \\
        \overset{(ii)}{=}&\ 2\symac\kh{-GE^\top a_2(X) + 2\omega GZa_3(X) - GZG},
    \end{align*}
    where $(i),(ii)$ come from \eqref{eq:projection_transform} and \eqref{eq:4a_express}. Subsequently, taking the derivatives of $P_i$ with respect to $(X,G)$ along $(\eta,\zeta)$ yields $D_i\,(i=1,2)$.
\end{proof}

Finally, we give the proof of Proposition~\ref{pro:2rd_hessian} as follows.

\begin{proof}
    The rule \eqref{eq:RieHess_2terms} guides us to calculate $\mathfrak{D}(\nabla \bar{f}(X,G))$ and $\nabla^2 \bar{f}(X,G)[\eta,\zeta]$, and then project the sum of them onto $\tangent_{(X,G)}\manifold_h$ to derive $\nabla^2_{\manifold_h} \bar{f}(X,G)[\eta,\zeta]$. To begin with, we conduct the following computations,
    \begin{equation}\label{eq:preparation}
        \begin{aligned}
            \diff a_0(X)[\eta] =& -a_0(X)(\eta X^\top + X\eta^\top)a_0(X),
            \\
            \diff a_1(X)[\eta] =& \kh{\eta X^\top + X\eta^\top}a_0(X) + XX^\top \diff a_0(X)[\eta],
            \\
            \diff a_2(X)[\eta] =&\ \diff a_0(X)[\eta]X + a_0(X)\eta.
        \end{aligned}
    \end{equation}
    Taking $(E,Z)=\nabla \bar{f}(X,G)=(\nabla f(X), 0)$ in Lemma~\ref{lem:compute_D}, we obtain $\mathfrak{D}(\nabla \bar{f}(X,G))$ by
    \begin{align*}
        D_1 =&\ \diff \kh{\tilde{X}\mapsto \projection_{\tangent_{\tilde{X}}\hanifold}(\nabla f(X)(I-G))}(X)[\eta] - \projection_{\tangent_X\hanifold}(\nabla f(X)\zeta) 
        \\
        &+ \diff a_1(X)[\eta]\nabla f(X)G + a_1(X)\nabla f(X)\zeta,
        \\
        D_2 =&\ 2\symac\kh{-\zeta \nabla f(X)^\top a_2(X)- G\nabla f(X)^\top\diff a_2(X)[\eta]}.
    \end{align*}
    Additionally, note that $\nabla^2 \bar{f}(X,G)[\eta,\zeta]=\kh{\nabla^2 f(X)[\eta], 0}$. Through the lens of~\eqref{eq:RieHess_2terms}, it suffices to set $(E,Z)=(D_1+\nabla^2 f(X)[\eta], D_2)$ in \eqref{eq:projection_eq} to get the representation of $\nabla^2_{\manifold_h} \bar{f}(X,G)[\eta,\zeta]$. In detail,
    \begin{equation}\label{eq:barK_1}
        \bar{K} = \projection_{\tangent_H\hanifold^r}\kh{\kh{D_1+\nabla^2 f(X)[\eta]}V} = \zkh{\projection_{\tangent_X\hanifold}\kh{\kh{D_1+\nabla^2 f(X)[\eta]}VV^\top}} V,
    \end{equation}
    which results from $V\in\stiefel(n,r)$ and \eqref{eq:projection_transform}. To simplify the computation, consider the curves $H(t)\in\hanifold^r$ and $V(t)\in\stiefel(n,r)$ satisfying $H(0)=H,\,H^\prime(0)=K$ and $V(0)=V,\,V^\prime(0)=V_p$. Moreover, we turn to the differential of the curve $Y(t):=\projection_{\tangent_{\tilde{X}(t)}\hanifold}(\nabla f(X)V(t)V(t)^\top)$ at $t=0$, where $\tilde{X}(t):=H(t)V^\top(t)$. In fact, by \eqref{eq:projection_VVT}, we have 
    \[
    Y(t)=\projection_{\tangent_{\tilde{X}(t)}\hanifold}(\nabla f(X)V(t)V(t)^\top) = \projection_{\tangent_{\tilde{X}(t)}\hanifold}(\nabla f(X))V(t)V(t)^\top,
    \]
    and thus the differential can be computed from the following two ways,
    \begin{align}
        &Y^\prime(0) = \diff \kh{\tilde{X}\mapsto \projection_{\tangent_{\tilde{X}}\hanifold}(\nabla f(X)VV^\top)}(X)[\eta] - \projection_{\tangent_{{X}}\hanifold}(\nabla f(X)\zeta)        \label{eq:Yt_1}
        \\
        &Y^\prime(0) = \diff \kh{\tilde{X}\mapsto \projection_{\tangent_{\tilde{X}}\hanifold}(\nabla f(X))}(X)[\eta]VV^\top - \projection_{\tangent_{{X}}\hanifold}(\nabla f(X))\zeta     \label{eq:Yt_2}
    \end{align}
    Comparing the above two expressions, applying Proposition~\ref{pro:decompose_geo_H}, and considering the Riemannian Hessian of $X\mapsto f(X)$ on $\hanifold$, we find
    \begin{align}
        &\,\projection_{\tangent_X\hanifold}\kh{\zkh{\diff \kh{\tilde{X}\mapsto \projection_{\tangent_{\tilde{X}}\hanifold}(\nabla f(X)VV^\top)}(X)[\eta]+\nabla^2 f(X)[\eta]}VV^\top}  \nonumber
        \\
        =&\ \projection_{\tangent_X\hanifold}\kh{{\diff \kh{\tilde{X}\mapsto \projection_{\tangent_{\tilde{X}}\hanifold}(\nabla f(X))}(X)[\eta]+\nabla^2 f(X)[\eta]}}VV^\top    \nonumber
        \\
        =&\ \nabla^2_\hanifold f(X)[\eta]VV^\top.   \label{eq:hessianHf}
    \end{align}
    Substituting the expression of $D_1$, the result \eqref{eq:hessianHf}, and $\zeta=-V_pV^\top-VV^\top_p$ into \eqref{eq:barK_1}, we obtain
    \[
    \bar{K} = \nabla^2_\hanifold f(X)[\eta]V  + \projection_{\tangent_H{\hanifold^r}}\kh{\nabla f(X)V_p-a_1(X)\nabla f(X)V_p},
    \]
    and expand $a_1(X)$ to achieve the expression of $\bar{K}$ in \eqref{eq:hessian_expression}. 
    
    By \eqref{eq:projection_eq}, we next address the  term
    \begin{equation}\label{eq:compute_Vp}
        \bar{V}_p =\ G\kh{(D_1+\nabla^2 f(X)[\eta])^\top H - 2\omega D_2V}M^{-1}_{H,\omega}.
    \end{equation}
    According to expressions~\eqref{eq:Yt_1} and \eqref{eq:Yt_2}, it follows that
    \begin{align}
        \diff \kh{\tilde{X}\mapsto \projection_{\tangent_{\tilde{X}}\hanifold}(\nabla f(X)VV^\top}(X)[\eta]G =& \kh{\projection_{\tangent_{{X}}\hanifold}(\nabla f(X)\zeta)-\projection_{\tangent_{{X}}\hanifold}(\nabla f(X))\zeta}G   \nonumber
        \\
        =& \kh{-\nabla f(X)VV^\top_p
        +\projection_{\tangent_{{X}}\hanifold}(\nabla f(X))VV^\top_p}G      \nonumber
        \\
        =& -\projection_{\normal_H\hanifold^r}(\nabla f(X)V)V^\top_p,   \label{eq:P_normal}
    \end{align}
    where the second and last equalities hold from $\projection_{\tangent_X\hanifold}(\nabla f(X)\zeta)G = -\nabla f(X)VV_p^\top$ and Proposition~\ref{pro:decompose_geo_H}, respectively. Substituting the expressions of $D_1$, $D_2$, and \eqref{eq:P_normal} into~\eqref{eq:compute_Vp} yields
    \[
    \begin{aligned}
        \bar{V}_p{=}&-GV_p[\projection_{\normal_H\hanifold^r}(\nabla f(X)V)]^\top HM^{-1}_{H,\omega}  + G\nabla f(X)^\top \kh{\diff a_1(X)[\eta]H + 2\omega Da_2(X)[\eta]V  }M^{-1}_{H,\omega} 
        \\
        &+ G(\nabla^2 f(X)[\eta])^\top HM^{-1}_{H,\omega}+ G\zeta\nabla f(X)^\top\kh{a_1(X)H-H+2\omega a_2(X)V}M^{-1}_{H,\omega}
        \\
        {=}&-GV_p[\projection_{\normal_H\hanifold^r}(\nabla f(X)V)]^\top HM^{-1}_{H,\omega}+G\nabla f(X)^\top (I-HM^{-1}_{H,\omega}H^\top)KM^{-1}_{H,\omega} 
        \\
        &+ G(\nabla^2 f(X)[\eta])^\top HM^{-1}_{H,\omega},
    \end{aligned}
    \]
    where we employ \eqref{eq:4a}, \eqref{eq:4a_express} and \eqref{eq:preparation} to get the last equality.
\end{proof}

\section{Proof of Proposition~\ref{pro:2rd_retrac}}\label{app:2rd_retrac}
\begin{proof}
    The first step is to verify the map $\retrac$ is well-defined. To see this, we consider another representation $\kh{H_0,V_0}$ for $(X,G)$ and denote the induced quantities with a zero in the subscript. The proof of Proposition~\ref{pro:1st_retrac} shows $\kh{H_0,V_0}=(HQ,VQ)$ and $(K_0,V_{p,0})=(KQ,V_pQ)$ for some $Q\in\orth(r)$. Similarly, it holds that $L_0=LQ$ and $W_0=WQ$, due to the constructions of $L$ and $W$. Consequently, the value of $\retrac_{(X,G)}(\eta,\zeta)$ is independent of the choice of representations since $[\projection_{\hanifold^r}\kh{(X+\eta)W}]W^\top=[\projection_{\hanifold^r}\kh{(X+\eta)W_0}]W_0^\top$ {and} $WW^\top = W_0W_0^\top$. 
    
    Now, it remains to show that $\retrac$ is second-order, for which we outline the following curves in preparation,
    \[
    \begin{aligned}
        &L(t):=V+tV_p - t^2V_pK^\top HM^{-1}_{H,\omega},\ Z(t):=(L(t)^\top L(t))^{-1/2},\ W(t):=L(t)Z(t),
        \\
        &\beta(t):=(X+t\eta)W(t),\ X(t):= \zkh{\projection_{\hanifold^r}\kh{\beta(t)}}W(t)^\top,\ G(t):=I-W(t)W(t)^\top,
        \\
        &\gamma(t) := \kh{X(t),G(t)} =\retrac_{(X,G)}(t\eta,t\zeta) .
    \end{aligned}
    \]
    Take the derivative on both sides of the equation $Z(t)^2=(L(t)^\top L(t))^{-1}$,
    \begin{equation}\label{eq:diff_Z2}
        Z^\prime(t)Z(t) + Z(t)Z^\prime(t) = - Z(t)^2\kh{L^\prime(t)^\top L(t)+L(t)^\top L^\prime(t)}Z(t)^2. 
    \end{equation}
    Setting $t=0$ and substituting $Z(0)=I,\,L(0)^\top L^\prime(0)=0$ lead to $Z^\prime(0)=0$. Identifying the differential of $\projection_{\hanifold^r}$ at $H$ as $\projection_{\tangent_H\hanifold^r}$ (see \cite[Lemma 4]{absil2012projectionlike}) gives
    \[
        X^\prime(0)=\projection_{\tangent_H\hanifold^r}(\beta^\prime(0))W(0)^\top + \projection_{\hanifold^r}(H)W^\prime(0)^\top = KV^\top + H V_p^\top.
    \]
    By evaluating the equality $G^\prime(t)=-W^\prime(t)W(t)^\top-W(t)W^\prime(t)^\top$ at $t=0$, we have
    \[
        \begin{cases}
            \gamma(0) = (X(0),G(0)) = (X,G),
            \\
            \gamma^\prime(0) = \kh{X^\prime(0), G^\prime(0)} = \kh{KV^\top+HV^\top_p,-V_pV^\top-VV^\top_p}=(\eta,\zeta),
        \end{cases}
    \]
    which verifies $\retrac$ is a retraction. To proceed, define the curve $\alpha(t):=\projection_{\hanifold^r}(\beta(t))$ to be the projection of $\beta(t)$ onto $\hanifold^r$, and the operator $\projection(t)$ to be the projection onto $\tangent_{\alpha(t)}\hanifold^r$, which turns out to be a smooth and linear operator. Notice that $\beta(t)-\alpha(t)$ is orthogonal to $\tangent_{\alpha(t)}\hanifold^r$, i.e., $\projection(t)(\beta(t)-\alpha(t))=0$. Differentiate this equation twice,
    \[
        \ddot{\projection}(t)(\beta(t)-\alpha(t))+ 2\projection^{\prime}(t)(\beta^\prime(t)-\alpha^\prime(t)) + \projection(t)(\ddot{\beta}(t)-\ddot{\alpha}(t))=0.
    \]
    The observation $\alpha(0)=\beta(0)=H$ and $\alpha^\prime(0)=\beta^\prime(0)=K$ gives
    \begin{equation}\label{eq:ddalpha}
        \projection_{\tangent_H\hanifold^r}(\ddot{\alpha}(0)) =\projection_{\tangent_H\hanifold^r}(\ddot{\beta}(0)).
    \end{equation}
    Subsequently, we differentiate \eqref{eq:diff_Z2} at $t=0$ to obtain $\ddot{Z}(0)=-V_p^\top V_p$ and
    \[
        \ddot{W}(0)=\ddot{L}(0)Z(0)+2L^\prime(0)Z^\prime(0)+L(0)\ddot{Z}(0)    = -2V_pK^\top HM^{-1}_{H,\omega} - VV_p^\top V_p.
    \]
    Then the focus moves to the second-order extrinsic derivatives of $X(t)$ and $G(t)$,\!\!\!\!
    \[
        \begin{aligned}
            \ddot{X}(t)=&\ \ddot{\alpha}(t)W(t)^\top + 2\alpha^\prime(t)W^\prime(t)^\top + \alpha(t)\ddot{W}(t)^\top=\ddot{\alpha}(0)V^\top + 2KV_p^\top + H\ddot{W}(0)^\top,
            \\
            \ddot{G}(t)=&-\ddot{W}(t)W(t)^\top -2W^\prime(t)W^\prime(t)^\top - W(t)\ddot{W}(t)^\top
            \\
            =&\ 2V_pK^\top HM^{-1}_{H,\omega}V^\top + 2VM^{-1}_{H,\omega}H^\top KV^\top_p + 2VV_p^\top V_pV^\top - 2V_pV^\top_p.
        \end{aligned}
    \]
    To see $\projection_{(X,G)}(\ddot{X}(0), \ddot{G}(0))$, we invoke \eqref{eq:projection_eq} with $(E,Z)=(\ddot{X}(0), \ddot{G}(0))$,
    \[
    \begin{aligned}
        \bar{K} =\projection_{\tangent_H\hanifold^r}\kh{\ddot{X}(0)V} = \projection_{\tangent_H\hanifold^r}\kh{\ddot{\beta}(0)+H\ddot{W}(0)^\top V},
    \end{aligned}
    \]
    where the equality comes from \eqref{eq:ddalpha}. Employing the equation $\ddot{\beta}(t)=2\eta W^\prime(t) + X\ddot{W}(t)$ and recalling the expression of $\ddot{W}(0)$ produce
    \[
        \ddot{\beta}(0)+H\ddot{W}(0)^\top V = 2\eta W^\prime(0)+X\ddot{W}(0) + H\ddot{W}(0)^\top V = 0, 
    \]
    which means $\bar{K}=0$. Regarding $\bar{V}_p$, it follows from the rule \eqref{eq:projection_eq} that
    \[
        \begin{aligned}
             \bar{V}_p =&\ G\kh{\ddot{X}(0)^\top H-2\omega \ddot{G}(0)V} M^{-1}_{H,\omega}
            \\
            =&\ G\kh{2V_pK^\top H + \ddot{W}(0)H^\top H -4\omega V_pK^\top HM^{-1}_{H,\omega} } M^{-1}_{H,\omega}
            \\
            =&\ 2V_pK^\top H\kh{I -M^{-1}_{H,\omega}H^\top H -2\omega M^{-1}_{H,\omega} } M^{-1}_{H,\omega}
            \\
            =&\ 0.
        \end{aligned}
    \]
    Consequently, $\gamma^{\prime \prime}(0)\!=\!\projection_{(X,G)}(\ddot{X}(0), \ddot{G}(0))\!=\!0$ implies that $\retrac$ is a {second-order retraction}.
\end{proof}

\section{Implementation details of deep learning tasks}\label{app:network}
In this section, we introduce the implementation details and hyper-parameter specifications for the numerical experiments in section~\ref{sec:LRNN}.

{\emph{{Classification task on MNIST.}}} The MNIST dataset consists of grayscale images of handwritten digits from zero to nine, with $60,000$ images for training and $10,000$ images for testing. To classify input images, a multi-layer perception (MLP) is chosen as the baseline network architecture, with a ReLU appended to each linear layer. Specifically, the vanilla network has three layers, where the hidden layers contain $256$ and $64$ neurons, respectively, and the output layer contains $10$ neurons. For the models incorporating low-rank compression, we outline the rank configuration in Table~\ref{tab:network_arch}. Across all the implementations, a batch size of $256$, a learning rate of $0.01$, and a momentum of $0.9$ are adopted.

\begin{table*}[htbp]
    \setlength{\tabcolsep}{12pt}
    \centering
    {
    \caption{Network architectures for classification tasks and the rank configuration for the models incorporating low-rank compression, where ``-" means the original layer is retained without compression.}
    \label{tab:network_arch}
    \begin{tabular}{lcc}
        \toprule
        {Layer} & {Shape} & {Rank} \\
        \midrule
        \multicolumn{3}{c}{MLP} \\
        \midrule
        {linear1-ReLU} & $784\times 256$ &   16  \\
        {linear2-ReLU} & $256\times 64$ &  8 \\
        linear3-ReLU  & $64\times 10$ & -  \\
        \midrule
        \multicolumn{3}{c}{CNN} \\
        \midrule
        conv1-ReLU & $\phantom{00}3\times 128\times 3\times 3$ & - \\ 
        conv2-ReLU & $128\times 128\times 3\times 3$ & - \\ 
        conv3-ReLU & $128\times 128\times 3\times 3$ & $32$ \\ 
        MAX-pooling layer & $3\times 3$ & - \\
        conv4-ReLU & $128\times 256\times 3\times 3$ & $32$ \\ 
        conv5-ReLU & $256\times 256\times 3\times 3$ & $16$ \\ 
        conv6-ReLU & $256\times 256\times 3\times 3$ & $16$ \\ 
        MAX-pooling layer & $3\times 3$ & - \\ 
        conv7-ReLU & $256\times 320\times 3\times 3$ & $16$ \\ 
        conv8-ReLU & $320\times 320\times 3\times 3$ & $16$ \\ 
        conv9-ReLU & $320\times 10\times 1\times 1$ & - \\ 
        Average pooling layer& $5\times 5$ & - \\ 
        
        \bottomrule	
    \end{tabular}
    }
\end{table*}

\emph{{Classification task on CIFAR-10.}} The CIFAR-10 collection features a total of $60,000$ color images sized at $32\times32$ pixels ($50,000$ for training and $10,000$ for testing), spanning ten different categories. We use a convolutional neural network (CNN) as the baseline, with the architecture detailed in Table~\ref{tab:network_arch}. Note that a convolutional layer with $c_{in}$ input channels, $c_{out}$ output channels, and $d\times d$ spatial resolution has $c_{in}\times c_{out}\times d\times d$ parameters, and thus it can be viewed as a linear layer with the weight matrix sized at $c_{out}\times c_{in}d^2$ applied to appropriately reshaped input volumes. In this manner, a rank-$r$ approximation decomposes the weight into two smaller matrices of sizes $c_{out}\times r$ and $r\times c_{in}d^2$, respectively, and they work as two sequential convolutional layers: the first of the shape $r\times c_{in}\times d\times d$, followed by the second of the shape $c_{out}\times r\times 1\times 1$.

We train all the models for $200$ epochs with a batch size of $128$, utilizing a linear decay schedule to adjust the learning rate. Based on the top-1 test accuracy, we choose the optimal hyper-parameters via grid search with the following grid specifications: initial learning rates in $\{0.001+k\times 0.0005:\,k=0,1,\ldots,20\}$, final learning rates in $\{0.001,0.0001,0.00001\}$, weight decay in $\{0.01,0.001,0.0001\}$, and momentum in $\{0.90,0.95,0.98,0.99\}$.

\printbibliography

\end{document}